\documentclass[12pt,a4paper,twoside,reqno]{amsart}
\usepackage[a4paper,top=3cm,bottom=3cm,inner=2.5cm,outer=2.5cm]{geometry}

\usepackage[utf8]{inputenc}
\usepackage{amsmath,amssymb,amsthm,amsfonts}
\usepackage[british]{babel}
\usepackage{xcolor}
\usepackage{xparse}
\usepackage{multirow}
\usepackage{enumitem}
\usepackage{pst-node}
\usepackage{xspace}
\usepackage{placeins}
\usepackage{xspace}
\usepackage{fix-cm}
\usepackage[nomessages]{fp}
\usepackage{array}
\usepackage[colorlinks,linkcolor=blue,citecolor=darkgray,urlcolor=darkgray,final,hyperindex,linktoc=page,hyperfootnotes=true]{hyperref}
\newcolumntype{F}{>{$}c<{\hspace{-0.9ex}$}}
\newcolumntype{:}{>{$}m{0.8ex}<{$}}
\newcolumntype{R}{>{$}r<{$}}
\newcolumntype{C}{>{$}c<{$}}
\newcolumntype{L}{>{$}l<{$}}
\newcolumntype{N}{@{}>{$}l<{$}}
\newlength\horspace
\newcommand{\h}[1][1.0]{\hspace{#1\horspace}}
\newlength\verspace
\newlength\negverspace
\usepackage{tikz-cd}
\tikzset{iso/.style={draw=none,every to/.append style={edge node={node [sloped, allow upside down, auto=false]{$\cong$}}}}}
\tikzset{simeq/.style={draw=none,every to/.append style={edge node={node [sloped, allow upside down, auto=false]{$\simeq$}}}}}
\tikzset{simeqS/.style={draw=none,every to/.append style={edge node={node [sloped, allow upside down, auto=false]{$\raisebox{0.8em}{$\simeq$}$}}}}}
\tikzset{simeqSRight/.style={draw=none,every to/.append style={edge node={node [sloped, allow upside down, auto=false]{$\raisebox{-1em}{\rotatebox{180}{$\simeq$}}$}}}}}
\tikzset{aiso/.style={simeqS,preaction={draw,->}}}
\tikzset{aisos/.style={simeqSs,preaction={draw,->}}}
\tikzset{Right/.style={double distance=1.7pt,>={Implies},->}}
\tikzset{twoiso/.style={simeqSRight,preaction={draw,Right}}}
\tikzset{dotdot/.style={dash pattern=on 0.25ex off 0.2ex, dash phase=0ex}}
\tikzset{simeqSs/.style={draw=none,every to/.append style={edge node={node [sloped, allow upside down, auto=false]{$\raisebox{-0.8em}{\rotatebox{180}{$\simeq$}}$}}}}}
\tikzset{RightA/.style={double distance=3.5pt,>={Implies},->},%
	triple/.style={-,preaction={draw,RightA}},%
	quadruple/.style={preaction={draw,RightA,shorten >=0pt},shorten >=1pt,-,double,double distance=0.8pt}}
\tikzset{simS/.style={draw=none,every to/.append style={edge node={node [sloped, allow upside down, auto=false]{$\raisebox{0.8em}{$\sim$}$}}}}}
\tikzset{aeq/.style={simS,preaction={draw,->}}}

\DeclareMathOperator{\bbicolim}{bicolim}

\newcommand{\bicolim}[2]{{\bbicolim}^{#1}\h{#2}}

\newcommand{\sigmabicolim}[1]{\sigma\mbox{-}\h[1.5]\bicolim{}{#1}}

\newtheorem{teor}{Theorem}[section]

\newtheorem{prop}[teor]{Proposition}
\newtheorem{costr}[teor]{Construction}
\newtheorem{teorintro}{Theorem}
\newtheorem{defneintro}[teorintro]{Definition}

\theoremstyle{definition}
\newtheorem{defne}[teor] {Definition}
\newtheorem{oss}[teor]{Remark}
\newtheorem*{notazione}{Notation}
\newtheorem{esempio}[teor]{Example}

\newenvironment{cd}{\[\begin{tikzcd}[row sep=7ex,column sep=7ex,ampersand replacement=\&]}{\end{tikzcd}\]\ignorespacesafterend}
\newenvironment{cds}[2]{\[\begin{tikzcd}[row sep=#1ex, column sep=#2ex,ampersand replacement=\&]}{\end{tikzcd}\]\ignorespacesafterend}
\newenvironment{cdN}{\begin{tikzcd}[row sep=7ex,column sep=7ex,ampersand replacement=\&]}{\end{tikzcd}\ignorespacesafterend}
\newenvironment{cdsN}[2]{\begin{tikzcd}[row sep=#1ex, column sep=#2ex,ampersand replacement=\&]}{\end{tikzcd}\ignorespacesafterend}
\newenvironment{eqD}[1]{\begin{equation}\label{#1}}{\end{equation}\ignorespacesafterend}
\newenvironment{eqD*}{\begin{equation*}}{\end{equation*}\ignorespacesafterend}

\def\:{\colon}

\def\phi{\varphi}

\def\dfn#1{{\bfseries\itshape #1}}
\def\predfn#1{{\itshape #1}}
\newcommand{\st}{^{\ast}}

\def\c{\circ}

\DeclareFontFamily{OT1}{pzc}{}
\DeclareFontShape{OT1}{pzc}{m}{it}{<->s*[1.21]pzcmi7t}{}
\DeclareMathAlphabet{\mathpzc}{OT1}{pzc}{m}{it}

\DeclareFontFamily{U}{dutchcal}{\skewchar\font=45}
\DeclareFontShape{U}{dutchcal}{m}{n}{<->s*[1.05] dutchcal-r}{}
\DeclareMathAlphabet{\mathlcal}{U}{dutchcal}{m}{n}

\newcommand{\catfont}[1]{\ensuremath{\mathpzc{#1}}\xspace}
\newcommand{\Cal}[1]{\ensuremath{\mathcal{#1}}\xspace}

\newcommand{\A}{\catfont{A}}

\newcommand{\C}{\catfont{C}}
\newcommand{\D}{\catfont{D}}

\newcommand{\K}{\catfont{K}}

\newcommand{\Cat}{\catfont{Cat}}
\newcommand{\Bicat}{\catfont{Bicat}}

\newcommand{\Tricat}{\catfont{Tricat}}

\makeatletter

\newcommand{\Sh}[2][\@nil]{%
	\def\tmp{#1}%
	\ifx\tmp\@nnil{\ensuremath{\catfont{Sh}\hspace{-0.15ex}\left({#2}\right)}}%
	\else{\ensuremath{\catfont{Sh}\hspace{-0.15ex}\left({#2},{#1}\right)}}\fi}

\newcommand{\St}[2][\@nil]{%
	\def\tmp{#1}%
	\ifx\tmp\@nnil{\ensuremath{\catfont{St}\hspace{-0.15ex}\left({#2}\right)}}%
	\else{\ensuremath{\catfont{St}\hspace{-0.15ex}\left({#2},{#1}\right)}}\fi}
\makeatother

\newcommand{\Grothdiag}[1]{\Intdiag{#1}}
\newcommand{\Intdiag}[1]{\ensuremath{\scaleu{\int} \hspace{-0.15ex} #1}}

\newcommand{\HomC}[3]{{#1}\left({#2},\h[1]{#3}\right)}
\newcommand{\x}[1][]{\h[-1]\times_{#1}\h[-1]}
\newcommand{\xp}[2]{\h[-1]{\times}^{#2}_{#1}\h[-1]}
\newcommand{\opn}[1]{\operatorname{#1}}
\newcommand{\id}[1]{\operatorname{id}_{#1}}

\newcommand{\Psm}[2]{\ensuremath{\opn{Ps}\left[#1,#2\right]}\xspace}

\newcommand{\msigma}[2]{\ensuremath{\left[#1,#2\right]_{\sigma}}\xspace}

\newcommand{\op}{\ensuremath{^{\operatorname{op}}}}

\newcommand{\restr}[2]{{\left.\kern-\nulldelimiterspace {#1}\vphantom{\big|} \right|_{#2}}}

\newcommand{\scaleu}[2][1.2]{{\scalebox{#1}{$#2$}}}

\makeatletter
\newcommand{\ar}[2][]{\xrightarrow[#1]{#2}}
\def\xlongrightarrowfill@{\arrowfill@\relbar\relbar\longrightarrow}
\newcommand{\arr}[2][]{%
	\ext@arrow 0099\xlongrightarrowfill@{#1}{#2}}
\newcommand{\aarr}[2][]{%
	\ext@arrow 0099\xlongrightarrowfill@{#1}{#2}} 

\newcommand{\aR}[2][]{%
	\ext@arrow 0055{\Rightarrowfill@}{#1}{#2}}
\def\xLongrightarrowfill@{\arrowfill@\Relbar\Relbar\Longrightarrow}
\newcommand{\aRR}[2][]{%
	\ext@arrow 0099\xLongrightarrowfill@{#1}{#2}}
\def\aitofill@{\arrowfill@{\lhook\joinrel\relbar}\relbar\rightarrow}
\newcommand{\aito}[2][]{%
	\ext@arrow 3095\aitofill@{#1}{#2}}
\def\Longaitofill@{\arrowfill@{\lhook\joinrel\relbar\joinrel\relbar}\relbar\rightarrow}
\newcommand{\aitoo}[2][]{%
	\ext@arrow 0099\Longaitofill@{#1}{#2}}

\def\xlongleftarrowfill@{\arrowfill@\longleftarrow\relbar\relbar}
\newcommand{\all}[2][]{%
	\ext@arrow 0099\xlongleftarrowfill@{#1}{#2}}
\newcommand{\aL}[2][]{%
	\ext@arrow 0055{\Leftarrowfill@}{#1}{#2}}
\def\xLongleftarrowfill@{\arrowfill@\Longleftarrow\Relbar\Relbar}
\newcommand{\aLL}[2][]{%
	\ext@arrow 0099\xLongleftarrowfill@{#1}{#2}}

\def\xmapstofill@{\arrowfill@{\mapstochar\relbar}\relbar\rightarrow}
\newcommand{\am}[2][]{%
	\ext@arrow 0395\xmapstofill@{#1}{#2}}
\def\xlongmapstofill@{\arrowfill@\relbar\relbar\longmapsto}
\newcommand{\amm}[2][]{%
	\ext@arrow 0399\xlongmapstofill@{#1}{#2}}

\newcommand{\eqq}{\DOTSB\protect\Relbar\protect\joinrel\Relbar}
\def\xeqqfill@{\arrowfill@\Relbar\Relbar\eqq}
\newcommand{\aeqq}[2][]{%
	\ext@arrow 0099\xeqqfill@{#1}{#2}}

\def\xRrightarrowfill@{\arrowfill@\equiv\equiv\Rrightarrow}
\newcommand{\aM}[2][]{\ext@arrow 0359\xRrightarrowfill@{#1}{#2}}
\newcommand{\Llongrightarrow}{%
	\DOTSB\protect\equiv\protect\joinrel\Rrightarrow}
\def\xLlongrightarrowfill@{\arrowfill@\equiv\equiv\Llongrightarrow}
\newcommand{\aMM}[2][]{%
	\ext@arrow 0099\xLlongrightarrowfill@{#1}{#2}}
\makeatother

\newcommand{\iso}{\cong}

\newcommand{\aequi}{\ensuremath{\stackrel{\raisebox{-1ex}{\kern-.3ex$\scriptstyle\sim$}}{\rightarrow}}}
\newcommand{\aequii}{\ensuremath{\stackrel{\raisebox{-1ex}{\kern-.3ex$\scriptstyle\sim$}}{\longrightarrow}}}

\newcommand{\PB}[1]{\arrow[#1,phantom,"\scalebox{1.6}{\color{black}$\lrcorner$}",very near start]}
\newcommand{\Ar}[4][]{\arrow[#2,"{#3}"{#1},""{name=#4, anchor=center}]}
\newcommand{\Ars}[4][]{\arrow[#2,"{#3}"'{#1},""{name=#4, anchor=center}]}
\newcommand{\Arb}[6][]{\arrow[#2,"{#3}"{#1},from=#4,to=#5,shorten <= #6 em, shorten >= #6 em]}
\newcommand{\Arbs}[6][]{\arrow[#2,"{#3}"'{#1},from=#4,to=#5,shorten <= #6 em, shorten >= #6 em]}

\NewDocumentEnvironment{cdl}{s O{7} O{7} b}{%
	\IfBooleanF{#1}{\begin{equation*}}\begin{tikzcd}[row sep=#2ex,column sep=#3ex,ampersand replacement=\&]
			#4
		\end{tikzcd}\IfBooleanF{#1}{\end{equation*}}\ignorespacesafterend}{}
\NewDocumentCommand{\csq}{s O{n} O{7} O{7} O{} O{2.7} O{2.2} O{0.5} O{n}}{%
	\def\foocsq##1##2##3##4##5##6##7##8{%
		\IfBooleanTF{#1}{\begin{cdl}*}{\begin{cdl}}[#3][#4]
				{##1}\ifx#2p{\PB{rd}}\fi\arrow[r,"{##5}"]\ifx#9l{\arrow[d,equal,"{##6}"']}\else{\arrow[d,"{##6}"']}\fi\&{##2}\ifx#9r{\arrow[d,equal,"{##7}"]}\else{\arrow[d,"{##7}"]}\fi\ifx#2l{\arrow[ld,Rightarrow,shorten <=#6ex,shorten >=#7ex,"{#5}"{pos=#8}]}\fi\\
				{##3}\ifx#9d{\arrow[r,equal,"{##8}"']}\else{\arrow[r,"{##8}"']}\fi\ifx#2o{\arrow[ur,Rightarrow,shorten <=#6ex,shorten >=#7ex,"{#5}"{pos=#8}]}\fi\&{##4}
		\end{cdl}}%
		\foocsq
}

\newcommand{\commaunivvN}[9][]{%
	\def\foocommaunivvN##1##2##3##4##5{%
		\begin{cdN}
			#2\arrow[rrd,bend left,"{#3}",""'{name=A}]\arrow[rdd,bend right=35,"{#4}"',""{name=B}]\arrow[rd,"{#5}"{#1}]\&[-4ex]\\[-4ex]
			\&#6 \arrow[r,"{##1}"] \arrow[d,"{##2}"'] \& #7 \arrow[ld,Rightarrow,"{##5}",shorten <=2.9ex,shorten >=2.5ex] \arrow[d,"{##3}"] \\
			\&#8 \arrow[r,"{##4}"'] \& #9
	\end{cdN}}%
	\foocommaunivvN%
}

\newcommand{\bicommaunivvN}[9][]{%
	\def\foobicommaunivvN##1##2##3##4##5##6##7{%
		\begin{cdN}
			#2\arrow[rrd,bend left,"{#3}",""'{name=A}]\arrow[rdd,bend right=35,"{#4}"',""{name=B}]\arrow[rd,"{#5}"{#1}]\&[-4ex]\\[-4ex]
			\&#6 \arrow[from=A,Rightarrow,"{##6}"{pos=0.4},shorten <=0.8ex,shorten >=0.8ex,]\arrow[to=B,Rightarrow,"{##7}",shorten <=0.3ex,shorten >=0.3ex,]\arrow[r,"{##1}"] \arrow[d,"{##2}"'] \& #7 \arrow[ld,Rightarrow,"{##5}",shorten <=2.9ex,shorten >=2.5ex] \arrow[d,"{##3}"] \\
			\&#8 \arrow[r,"{##4}"'] \& #9
	\end{cdN}}%
	\foobicommaunivvN%
}

\newcommand{\biisocomma}[9][0.5]{%
	\def\foobiisocomma##1##2##3{%
		\begin{cd}
			#2 \arrow[r,"{#6}"] \arrow[d,"{#7}"'] \& #3 \arrow[ld,twoiso,shorten <=##2ex,shorten >=##3ex,"{##1}"{pos=#1}] \arrow[d,"{#8}"]\\
			#4 \arrow[r,"{#9}"'] \& #5
	\end{cd}}%
	\foobiisocomma%
}

\NewDocumentCommand{\twonats}{s O{2.2} O{8} O{7} O{1.05} O{3.45} O{2}}{%
	\def\footwonats##1##2##3##4##5##6##7##8##9{%
		\def\foofootwonats####1####2####3####4####5{%
			\IfBooleanTF{#1}{\begin{cdl}*}{\begin{cdl}}[#3][#4]
					##1 \Ar{r}{##9}{} \Ars{d,bend right=40}{##5}{A} \Ar{d,bend left=40}{##6}{B} \&
					##2 \Ars{d,bend left}{##8}{Q} \arrow[ld,Rightarrow,shift left=#7,"{####4}"{pos=0.48},shorten <=#5ex, shorten >=#6ex]\&[-2ex]
					##1 \Ar{r}{##9}{} \Ar{d,bend right}{##5}{R} \&
					##2 \Ars{d,bend right=40}{##7}{C} \Ar{d,bend left=40}{##8}{D} \arrow[ld,Rightarrow,shift right=#7,"{####5}"'{pos=0.52},shorten <=#6ex, shorten >=#5ex] \\
					##3 \Ars{r}{####1}{} \&
					##4 \&
					##3 \Ars{r}{####1}{} \&
					##4
					\Arbs{Rightarrow}{\,{####2}}{B}{A}{0.3}
					\Arbs{Rightarrow}{\,{####3}}{D}{C}{0.3}
					\Arb{equal}{}{Q}{R}{#2}
			\end{cdl}}%
			\foofootwonats}\footwonats}

\def\dfn#1{{\bfseries #1}}
\def\predfn#1{{\itshape #1}}

\newcommand{\cY}{\Cal{Y}}

\newcommand{\cC}{\Cal{C}}

\newcommand{\wt}{\widetilde}

\newcommand{\laxsliceslant}[2]{{\raisebox{.1em}{$#1$}\mkern-1mu\left/_{\mkern-4.1mu\operatorname{lax}}\right.\raisebox{-.25em}{$#2$}}}
\newcommand{\laxslice}[2]{\laxsliceslant{#1}{#2}}

\newcommand{\aP}[1]{
	\begin{cdsN}{7}{3}
		\hspace*{-1ex}\arrow[r,quadruple,"{#1}"]\&\hspace*{-1ex}
	\end{cdsN}
}

\newcommand{\Groth}[1]{\int{\hspace{-0.5ex} #1}}

\begin{document}

\title{2-stacks over bisites}
\author{Elena Caviglia}
\address{School of Computing and Mathematical Sciences, University of Leicester, United Kingdom}
\email{ec363@leicester.ac.uk}

\begin{abstract}
We generalize the concept of stack one dimension higher, introducing a notion of 2-stack suitable for a trihomomorphism from a 2-category equipped with a bitopology into the tricategory of bicategories. Moreover, we give a characterization of 2-stacks in terms of explicit conditions, that are easier to use in practice. These explicit conditions are effectiveness conditions for appropriate data of descent on objects, morphisms and 2-cells, generalizing the usual stacky gluing conditions one dimension higher. Furthermore, we prove some new results on bitopologies. The main one is that every object of a subcanonical bisite can be seen as the sigma-bicolimit of each covering bisieve over it. This generalizes one dimension higher a well-know result for subcanonical Grothendieck sites.
\end{abstract}

\maketitle

\setcounter{tocdepth}{1}
\tableofcontents
				
\section*{Introduction}
Stacks generalize one dimension higher the fundamental concept of sheaf. They are pseudofunctors that are able to glue together weakly compatible local data into global data. The global data then recover the local ones in an equally weak way. Stacks were introduced by Giraud in \cite{Giraud}, and they have then had enormous success. The importance of stacks in geometry is due to their ability to take into account automorphisms of objects. While many classification problems do not have a moduli space as solution because of the presence of automorphisms, it is often nonetheless possible to construct a moduli stack. And in this sense stacks can also be seen as generalized spaces. Something similar also happens when one considers quotients of schemes by non-free group actions: in some cases the quotient is not a scheme but nonetheless a stack.

In recent years, the research community has begun generalizing the notion of stack one dimension higher. In \cite{Lurie}, Lurie studied a notion of  $(\infty,1)$-stack, that yields a notion of $(2,1)$-stack for a trihomomorphism that takes values in $(2,1)$-categories, when truncated to dimension 3. And Campbell, in their PhD thesis \cite{Campbellthesis}, introduced a notion of 2-stack that involves a trihomomorphism from a  one-dimensional category  into the tricategory $\Bicat$ of bicategories.

In this paper, we introduce a notion of 2-stack that is suitable for a trihomomorphism from a 2-category endowed with a bitopology into $\Bicat$. The notion of bitopology that we consider is the one introduced by Street in \cite{Streetcharbicatstacks} for bicategories. We achieve our definition of 2-stack by generalizing a characterization of stack due to Street \cite{Streetcharbicatstacks}. Our definition is the following.

\begin{defneintro}\label{2stackintro}
Let $(\K,\tau)$ be a bisite. A trihomomorphism $F:\K\op \to \Bicat$ is a \dfn{2-stack} if for every object $C\in \K$ and every covering bisieve $S\: R\Rightarrow y(C)$ in $\tau(C)$ the pseudofunctor
		$$-\c S\: \Tricat(\K\op, \Bicat)(y(C), F) \longrightarrow \Tricat(\K\op, \Bicat)(R, F)$$
		is a biequivalence. 
\end{defneintro}

The main motivation for us to introduce this notion of 2-stack comes from the theory of quotient stacks. In \cite{Genprinbundquost}, we generalized principal bundles and quotient stacks to the categorical context of sites. We then aimed at a generalization of our theory one dimension higher, to the context of bisites. We have constructed a higher dimensional analogue of principal bundles and of our explicit quotient prestacks. But there was no notion of higher dimensional stack suitable for the produced analogues of quotient prestacks in this two-categorical context. The results of this paper fill this gap. Indeed, thanks to this paper, we are able to prove that, if the bisite satisfies some mild conditions, our analogues of quotient stacks one dimension higher are 2-stacks. This result, together with the whole theory of principal bundles and quotient stacks in the context of bisites, will appear very soon in our forthcoming paper \cite{Prin2bunquo2stacks} as well as in our PhD thesis \cite{mythesis}.

Despite being born to suit quotient 2-stacks, our notion of 2-stack has a great potential of applications in many other contexts. We believe that instances of this notion could be found in algebraic geometry and algebraic topology, as well as in other areas of mathematics. In future work, we plan to explore new geometrical examples of 2-stacks and to further expand the theory of 2-stacks. In particular, we plan to introduce a notion of geometric 2-stack that will be important for applications in geometric contexts.

Since Definition \ref{2stackintro} is quite abstract and hard to apply in practice, we also prove a useful characterization in terms of explicit gluing conditions that can be checked more easily. A key idea behind this characterization is to use the Tricategorical Yoneda Lemma proved by Buhn\'e in their PhD thesis \cite{Buhneithesis} to translate the biequivalences required by Definition \ref{2stackintro} into effectiveness conditions of appropriate data of descent. As a biequivalence is equivalently a pseudofunctor which is surjective on equivalence classes of objects, essentially surjective on morphisms and fully faithful on 2-cells, we obtain effectiveness conditions for data of descent on objects, morphisms and 2-cells. The correct notions of data of descent are encoded in the notions of tritransformation, trimodification and perturbation, that form the codomain of the biequivalences of Definition \ref{2stackintro}. Our characterization of 2-stacks is the following:

\begin{teorintro}\label{char2stacksintro}
	A trihomomorphism $F\: \K\op \to \Bicat$ is a 2-stack if and only if for every $C\in \K$ and every covering bisieve $S\in \tau(C)$ the following conditions are satisfied:
	\begin{itemize}
		\item [(O)] every weak descent datum for $S$ of elements of $F$ is weakly effective;
		\item [(M)] every descent datum for $S$ of morphisms of $F$ is effective;
		\item [(2C)] every matching family for $S$ of 2-cells of $F$ has a unique amalgamation.
	\end{itemize}
\end{teorintro}

The obtained gluing conditions (O), (M) and (2C) generalize one dimension higher the usual gluing conditions satisfied by a stack.

Notice that it would have been hard to give the definition of 2-stack in these explicit terms from the beginning, as we would not have known the correct coherences to ask in the various gluing conditions. Our natural implicit definition is instead able to guide us in finding the right coherence conditions to require.

Before introducing our notion of 2-stack, we first prove some useful results on bisites. All these results will be crucial for us in our forthcoming paper \cite{Prin2bunquo2stacks} to prove that our analogues of quotient stacks one dimension higher are stacks. We believe that these results are useful and interesting per se. In particular, we focus on the case of subcanonical bitopologies, that are the ones for which all representables are stacks. They generalize one dimension higher the concept of subcanonical Grothendieck topology, that requires all representables to be sheaves. Almost all Grothendieck sites of interest in geometry are subcanonical, so it is useful to study in detail subcanonical bitopologies.

An important well-known result in the context of sites is that every object of a subcanonical site is the colimit of each covering sieve over it. This result appears for instance in \cite{Elephant} and it is very helpful when dealing with subcanonical topologies. In our paper \cite{Genprinbundquost}, it is one of the main ingredients of the proof that the produced generalized quotient prestacks are stacks when the site is nice enough. 

In this paper, we prove that every object of a subcanonical bisite can be expressed as some kind of two-dimensional colimit of each covering bisieve over it. This generalizes one dimension higher the well-known analogous result for sites described above. We achieve this result for bisites using sigma-bicolimits, that were introduced by Gray in \cite{Graybook} and then studied by Descotte, Dubuc and Szyld in \cite{DescotteDubucSzyld}. Sigma-bicolimits are a particular kind of conical two-dimensional colimit with coherent 2-cells inside the cocones. Every (weighted) bicolimit can be reduced to a (conical) sigma-bicolimit. We obtain the following theorem for subcanonical bisites.

\begin{teorintro}\label{teorsigmabicolimbisievesintro}
	Let $\tau$ be a subcanonical bitopology and let $S\: R \Rightarrow y(C)$ be a covering bisieve over $C$. Then $S$ is a sigma-bicolim bisieve, that is 
$$C=\sigmabicolim{F}$$
where $F\: \Groth{R}\to \K$ is the 2-functor of  projection to the first component.
\end{teorintro}

We also prove a result that considerably helps dealing with change of base of sigma-bicolim bisieves (Proposition \ref{coconofstar}).

Theorem \ref{teorsigmabicolimbisievesintro}, together with this result, will be crucial in \cite{Prin2bunquo2stacks} to prove that, under mild assumptions, the analogues of quotient stacks one dimension higher are 2-stacks.

\subsection*{Outline of the paper}
In section \ref{sectricat}, we recall tricategories and cells between them, as well as the Tricategorical Yoneda Lemma. These will be important for our definition of 2-stack and even more for the characterization of 2-stacks in terms of explicit gluing conditions.

In section \ref{secbisites}, we consider bisites, that are 2-categories endowed with a bitopology in the sense of Street's \cite{Streetcharbicatstacks}. We prove some useful results on bisites. The main result of this section is that every object of a subcanonical bisite can be expressed as a sigma-bicolimit of every covering bisieve over it (Theorem \ref{teorsigmabicolimbisieves}). We then prove a result that considerably helps dealing with change of base of sigma bicolim-bisieves (Proposition \ref{coconofstar}).

In section \ref{sec2stacks}, we introduce a notion of 2-stack suitable for a trihomomorphism from a bisite into the tricategory $\Bicat$ of bicategories (Definition \ref{2stack}).  Moreover, we present a characterization of 2-stack in terms of explicit gluing conditions (Theorem \ref{char2stacks}). This generalizes one dimension higher the usual stacky gluing conditions, and provides substantial benefits when using 2-stacks or proving that a trihomomorphism is a 2-stack.

\section{Preliminaries on tricategories}\label{sectricat}

In this section we recall tricategories and cells between them. Our main references will be  \cite{JohnsonYau} and \cite{Gurskithesis}. 

\begin{defne}
	A \dfn{tricategory} $T$ is given by the following data:
	\begin{itemize}
		\item a set $\opn{Ob}(T)$ of objects (usually simply denoted as $T$);
		\item given $a,b\in T$ a bicategory $T(a,b)$;
		\item given $a,b,c\in T$ a composition functor 
		$$\otimes\: T(b,c)\times T(a,b) \to T(a,c)$$
		\item given $a\in T$ a functor 
		$$I_a\: 1 \to T(a,a)$$
		(where $1$ is the unit bicategory);
		\item given $a,b,c,d \in T$ an adjoint equivalence $\alpha$ in \\$\Bicat(T(c,d)\x T(b,c)\x T(a,b),T(a,d))$
		\begin{cd}
			{T(c,d)\x T(b,c) \x T(a,b)} \arrow[r,"{\otimes \x 1}"] \arrow[d,"{1 \x \otimes}"'] \& {T(b,d) \x T(a,b)}\arrow[d,"{ \otimes}"] \arrow[ld,"{\alpha}", Rightarrow, shorten <= 6ex, shorten >= 7ex ]\\
			{T(c,d)\x T(a,c)} \arrow[r,"{\otimes}", ]\& {T(a,d);} 
			\end{cd}
		\item given $a,b\in T$ adjoint equivalences $l$ and $r$ in $\Bicat(T(a,b), T(a,b))$
		
		\begin{cdN}
			{T(a,b)} \arrow[rd,"{1}"',""{name=G}] \arrow[r,"{I_b \x 1}"]\& {T(b,b)\x T(a,b)}\arrow[to=G,"{l}", Rightarrow] \arrow[d,"{\otimes}"]\\
			{} \& {T(a,b)}
		\end{cdN}
	\hspace{0.1ex}
\begin{cdN}
		{T(a,b)} \arrow[rd,"{1}"',""{name=G}] \arrow[r,"{1 \x I_a}"]\& {T(a,b)\x T(a,a)}\arrow[to=G,"{r}", Rightarrow] \arrow[d,"{\otimes}"]\\
		{} \& {T(a,b);}
	\end{cdN}
\vspace{2.5mm}
\item given $a,b,c,d,e\in T$, an invertible modification $\pi$ in $\Bicat(T^4(a,b,c,d,e),T(a,e))$
\begin{eqD*}
\begin{cdsN}{3.4}{2.5}
	{} \& {T^4} \arrow[rr,"{\otimes \x 1 \x 1}"] \arrow[ld,"{1 \x 1\x \otimes}"'] \arrow[rd,"{1 \x \otimes \x 1}"]\& {} \& {T^3} \arrow[ld,"{\alpha\x 1}", Rightarrow, shorten <=1.2ex,shorten >=1.2ex ]\arrow[rd,"{\otimes \x 1}"] \& {} \\
	{T^3} \arrow[rd,"{1 \x \otimes}"'] \& {} \& {T^3} \arrow[ll,"{1 \x \alpha}",Rightarrow, shorten <=3.2ex,shorten >=3.2ex] \arrow[rr,"{\otimes \x 1}"] \arrow[ld,"{1\x \otimes}"]\& {} \& {T^2} \arrow[llld,"{\alpha}", Rightarrow, shorten <=6.7ex,shorten >=6.7ex] \arrow[ld,"{\otimes}"] \\
	{} \& {T^2} \arrow[rr,"{\otimes}"'] \& {} \& {T} \& {} 
\end{cdsN}
\aM{\pi}
\begin{cdsN}{3.4}{2.5}
	{} \& {T^4} \arrow[rr,"{\otimes \x 1 \x 1}"] \arrow[ld,"{1 \x 1\x \otimes}"'] \& {} \& {T^3} \arrow[llld,"{}", equal, shorten <=8ex, shorten >= 8 ex]\arrow[ld,"{1 \x \otimes}"]\arrow[rd,"{\otimes \x 1}"] \& {} \\
	{T^3} \arrow[rr,"{\otimes \x 1}"] \arrow[rd,"{1 \x \otimes}"'] \& {} \& {T^2} \arrow[ld,"{\alpha}", Rightarrow, shorten <= 1ex, shorten >= 1ex] \arrow[rd,"{\otimes}"] \& {} \& {T^2} \arrow[ll,"{\alpha}", Rightarrow, shorten <=2.5 ex, shorten >= 2.5 ex] \arrow[ld,"{\otimes}"] \\
	{} \& {T^2} \arrow[rr,"{\otimes}"'] \& {} \& {T;} \& {} 
\end{cdsN}
\end{eqD*}
\item given $a,b,c\in T$, invertible modifications $\mu$, $\lambda$ and $\rho$ 
\begin{eqD*}
\bicommaunivvN{T^2}{1}{1}{1\x I\x 1}{T^3}{T^2}{T^2}{T}{\otimes\x 1}{1\x\otimes}{\otimes}{\otimes}{\alpha}{r\x 1}{1\x l}
\h\aM{\mu}\h
\csq*[l][7][7][1]{T^2}{T^2}{T^2}{T}{1}{1}{\otimes}{\otimes}
\end{eqD*}
\begin{eqD*}
	\begin{cdsN}{5.5}{5.5}
		{} \& {T^3} \arrow[rd,"{\otimes \x 1}"] \arrow[d,"{l \x 1}", Rightarrow,shorten <=1.5ex,shorten >=1.5ex] \& {} \\
		{T^2} \arrow[dd,"{\otimes}"'] \arrow[rr,"{1}"', ""{name=O}]  \arrow[ru,"{I \x 1 \x 1}"] \& \hphantom{.} \& {T^2} \arrow[lldd,"{}", equal, shorten <= 11.3ex, shorten >=11.3 ex] \arrow[dd,"{\otimes}"] \\
		{} \& {} \& {} \\
		{T} \arrow[rr,"{1}"'] \& {} \& {T}
	\end{cdsN}
\h[3]\aM{\lambda}\h[3]
\begin{cdsN}{5.5}{5.5}
	{} \& {T^3} \arrow[lddd,equal, shorten <=14ex, shorten >= 14ex]\arrow[dd,"{1 \x \otimes}"] \arrow[rd,"{\otimes \x 1}"]  \& {} \\
	{T^2} \arrow[dd,"{\otimes}"]  \arrow[ru,"{I \x 1 \x 1}"] \& \hphantom{.} \& {T^2} \arrow[ld,"{\alpha}", Rightarrow, shorten <=3ex, shorten >= 3ex] \arrow[dd,"{\otimes}"'] \\
	{} \& {T^2} \arrow[rd,"{\otimes}"'] \arrow[d,Rightarrow,"{l}",shorten <=1.5ex, shorten >= 1.5ex] \& {} \\
	{T} \arrow[rr,"{1}"']  \arrow[ru,"{1\x I}"']\& \hphantom{.} \& {T} 
\end{cdsN}
\end{eqD*}
\begin{eqD*}
	\begin{cdsN}{5.5}{5.5}
		{} \& {T^3} \arrow[rd,"{1\x\otimes}"] \arrow[d,"{1 \x \hat{r}}", Leftarrow,shorten <=1.5ex,shorten >=1.5ex] \& {} \\
		{T^2} \arrow[dd,"{\otimes}"'] \arrow[rr,"{1}"', ""{name=O}]  \arrow[ru,"{1 \x 1 \x I}"] \& \hphantom{.} \& {T^2} \arrow[lldd,"{}", equal, shorten <= 11.3ex, shorten >=11.3 ex] \arrow[dd,"{\otimes}"] \\
		{} \& {} \& {} \\
		{T} \arrow[rr,"{1}"'] \& {} \& {T}
	\end{cdsN}
	\h[3]\aM{\rho}\h[3]
	\begin{cdsN}{5.5}{5.5}
		{} \& {T^3} \arrow[lddd,equal, shorten <=14ex, shorten >= 14ex]\arrow[dd,"{\otimes\x 1}"] \arrow[rd,"{1\x \otimes}"]  \& {} \\
		{T^2} \arrow[dd,"{\otimes}"]  \arrow[ru,"{1 \x 1 \x I}"] \& \hphantom{.} \& {T^2} \arrow[ld,"{\alpha}", Leftarrow, shorten <=3ex, shorten >= 3ex] \arrow[dd,"{\otimes}"'] \\
		{} \& {T^2} \arrow[rd,"{\otimes}"']\arrow[d,Leftarrow,"{\hat{r}}",shorten <=1.5ex, shorten >= 1.5ex]  \& {} \\
		{T} \arrow[rr,"{1}"']  \arrow[ru,"{1\x I}"']\& \hphantom{.} \& {T} 
	\end{cdsN}
\end{eqD*}
	\end{itemize}
See Definition 3.1.2 of \cite{Gurskithesis} for the axioms that these data are required to satisfy.
\end{defne}

\begin{notazione}
	In the previous definition we used the compact notation $T^n$ in place of the appropriate products of $n$ hom-bicategories of $T$. This notation will be used throughout the paper.
\end{notazione}

\begin{oss}
	Given $a,b\in T$ the objects of $T(a,b)$ are called \dfn{morphisms} of $T$ with source $a$ and target $b$, the morphisms  between them are called \dfn{2-cells} of $T$ and the 2-cells are called \dfn{3-cells} of $T$.
\end{oss}

We will now recall the cells between tricategories.

\begin{defne}\label{trihomomorphism}
	Let $T$ and $T'$ be tricategories. A \dfn{trihomomorphism} $F\: T \to T'$ is given by the following data:
	
	\begin{itemize}
		\item a function $F:\opn{Ob}(T) \to \opn{Ob}(T')$;
		\item given $a,b\in T$ a pseudofunctor 
		$F_{a,b}\: T(a,b) \to T'(F(a),F(b));$
		\item given $a,b,c,d\in T$ an adjoint equivalence $\chi\: \otimes' \c (F\x F) \Rightarrow F \c \otimes $ with left adjoint 
		\begin{cd}
			{T(b,c) \x T(a,b)} \arrow[r,"{F\x F}"] \arrow[d,"{\otimes}"'] \& {T'(F(b),F(c))\x T'(F(a), F(b))}\arrow[d,"{ \otimes'}"] \arrow[ld,"{\chi}", Rightarrow, shorten <= 6ex, shorten >= 7ex ]\\
			{T(a,c)} \arrow[r,"{F}"', ]\& {T'(F(a),F(c));} 
		\end{cd}
		\item given $a\in T$ an adjoint equivalence $\iota\: I'_{F(a)} \Rightarrow F\c I_a$ with left adjoint 
		\begin{cd}
			{1} \arrow[rd,"{I'_{F(a)}}"',""{name=G}] \arrow[r,"{I_a}"]\& {T(a,a)}\arrow[to=G,"{\iota}", Leftarrow] \arrow[d,"{F}"]\\
			{} \& {T'(F(a),F(a));}
		\end{cd}
	\item given $a,b,c,d\in T$  an invertible modification 
	\begin{eqD*}
		\begin{cdsN}{3.4}{2.5}
			{} \& {T^3} \arrow[rr,"{F \x F\x F}"] \arrow[ld,"{1 \x \otimes}"'] \arrow[rd,"{ \otimes \x 1 }"]\& {} \& {T'^3} \arrow[ld,"{\chi\x 1}", Rightarrow, shorten <=1.2ex,shorten >=1.2ex ]\arrow[rd,"{1 \x \otimes'}"] \& {} \\
			{T^2} \arrow[rd,"{1 \x \otimes}"'] \& {} \& {T^2} \arrow[ll,"{ \alpha}",Rightarrow, shorten <=3.2ex,shorten >=3.2ex] \arrow[rr,"{F\x F}"'] \arrow[ld,"{ \otimes}"]\& {} \& {T'^2} \arrow[llld,"{\chi}", Rightarrow, shorten <=6.7ex,shorten >=6.7ex] \arrow[ld,"{\otimes'}"] \\
			{} \& {T} \arrow[rr,"{F}"'] \& {} \& {T'} \& {} 
		\end{cdsN}
		\aM{\omega}
		\begin{cdsN}{3.4}{2.5}
			{} \& {T^3} \arrow[rr,"{F\x F \x F}"] \arrow[ld,"{1\x \otimes}"'] \& {} \& {T'^3} \arrow[llld,"{1 \x \chi}", shorten <=6.5ex, shorten >= 6.5 ex, Rightarrow]\arrow[ld,"{1 \x \otimes'}"]\arrow[rd,"{\otimes' \x 1}"] \& {} \\
			{T^2} \arrow[rr,"{F\x F}"'] \arrow[rd,"{\otimes}"'] \& {} \& {T'^2} \arrow[ld,"{\chi}", Rightarrow, shorten <= 1ex, shorten >= 1ex] \arrow[rd,"{\otimes'}"] \& {} \& {T'^2} \arrow[ll,"{\alpha'}", Rightarrow, shorten <=2.5 ex, shorten >= 2.5 ex] \arrow[ld,"{\otimes'}"] \\
			{} \& {T} \arrow[rr,"{F}"'] \& {} \& {T';} \& {} 
		\end{cdsN}
	\end{eqD*}
\item given $a,b\in T$ invertible modifications
\begin{eqD*}
\begin{cdsN}{5.5}{5.5}
	{} \& {T'^2}  \arrow[rd,"{\otimes '}"] \arrow[rddd,"{\chi}", Rightarrow, shorten <=10ex, shorten >= 10ex]   \& {} \\
	{T'} \arrow[rd,"{\iota \x 1}", Rightarrow, shorten <=3ex, shorten >= 3ex]    \arrow[ru,"{I'\x 1}"] \& \hphantom{.} \& {T'} \\
	{} \& {T^2} \arrow[uu,"{F\x F}"]  \arrow[rd,"{\otimes}"''] \arrow[d,Rightarrow,"{l}",shorten <=1.5ex, shorten >= 1.5ex] \& {} \\
	{T} \arrow[uu,"{F}"] \arrow[rr,"{1}"']  \arrow[ru,"{1\x I}"']\& \hphantom{.} \& {T} \arrow[uu,"{F}"''] 
\end{cdsN}
\h[3]\aM{\gamma}\h[3]
\begin{cdsN}{5.5}{5.5}
	{} \& {T'^2} \arrow[rd,"{\otimes'}"] \arrow[d,"{l}", Rightarrow,shorten <=1.5ex,shorten >=1.5ex] \& {} \\
	{T'}  \arrow[rr,"{1}"', ""{name=O}]  \arrow[ru,"{I' \x 1}"] \& \hphantom{.} \& {T'} \arrow[lldd,"{}", equal, shorten <= 11.3ex, shorten >=11.3 ex]  \\
	{} \& {} \& {} \\
	{T} \arrow[uu,"{F}"] \arrow[rr,"{1}"'] \& {} \& {T} \arrow[uu,"{F}"]
\end{cdsN}
\end{eqD*}
\begin{eqD*}
	\begin{cdsN}{5.5}{5.5}
		{} \& {T^2} \arrow[rd,"{\otimes}"] \arrow[d,"{\hat{r}}", Leftarrow,shorten <=1.5ex,shorten >=1.5ex] \& {} \\
		{T}  \arrow[rr,"{1}"', ""{name=O}]  \arrow[ru,"{1\x I}"] \& \hphantom{.} \& {T} \arrow[lldd,"{}", equal, shorten <= 11.3ex, shorten >=11.3 ex]  \\
		{} \& {} \& {} \\
		{T'} \arrow[uu,"{F}", leftarrow] \arrow[rr,"{1}"'] \& {} \& {T'} \arrow[uu,"{F}", leftarrow]
	\end{cdsN}
	\h[3]\aM{\delta}\h[3]
	\begin{cdsN}{5.5}{5.5}
		{} \& {T^2}  \arrow[rd,"{\otimes}"]   \& {} \\
		{T}   \arrow[ru,"{1 \x I}"] \& \hphantom{.} \& {T} \\
		{} \& {T'^2} \arrow[ru,"{\chi}", Rightarrow, shorten <=3ex, shorten >= 3ex]  \arrow[uu,"{F\x F}", leftarrow]  \arrow[rd,"{\otimes'}"''] \arrow[d,Leftarrow,"{l}",shorten <=1.5ex, shorten >= 1.5ex] \& {} \\
		{T'} \arrow[ruuu,"{1 \x \iota}", Rightarrow, shorten <=10ex, shorten >= 10ex]   \arrow[uu,"{F}", leftarrow] \arrow[rr,"{1}"']  \arrow[ru,"{1\x I'}"']\& \hphantom{.} \& {T'} \arrow[uu,"{F}"'', leftarrow] 
	\end{cdsN}
	\end{eqD*}
	\end{itemize}
See Definition 3.3.1 of \cite{Gurskithesis} for the axioms that these data are required to satisfy.
\end{defne}

\begin{defne}
	Let $F,G\: T \to T'$ be trihomomorphisms. A \dfn{tritransformation} $\theta\: F \Rightarrow G$ is given by the following data:
	\begin{itemize}
		\item given $a\in T$ a morphism $\theta_a\: F(a) \to G(a)$ in $T'$;
		\item given $a,b\in T$ an adjoint equivalence 
		\csq[l][7][7][\theta]{T(a,b)}{T'(F(a),F(b))}{T'(G(a),G(b))}{T'(F(a),G(b));}{F}{G}{T'(1, \theta_b)}{T'(\theta_a,1)}
		\item given $a,b,c\in T$ invertible modifications 
		\begin{eqD*}
			\scalebox{0.6}[0.65]{
		\begin{cdN}
			{} \& {} \& {T(b,c)\x T(a,b)} \arrow[ddl,"{G\x G}"'] \arrow[ddll,"{\otimes}"', bend right=15] \arrow[r,"{F\x F}"] \arrow[d,"{G\x F}"]\& {T'(F(b),F(c))\x T'(F(a),F(b))} \arrow[ld,"{\theta\x 1}", Rightarrow, ,shorten <=9.5ex, shorten >= 11.5ex] \arrow[d,"{T'(1,\theta_c)\x 1}"] \\
			{} \& {} \& {T'(G(b),G(c))\x T'(F(a),F(b))} \arrow[ld,"{1\x \theta}", Rightarrow,shorten <=9.5ex, shorten >= 11ex]\arrow[r,"{T'(\theta_b)\x 1}"] \arrow[d,"{1 \x T'(1,\theta_b)}"]\& {T'(F(b),G(c))\x T'(F(a),F(b))} \arrow[ld,"{\alpha}",shorten <=10.5ex, shorten >= 10.5ex, Rightarrow] \arrow[d,"{\otimes}"] \\
			{T(a,c)} \arrow[r,"{\chi}", Leftarrow,shorten <=1.5ex, shorten >= 1.5ex] \arrow[rrd,"{G}"', bend right=15] \& {T'(G(b),G(c))\x T'(G(a),G(b))} \arrow[r,"{1\x T'(\theta_a,1)}"] \arrow[rd,"{\otimes}"] \& {T'(G(b),G(c))\x T'(F(a),G(b))} \arrow[d,"{\hat{\alpha}}", Rightarrow,shorten <=1.5ex, shorten >= 1.5ex] \arrow[r,"{\otimes}"] \& {T'(F(a), G(c))} \\
			{} \& {} \& {T'(G(a), G(c))} \arrow[ru,"{T'(\theta_a,1)}"'] \& {} 
	\end{cdN}}
		\end{eqD*}
	\begin{eqD*}
		\begin{cdN}
			{\hphantom{.}} \arrow[d,"{\Pi}", triple]\\
			{\hphantom{.}}
		\end{cdN}
	\end{eqD*}
\begin{eqD*}
	\scalebox{0.6}[0.65]{
		\begin{cdN}
			{} \& {} \& {T(b,c)\x T(a,b)}  \arrow[ddll,"{\otimes}"', bend right=15] \arrow[r,"{F\x F}"] \arrow[d,"{F\x F}"]\& {T'(F(b),F(c))\x T'(F(a),F(b))} \arrow[ldd,"{\alpha}", Rightarrow, ,shorten <=15.5ex, shorten >= 15.5ex] \arrow[d,"{T'(1,\theta_c)\x 1}"] \\
			{} \& {} \& {T'(F(b),F(c))\x T'(F(a),F(b))} \arrow[lld,"{\chi}", Rightarrow,shorten <=9.5ex, shorten >= 9.5ex]  \arrow[d,"\otimes"]\& {T'(F(b),G(c))\x T'(F(a),F(b))}  \arrow[d,"{\otimes}"] \\
			{T(a,c)} \arrow[rr,"{F}"] \arrow[rrd,"{G}"', bend right=15] \& {}   \& {T'(F(a),F(c))}  \arrow[d,"{\hat{\alpha}}", Rightarrow,shorten <=1.5ex, shorten >= 1.5ex] \arrow[r,"{T'(1,\theta_c)}"] \& {T'(F(a), G(c))} \\
			{} \& {} \& {T'(G(a), G(c))} \arrow[ru,"{T'(\theta_a,1)}"'] \& {} 
	\end{cdN}}
\end{eqD*}
			and
			
			\begin{eqD*}
				\scalebox{0.7}{
				\begin{cdN}
					{} \& {1}\arrow[ld,"{\iota}", Rightarrow ,shorten <=7.5ex, shorten >= 7.5ex] \arrow[d,"{I_{F(a)}}"] \arrow[ld,"{I_a}"',""{name=B}, bend right= 25] \arrow[rdd,"{\theta_a}",""{name=A}, bend left=35] \& {} \\
					{T(a,a)} \arrow[r,"{F}"] \arrow[d,"{G}"']\& {T'(F(a), F(a))} \arrow[ld,"{\theta}", Rightarrow ,shorten <=5.5ex, shorten >= 5.5ex]  \arrow[from=A,"{\hat{r}}", Rightarrow ,shorten <=2.5ex, shorten >= 2.5ex] \arrow[rd,"{T'(1,\theta_a)}"] \& {} \\
					{T'(G(a), G(a))} \arrow[rr,"{T'(\theta_a,1)}"'] \& {} \& {T'(F(a), G(a))} 
					\end{cdN}
				\h[-40]
				\begin{cdN}
			{\hphantom{.}}  \arrow[r,"{M}", triple]\& {\hphantom{.}}
			\end{cdN}
		\h[-40]
		\begin{cdN}
			{} \& {1} \arrow[ld,"{I_a}"',""{name=B}, bend right= 25] \arrow[rdd,"{\theta_a}",""{name=A}, bend left=35] \arrow[ddl,"{I_{G(a)}}",""{name=C},bend left=35] \& {} \\
			{T(a,a)} \arrow[d,"{G}"]\& {\hphantom{.}} \arrow[l,"{\iota}", Rightarrow ,shorten <=3.5ex, shorten >= 3.5ex]\& {} \\
			{T'(G(a), G(a))} \arrow[from=A,"{\hat{l}}", Rightarrow ,shorten <=10.5ex, shorten >= 10.5ex, shift left= 3ex] \arrow[rr,"{T'(\theta_a,1)}"']\& {} \& {T'(F(a), G(a));} 
			\end{cdN}}
			\end{eqD*}
	\end{itemize}
See Definition 3.3.6 of \cite{Gurskithesis} for the axioms that these data are required to satisfy.
\end{defne}

\begin{defne}
	Let $\theta$ and $\phi$ be tritransformations with source $F$ and target $G$. A \dfn{trimodification} $m\: \theta \aM{} \phi$ is given by the following data:
	\begin{itemize}
		\item given $a\in T$ a 2-cell $m_a\: \theta_a \aR{} \phi_a$;
		\item given $a,b\in T$ an invertible modification 
		\begin{eqD*}
			\scalebox{0.8}{
			\begin{cdN}
				{T(a,b)}  \arrow[r,"{F}"] \arrow[d,"{G}"]\& {T'(F(a), F(b))} \arrow[d,"{T'(1, \theta_b)}"] \arrow[ld,"{\theta}", Rightarrow ,shorten <=4.5ex, shorten >= 4.5ex] \\
				{T'(G(a), G(b))} \arrow[r,"{T'(\theta_a,1)}", ""{name=I}]  \arrow[r,"{T'(\phi_a,1)}"',""{name=L}, bend right=68] \arrow[from=I, to=L,"{(m_a)\st}", Rightarrow ,shorten <=1.5ex, shorten >= 1.5ex]\& {T'(F(a), G(b))}
			\end{cdN}
		\h[-35]
			\begin{cdN}
			{\hphantom{.}}  \arrow[r,"{\wt{m}}"'{inner sep=1ex}, triple]\& {\hphantom{.}}
	\end{cdN}
\h[-45]
\begin{cdN}
	{T(a,b)}  \arrow[r,"{F}"] \arrow[d,"{G}"]\& {T'(F(a), F(b))} \arrow[d,"{T'(1, \theta_b)}",""{name=U}, bend left=68] \arrow[d,"{T'(1, \phi_b)}"', ""{name=V}] \arrow[ld,"{\phi}", Rightarrow ,shorten <=4.5ex, shorten >= 4.5ex] \\
	{T'(G(a), G(b))} \arrow[r,"{T'(\phi_a,1)}", ""{name=I}]  \arrow[from=V, to=U,"{(m_b)\st}", Leftarrow ,shorten <=0.5ex, shorten >= 1.5ex]\& {T'(F(a), G(b)).}
	\end{cdN}}
		\end{eqD*}
	\end{itemize}
See Definition 3.3.8 of \cite{Gurskithesis} for the axioms that these data are required to satisfy.
\end{defne}

\begin{defne}
	A \dfn{perturbation} $p\: m \aP{} n$ between trimodifications with the same source and target is a family of 3-cells $p_a\: m_a \aM{} n_a$ in the target tricategory T' indexed by objects of the source tricategory $T$ such that the following axiom hold:
	\twonats[2]{\theta\otimes F(f)}{G(f)\otimes \theta}{\phi\otimes F(f)}{G(f)\otimes \phi.}{n\otimes 1}{m\otimes 1}{1\otimes n}{1\otimes m}{\theta}{\phi}{\sigma \otimes 1}{1\otimes \sigma}{m}{n}
\end{defne}

We will now specialize some of the previous definitions to the particular case that we will need to consider later on. We will explicitly present  all the axioms involved as they will be useful in section \ref{sec2stacks}.

\begin{esempio}[tritransformation] \label{ourtritrans}
	Let $\K$ be a (strict) $2$-category and let $\Bicat$ be the tricategory of bicategories, pseudofunctors, pseudonatural transformations and modifications. Let then $R, F\:  \K\op \to \Bicat$ be trihomomorphisms.
	
	A tritransformation $\alpha\: R \Rightarrow F$ is given by:
	\begin{itemize}
		\item given $C\in K$ a pseudofunctor $\alpha_C\: R(C) \to F(C)$;
		\item given $C,D$ in $\K$ a pseudonatural transformation 
		\csq[l][7][7][\wt{\alpha}]{\K\op(C, D)}{\Bicat(R(C), R(D))}{\Bicat(F(C), F(D))}{\Bicat(R(C), F(D))}{R}{F}{\alpha_D\c -}{-\c \alpha_C}
		whose component $\alpha_f$ relative to $f:D \to C$
		\csq[l][7][7][\alpha_f]{R(C)}{R(D)}{F(C)}{F(D)}{R(f)}{\alpha_C}{\alpha_D}{F(f)}
		is an equivalence for every $f$;
		\item given $E\ar{g} D \ar{f} C$ in $\K$, invertible modifications
		\begin{eqD*}
			\scalebox{0.9}{
		\begin{cdN}
			{R(C)} \arrow[r,"{R(f)}"] \arrow[d,"{\alpha_C}"'] \&  {R(D)} \arrow[ld,"{\alpha_f}", Rightarrow ,shorten <=3.5ex, shorten >= 3.5ex] \arrow[r,"{R(g)}"]  \arrow[d,"{\alpha_D}"]\& {R(E)} \arrow[ld,"{\alpha_g}", Rightarrow ,shorten <=3.5ex, shorten >= 3.5ex] \arrow[d,"{\alpha_E}"] \\
			{F(C)} \arrow[r,"{F(f)}"]  \arrow[rr,"{F(f\c g)}"',""{name=D}, bend right = 55] \&  {F(D)}  \arrow[to=D,"{\chi_{f,g}}" ,shorten <=1.5ex, shorten >= 1.5ex, Rightarrow] \arrow[r,"{F(g)}"] \& {F(E)}
		\end{cdN}
	\h[-9]
		\begin{cdN}
		{\hphantom{.}}  \arrow[r,"{\beta}", triple]\& {\hphantom{.}}
	\end{cdN}
\h[-9]
\begin{cdN}
	{} \& {R(D)} \arrow[d,"{\chi_{f,g}}", Rightarrow ,shorten <=2.5ex, shorten >= 2.5ex] \arrow[rd,"{R(g)}", bend left= 30] \& {} \\
	{R(C)} \arrow[ru,"{R(f)}", bend left= 30] \arrow[rr,"{R(f\c g)}"',""{name=E}]  \arrow[d,"{\alpha_C}"']\& {\hphantom{.}} \& {R(E)} \arrow[lld,"{\alpha_{f\c g}}", Rightarrow ,shorten <=9.5ex, shorten >= 9.5ex]\arrow[d,"{\alpha_E}"] \\
	{F(C)} \arrow[rr,"{F(f\c g)}"']\& {} \& {F(E)} 
\end{cdN}}
	\end{eqD*}
and 
\begin{eqD*}
	\begin{cdN}
		{R(C)} \arrow[r,"{R(\id{C})}"', ""{name=K}] \arrow[r,"{}"{name=H}, equal, bend left= 60] \arrow[d,"{\alpha_C}"'] \arrow[from=H, to=K ,"{\iota_C}", Rightarrow ,shorten <=1ex]\& {R(C)} \arrow[ld,"{\alpha_{\id{C}}}", Rightarrow,shorten <=3.5ex, shorten >= 3.5ex] \arrow[d,"{\alpha_C}"] \\
		{F(C)} \arrow[r,"{F(\id{C})}"']\& {F(C)} 
	\end{cdN}
\h[-3]
	\begin{cdN}
		{\hphantom{.}}  \arrow[r,"{\gamma}", triple]\& {\hphantom{.}}
	\end{cdN}
\h[-3]
	\begin{cdN}
		{R(C)} \arrow[r,equal]  \arrow[d,"{\alpha_C}"'] \& {R(C)} \arrow[ld,"", equal,shorten <=4.5ex, shorten >= 4.5ex] \arrow[d,"{\alpha_C}"] \\
		{F(C)} \arrow[r,"{F(\id{C})}"'{name=H}, bend right= 60] \arrow[r,equal,""{name=G}] \arrow[from=G, to=H,"{\iota_C}",Rightarrow ,shorten <=1.5ex, shorten >= 1.5ex]\& {F(C)} 
	\end{cdN}
\end{eqD*}
	\end{itemize}
	 These data need to satisfy the following axioms.

	Given $L\ar{g} E \ar{f} D$ in $\K$, the following equality of modifications holds:
		\begin{eqD*}
			\begin{cdsN}{4.5}{4}
				R(C) \arrow[d,"{\alpha_C}"']\arrow[r,"{R(f)}"] \& R(D)\arrow[ld,Rightarrow,"{\alpha_f}", shorten <=1.85ex, shorten >=1.5ex] \arrow[d,"{\alpha_D}"]\arrow[r,"{R(g)}"]\& R(E)\arrow[ld,Rightarrow,"{\alpha_g}", shorten <=1.85ex, shorten >=1.5ex]\arrow[d,"{\alpha_E}"]\arrow[r,"{R(h)}"] \& R(L)\arrow[ld,Rightarrow,"{\alpha_h}", shorten <=1.85ex, shorten >=1.5ex]\arrow[d,"{\alpha_L}"] \\
				F(C)\arrow[rrr,bend right=35,"{F(f\c g\c h)}"'{pos=0.36},""{pos=0.25,name=B}] \arrow[r,"{F(f)}"'] \& F(D)\arrow[to=B,Rightarrow,"{\chi_{f,g\c h}}"{pos=0.23}] \arrow[rr,bend right,"{F(g\c h)}"'{pos=0.35,inner sep =0.2ex},""{name=A}] \arrow[r,"{F(g)}"']\& F(E) \arrow[to=A,Rightarrow,"{\chi_{g,h}}"] \arrow[r,"{F(h)}"'] \& F(L)
			\end{cdsN}\h[2]\aM{\beta_{g,h}}\h[2]
			\begin{cdsN}{4.5}{4}
				R(C) \arrow[d,"{\alpha_C}"']\arrow[r,"{R(f)}"] \& R(D)\arrow[rr,bend right,"{R(g\c h)}"'{pos=0.35,inner sep =0.2ex},""{name=A}]\arrow[ld,Rightarrow,"{\alpha_f}", shorten <=1.85ex, shorten >=1.5ex] \arrow[d,"{\alpha_D}"]\arrow[r,"{R(g)}"]\& R(E)\arrow[to=A,Rightarrow,"{\chi_{g,h}}"]\arrow[r,"{R(h)}"] \& R(L)\arrow[lld,Rightarrow,shift left=3.9ex,"{\alpha_{g\c h}}", shorten <=5ex, shorten >=4.3ex]\arrow[d,"{\alpha_L}"] \\
				F(C)\arrow[rrr,bend right=35,"{F(f\c g\c h)}"'{pos=0.36},""{pos=0.25,name=B}] \arrow[r,"{F(f)}"'] \& F(D)\arrow[to=B,Rightarrow,"{\chi_{f,g\c h}}"{pos=0.23}] \arrow[rr,bend right,"{F(g\c h)}"'{pos=0.35,inner sep =0.2ex}] \&\& F(L)
			\end{cdsN}
			\vspace*{-1.8ex}
		\end{eqD*}
		\begin{eqD*}\vspace*{-1.8ex}
			\begin{cdsN}{3}{3}
				\hphantom{.}\arrow[d,triple,"{\chi\hspace*{0.4ex}}"']\\
				\hphantom{.}
			\end{cdsN}\hspace{8cm}
			\begin{cdsN}{3}{3}
				\hphantom{.}\arrow[d,triple,"{\hspace*{0.4ex}\beta_{f,g\c h}}"]\\
				\hphantom{.}
			\end{cdsN}
		\end{eqD*}
		\begin{eqD*}
			\begin{cdsN}{4.5}{4}
				R(C) \arrow[d,"{\alpha_C}"']\arrow[r,"{R(f)}"] \& R(D)\arrow[ld,Rightarrow,"{\alpha_f}", shorten <=1.85ex, shorten >=1.5ex] \arrow[d,"{\alpha_D}"]\arrow[r,"{R(g)}"]\& R(E)\arrow[ld,Rightarrow,"{\alpha_g}", shorten <=1.85ex, shorten >=1.5ex]\arrow[d,"{\alpha_E}"]\arrow[r,"{R(h)}"] \& R(L)\arrow[ld,Rightarrow,"{\alpha_h}", shorten <=1.85ex, shorten >=1.5ex]\arrow[d,"{\alpha_L}"] \\
				F(C)\arrow[rr,bend right,"{F(f\c g)}"'{pos=0.65,inner sep =0.2ex},""{name=A}]\arrow[rrr,bend right=35,"{F(f\c g\c h)}"'{pos=0.63},""{pos=0.7,name=B}] \arrow[r,"{F(f)}"'] \& F(D)\arrow[to=A,Rightarrow,"{\chi_{f,g}}"]  \arrow[r,"{F(g)}"']\& F(E)\arrow[to=B,Rightarrow,"{\chi_{f\c g, h}}"{pos=0.23}] \arrow[r,"{F(h)}"'] \& F(L)
			\end{cdsN}\h[2]\hphantom{.\aM{}}\h[2]
			\begin{cdsN}{4.5}{4}
				R(C) \arrow[rrr,bend right=35,"{R(f\c g\c h)}"'{pos=0.36},""{pos=0.25,name=B}]\arrow[d,"{\alpha_C}"']\arrow[r,"{R(f)}"] \& R(D)\arrow[to=B,Rightarrow,"{\chi_{f,g\c h}}"{pos=0.23}]\arrow[rr,bend right,"{R(g\c h)}"'{pos=0.35,inner sep =0.2ex},""{name=A}] \arrow[r,"{R(g)}"]\& R(E)\arrow[to=A,Rightarrow,"{\chi_{g,h}}"]\arrow[r,"{R(h)}"] \& R(L)\arrow[llld,Rightarrow,shift left=5.9ex,"{\alpha_{f\c g\c h}}", shorten <=5ex, shorten >=4.3ex]\arrow[d,"{\alpha_L}"] \\
				F(C)\arrow[rrr,bend right=35,"{F(f\c g\c h)}"'{pos=0.36},""{pos=0.25}]  \& \&\& F(L)
			\end{cdsN}
			\vspace*{-1.8ex}
		\end{eqD*}
		\begin{eqD*}\vspace*{-1.8ex}
			\begin{cdsN}{3}{3}
				\hphantom{.}\arrow[d,triple,"{\beta_{f,g}\hspace*{0.4ex}}"']\\
				\hphantom{.}
			\end{cdsN}\hspace{8cm}
			\begin{cdsN}{3}{3}
				\hphantom{.}\arrow[d,triple,"{\hspace*{0.4ex}\chi}"]\\
				\hphantom{.}
			\end{cdsN}
		\end{eqD*}
		\begin{eqD*}
			\begin{cdsN}{4.5}{4}
				R(C)\arrow[rr,bend right,"{R(f\c g)}"'{pos=0.35,inner sep =0.2ex},""{name=A}] \arrow[d,"{\alpha_C}"']\arrow[r,"{R(f)}"] \& R(D)\arrow[r,"{R(g)}"]\arrow[to=A,Rightarrow,"{\chi_{f,g}}"] \& R(E)\arrow[lld,Rightarrow,shift left=3.9ex,"{\alpha_{f\c g}}"{pos=0.55}, shorten <=5ex, shorten >=4.3ex]\arrow[d,"{\alpha_E}"]\arrow[r,"{R(h)}"] \& R(L)\arrow[ld,Rightarrow,"{\alpha_h}", shorten <=1.85ex, shorten >=1.5ex]\arrow[d,"{\alpha_L}"] \\
				F(C)\arrow[rr,bend right,"{F(f\c g)}"'{pos=0.65,inner sep =0.2ex}]\arrow[rrr,bend right=35,"{F(f\c g\c h)}"'{pos=0.63},""{pos=0.7,name=B}] \& \& F(E)\arrow[to=B,Rightarrow,"{\chi_{f\c g, h}}"{pos=0.23}] \arrow[r,"{F(h)}"'] \& F(L)
			\end{cdsN}\h[2]\aM[\hspace*{-1ex}\beta_{f\c g,h}\hspace*{-1ex}]\relax\h[2]
			\begin{cdsN}{4.5}{4}
				R(C)\arrow[rr,bend right,"{R(f\c g)}"'{pos=0.65,inner sep =0.2ex},""{name=A}]\arrow[rrr,bend right=35,"{R(f\c g\c h)}"'{pos=0.36},""{pos=0.7,name=B}]\arrow[d,"{\alpha_C}"']\arrow[r,"{R(f)}"] \& R(D)\arrow[to=A,Rightarrow,"{\chi_{f,g}}"]  \arrow[r,"{R(g)}"]\& R(E)\arrow[to=B,Rightarrow,"{\chi_{f\c g, h}}"{pos=0.23}]\arrow[r,"{R(h)}"] \& R(L)\arrow[llld,Rightarrow,shift left=5.9ex,"{\alpha_{f\c g\c h}}", shorten <=5ex, shorten >=4.3ex]\arrow[d,"{\alpha_L}"] \\
				F(C)\arrow[rrr,bend right=35,"{F(f\c g\c h)}"'{pos=0.36},""{pos=0.25}]  \& \&\& F(L)
			\end{cdsN}
		\end{eqD*}
	
	\vspace{3.5mm}
(where the names of the modifications of type $\beta$ are given without taking into account the relative whiskerings and the modifications named $\chi$ are relative to the functoriality of $\chi$).

Given $f\: D\to C$ in $\K$, the following equalities of modifications hold:
\begin{samepage}
\begin{eqD*}
	\begin{cdsN}{5.5}{5}
		R(C)\arrow[d,"{\alpha_C}"'] \arrow[r,"{R(f)}"] \& R(D)\arrow[ld,Rightarrow,"{\alpha_f}", shorten <=1.85ex, shorten >=1.5ex]\arrow[d,"{\alpha_D}"]\arrow[r,bend right,"{R(\id{D})}"',""'{name=B}] \arrow[r,equal,"{}"{name=A}]\& R(D)\arrow[ld,Rightarrow,shift left=2.1ex,"{\alpha_{\id{D}}}"{pos=0.56,inner sep=0.25ex},end anchor={[xshift=-1.85ex]}, shorten <=1.85ex, shorten >=2ex]\arrow[d,"{\alpha_D}"] \\
		F(C)\arrow[rr,bend right=35,"{F(f)}"',""{name=C}] \arrow[r,"{F(f)}"'] \& F(D) \arrow[to=C,Rightarrow,"{\chi_{f,\id{D}}}"']\arrow[r,bend right,"{F(\id{D})}"]\& F(D)
		\arrow[from=A,to=B,Rightarrow,"{\zeta}", shorten <=0.56ex, shorten >=0.28ex]
	\end{cdsN}\h[2]\aM{\gamma_D}\relax\h[2]
	\begin{cdsN}{5.5}{5}
		R(C)\arrow[d,"{\alpha_C}"'] \arrow[r,"{R(f)}"] \& R(D)\arrow[ld,Rightarrow,"{\alpha_f}", shorten <=1.85ex, shorten >=1.5ex]\arrow[d,"{\alpha_D}"] \arrow[r,equal,"{}"{name=A}]\& R(D)\arrow[ld,equal,shorten <=3.7ex, shorten >= 3.7ex]\arrow[d,"{\alpha_D}"] \\
		F(C)\arrow[rr,bend right=35,"{F(f)}"',""{name=C}] \arrow[r,"{F(f)}"'] \& F(D) \arrow[r,"{}"{name=D}, equal] \arrow[to=C,Rightarrow,"{\chi_{f,\id{D}}}"']\arrow[r,bend right,"{F(\id{D})}"'{pos=0.46}, ""'{name=F}] \arrow[from=D, to=F,"{\zeta}", Rightarrow,shorten <=0.56ex, shorten >=0.28ex]\& F(D)
	\end{cdsN}
	\vspace*{-1.8ex}
\end{eqD*}
\begin{eqD*}\vspace*{-1.8ex}\hspace*{-0.7cm}
	\begin{cdsN}{3}{3}
		\hphantom{.}\arrow[d,triple,"{\beta_{f,\id{D}}\hspace*{0.4ex}}"']\\
		\hphantom{.}
	\end{cdsN}\hspace{6.5cm}
	\begin{cdsN}{3}{3}
		\hphantom{.}\arrow[d,triple,"{\hspace*{0.4ex}\chi}"]\\
		\hphantom{.}
	\end{cdsN}
\end{eqD*}
\begin{eqD*}
	\begin{cdsN}{5.5}{5}
		R(C)\arrow[rr,bend right=35,"{R(f)}"',""{name=C}]\arrow[d,"{\alpha_C}"'] \arrow[r,"{R(f)}"] \& R(D)\arrow[to=C,Rightarrow,"{\chi_{f,\id{D}}}"']\arrow[r,bend right,"{R(\id{D})}"',""'{name=B}] \arrow[r,equal,"{}"{name=A}]\& R(D)\arrow[lld,Rightarrow,shift left=4.35ex,"{\alpha_{f}}"{pos=0.56,inner sep=0.25ex},end anchor={[xshift=-1.45ex]}, shorten <=4ex, shorten >=4ex]\arrow[d,"{\alpha_D}"] \\
		F(C)\arrow[rr,bend right=35,"{F(f)}"'] \&\& F(D)
		\arrow[from=A,to=B,Rightarrow,"{\zeta}", shorten <=0.56ex, shorten >=0.28ex]
	\end{cdsN}\h[2]\hphantom{c\aM{}}\h[2]
	\begin{cdsN}{5.5}{5}
		R(C)\arrow[d,"{\alpha_C}"'] \arrow[r,"{R(f)}"] \& R(D)\arrow[ld,Rightarrow,"{\alpha_f}", shorten <=1.85ex, shorten >=1.5ex]\arrow[d,"{\alpha_D}"] \arrow[r,equal,"{}"{name=A}]\& R(D)\arrow[ld,equal,shorten <=3.7ex, shorten >= 3.7ex]\arrow[d,"{\alpha_D}"] \\
		F(C)\arrow[rr,bend right=35,"{F(f)}"',""{name=C}] \arrow[r,"{F(f)}"'] \& F(D) \arrow[r,"{}"{name=D}, equal] \arrow[to=C,Rightarrow,""', equal ,shorten <=1.5ex, shorten >= 1.5ex] \& F(D)
	\end{cdsN}
	\vspace*{-2ex}
\end{eqD*}
\begin{eqD*}\vspace*{-2ex}
	\begin{cdsN}{3}{3}
		\hphantom{.}\arrow[rd,triple,"{\chi\hspace*{0.4ex}}"']\\
		\&\hphantom{.}
	\end{cdsN}\hspace{4.5cm}
	\begin{cdsN}{3}{3}
		\&\hphantom{.}\arrow[ld,equal,"",shorten <=3.5ex, shorten >= 3.5ex]\\
		\hphantom{.}
	\end{cdsN}
\end{eqD*}
\begin{eqD*}
	\begin{cdsN}{5.5}{5}
		R(C)\arrow[rr,bend right=35,"{R(f)}"',""{name=C}]\arrow[d,"{\alpha_C}"'] \arrow[r,"{R(f)}"] \& R(D)\arrow[to=C,equal, shorten <=0.4ex, shorten >=0.5ex] \arrow[r,equal,"{}"{name=A}]\& R(D)\arrow[lld,Rightarrow,shift left=4.35ex,"{\alpha_{f}}"{pos=0.56,inner sep=0.25ex},end anchor={[xshift=-1.45ex]}, shorten <=4ex, shorten >=4ex]\arrow[d,"{\alpha_D}"] \\
		F(C)\arrow[rr,bend right=35,"{F(f)}"'] \&\& F(D)
	\end{cdsN}
\end{eqD*}

\end{samepage}
and 

\begin{samepage}
\begin{eqD*}
	\begin{cdsN}{5.5}{5}
		R(C)\arrow[d,"{\alpha_C}"']\arrow[r,bend right,"{R(\id{C})}"',""'{name=B}] \arrow[r,equal,"{}"{name=A}]  \& R(C)\arrow[ld,Rightarrow,shift left=2.1ex,"{\alpha_{\id{C}}}"{pos=0.56,inner sep=0.25ex},end anchor={[xshift=-1.85ex]}, shorten <=1.85ex, shorten >=2ex] \arrow[d,"{\alpha_C}"] \arrow[r,"{R(f)}"]\& R(D)\arrow[ld,Rightarrow,"{\alpha_f}", shorten <=1.85ex, shorten >=1.5ex] \arrow[d,"{\alpha_D}"]\\
		F(C)\arrow[rr,bend right=55,"{F(f)}"',""{name=C}] \arrow[r,"{F(\id{C})}"'] \& F(C) \arrow[to=C,Rightarrow,"{\chi_{\id{C},f}}"']\arrow[r,"F(f)"']\& F(D)
		\arrow[from=A,to=B,Rightarrow,"{\zeta}", shorten <=0.56ex, shorten >=0.28ex]
	\end{cdsN}\h[2]\aM{\gamma_C}\relax\h[2]
	\begin{cdsN}{5.5}{5}
		R(C)\arrow[d,"{\alpha_C}"'] \arrow[r,equal,"{}"{name=A}]  \& R(C)\arrow[ld,equal, shorten <=3.85ex, shorten >=3.85ex] \arrow[d,"{\alpha_C}"] \arrow[r,"{R(f)}"]\& R(D)\arrow[ld,Rightarrow,"{\alpha_f}", shorten <=1.85ex, shorten >=1.5ex] \arrow[d,"{\alpha_D}"]\\
		F(C)\arrow[r,bend right,"{F(\id{C})}"',""'{name=B}]\arrow[rr,bend right=55,"{F(f)}"',""{name=C}] \arrow[r,""{name=V}, equal] \& F(C) \arrow[to=C,Rightarrow,"{\chi_{\id{C},f}}"]\arrow[r,"F(f)"']\& F(D)
		\arrow[from=V,to=B,Rightarrow,"{\zeta}", shorten <=0.56ex, shorten >=0.28ex]
	\end{cdsN}
	\vspace*{-1.8ex}
\end{eqD*}
\begin{eqD*}\vspace*{-1.8ex}\hspace*{-0.7cm}
	\begin{cdsN}{3}{3}
		\hphantom{.}\arrow[d,triple,"{\beta_{\id{C},f}\hspace*{0.4ex}}"']\\
		\hphantom{.}
	\end{cdsN}\hspace{6.5cm}
	\begin{cdsN}{3}{3}
		\hphantom{.}\arrow[d,triple,"{\hspace*{0.4ex}\chi}"]\\
		\hphantom{.}
	\end{cdsN}
\end{eqD*}
\begin{eqD*}
	\begin{cdsN}{5.5}{5}
		R(C) \arrow[rr,"{R(f)}"', ""{name=P}, bend right=35]\arrow[d,"{\alpha_C}"']\arrow[r,bend right,"{R(\id{C})}"',""'{name=B}] \arrow[r,equal,"{}"{name=A}]  \& R(C) \arrow[to=P,"{\chi_{\id{C},f}}", Rightarrow ,shorten <=0.2ex, shorten >= 0.2ex] \arrow[r,"{R(f)}"]\& R(D)\arrow[lld,Rightarrow,shift left=4.35ex,"{\alpha_{f}}"{pos=0.56,inner sep=0.25ex},end anchor={[xshift=-1.45ex]}, shorten <=4ex, shorten >=4ex] \arrow[d,"{\alpha_D}"]\\[2ex]
		F(C) \arrow[rr,bend right=35,"{F(f)}"'] \&\& F(D)
		\arrow[from=A,to=B,Rightarrow,"{\zeta}", shorten <=0.56ex, shorten >=0.28ex]
	\end{cdsN}\h[2]\hphantom{C\aM{}}\h[2]
	\begin{cdsN}{5.5}{5}
		R(C)\arrow[d,"{\alpha_C}"'] \arrow[r,equal,"{}"{name=A}]  \& R(C)\arrow[ld,equal, shorten <=3.85ex, shorten >=3.85ex] \arrow[d,"{\alpha_C}"] \arrow[r,"{R(f)}"]\& R(D)\arrow[ld,Rightarrow,"{\alpha_f}", shorten <=1.85ex, shorten >=1.5ex] \arrow[d,"{\alpha_D}"]\\
		F(C) \arrow[rr,bend right=35,"{F(f)}"',""{name=C}] \arrow[r,""{name=V}, equal] \& F(C) \arrow[to=C,equal,"" ,shorten <=0.35ex, shorten >= 0.35ex]\arrow[r,"F(f)"']\& F(D)
	\end{cdsN}
	\vspace*{-2ex}
\end{eqD*}
\begin{eqD*}\vspace*{-2ex}
	\begin{cdsN}{3}{3}
		\hphantom{.}\arrow[rd,triple,"{\chi\hspace*{0.4ex}}"']\\
		\&\hphantom{.}
	\end{cdsN}\hspace{4.5cm}
	\begin{cdsN}{3}{3}
		\&\hphantom{.}\arrow[ld,equal,"",shorten <=3.5ex, shorten >= 3.5ex]\\
		\hphantom{.}
	\end{cdsN}
\end{eqD*}
\begin{eqD*}
	\begin{cdsN}{5.5}{5}
		R(C) \arrow[rr,"{R(f)}"', ""{name=P}, bend right=35]\arrow[d,"{\alpha_C}"'] \arrow[r,equal,"{}"{name=A}]  \& R(C) \arrow[to=P,"", equal ,shorten <=1.7ex, shorten >=1.7ex] \arrow[r,"{R(f)}"]\& R(D)\arrow[lld,Rightarrow,shift left=4.35ex,"{\alpha_{f}}"{pos=0.56,inner sep=0.25ex},end anchor={[xshift=-1.45ex]}, shorten <=4ex, shorten >=4ex] \arrow[d,"{\alpha_D}"]\\[2ex]
		F(C) \arrow[rr,bend right=35,"{F(f)}"'] \&\& F(D).
	\end{cdsN}
\end{eqD*}
\end{samepage}

\vspace{1.5mm}
(where the names of the modifications of type $\beta$ and $\gamma$ are given without taking into account the relative whiskerings and the modifications named $\chi$ are relative to the functoriality of $\chi$).

	\end{esempio}

\begin{esempio}[trimodification] \label{ourtrimod}
	Let $R, F\:  \K\op \to \Bicat$ be trihomomorphisms and let $\theta,\phi\: R \Rightarrow F$ be tritransformations. A trimodification $m\: \theta \aM{} \phi$ is given by the following data:
	\begin{itemize}
		\item given $D\in \K$ a pseudonatural transformation $m_D\: \theta_D \aR{} \phi_D$;
		\item given $D,E\in \K$ a modification 
		
\begin{eqD*}
	\hspace*{-3ex}
	\scalebox{0.8}{
		\begin{cdsN}{6}{5}
			{\K(E,D)}  \arrow[r,"{R}"] \arrow[d,"{F}"]\& {\Bicat(R(D), R(E))} \arrow[d,"{\Bicat(1, \theta_E)}"] \arrow[ld,"{\theta}", Rightarrow ,shorten <=4.5ex, shorten >= 4.5ex] \\
			{\Bicat(F(D), F(E))} \arrow[r,"{\Bicat(\theta_D,1)}", ""{name=I}]  \arrow[r,"{\Bicat(\phi_D,1)}"',""{name=L}, bend right=68] \arrow[from=I, to=L,"{(m_D)\st}", Rightarrow ,shorten <=1.5ex, shorten >= 1.5ex]\& {\Bicat(R(D), F(E))}
		\end{cdsN}
		\begin{cdsN}{5}{4}
			{\hspace*{-1ex}}  \arrow[r,"{\wt{m}}"'{inner sep=1ex}, triple]\& {\hspace*{-2ex}}
		\end{cdsN}
		\begin{cdsN}{6}{5}
			{\K(E,D)}  \arrow[r,"{R}"] \arrow[d,"{F}"]\& {\Bicat(R(D), R(E))} \arrow[d,"{\Bicat(1, \theta_E)}",""{name=U}, bend left=68] \arrow[d,"{\Bicat(1, \phi_E)}"', ""{name=V}] \arrow[ld,"{\phi}", Rightarrow ,shorten <=4.5ex, shorten >= 4.5ex] \\
			{\Bicat(F(D), F(E))} \arrow[r,"{\Bicat(\phi_D,1)}"', ""{name=I}]  \arrow[from=V, to=U,"{(m_E)\st}", Leftarrow ,shorten <=0.5ex, shorten >= 1.5ex]\& {\Bicat(R(D), F(E)).}
	\end{cdsN}}
\end{eqD*}
whose component relative to every $g\: E \to D$
\begin{eqD*}
	\begin{cdN}
			{R(D)} \arrow[r,"{\theta_D}"',""{name=P}] \arrow[d,"{R(g)}"'] \arrow[r,"{\phi_D}", ""{name=O}, bend left= 60] \arrow[from=O,to=P,"{m_D}", Leftarrow ,shorten <=0.5ex, shorten >= 0.5ex]\& {F(D)}  \arrow[d,"{F(g)}"]\\
			{R(E)} \arrow[ru,"{\theta_g}"' ,shorten <=3.5ex, shorten >= 3.5ex, Rightarrow] \arrow[r,"{\theta_E}"']\& {F(E)} 
		\end{cdN}
		\begin{cdN}
			{\hphantom{.}}  \arrow[r,"{{m_g}}"'{inner sep= 1ex}, triple]\& {\hphantom{.}}
		\end{cdN}
		\begin{cdN}
			{R(D)}  \arrow[d,"{R(g)}"'] \arrow[r,"{\phi_D}", ""{name=O}] \& {F(D)}  \arrow[d,"{F(g)}"]\\
			{R(E)} \arrow[ru,"{\phi_g}"' ,shorten <=3.5ex, shorten >= 3.5ex, Rightarrow] \arrow[r,"{\theta_E}"', ""{name=R}, bend right= 60] \arrow[r,"{\phi_E}",""{name=Q}]\arrow[from=R,to=Q,"{m_E}", Rightarrow ,shorten <=0.5ex, shorten >= 0.5ex]\& {F(E)} 	\end{cdN}
\end{eqD*}
is an invertible modification.
	\end{itemize}
These data need to satisfy the following axioms.

\begin{samepage}
	Given $L\ar{g} E \ar{f} D$ in $\K$, the following equality of modifications hold
	\begin{eqD*}
	\begin{cdN}
		{R(L)} \arrow[d,"{R(h)}"'] \arrow[r,"{\theta_L}"', ""{name=W}] \arrow[r,"{\phi_L}", ""{name=E}, bend left= 60] \arrow[from=W, to=E,"{m_L}", Rightarrow ,shorten <=0.5ex, shorten >= 0.5ex] \& {F(L)} \arrow[dd,"{F(g\c h)}",""{name=J}, bend left =60] \arrow[d,"{F(h)}"]\\
		{R(E)} \arrow[ru,"{\theta_h}"', Rightarrow ,shorten <=2.5ex, shorten >= 2.5ex] \arrow[r,"{\theta_E}"] \arrow[d,"{R(g)}"']\& {F(E)} \arrow[to=J,"{\eta_{g,h}}", Rightarrow ,shorten <=1ex, shorten >= 1ex] \arrow[d,"{F(g)}"] \\
		{R(D)} \arrow[ru,"{\theta_g}"', Rightarrow ,shorten <=2.5ex, shorten >= 2.5ex]  \arrow[r,"{\theta_D}"'] \& {F(D)} 
	\end{cdN}\h[2]\aM{m_h}\relax\h[2]
	\begin{cdN}
		{R(L)} \arrow[d,"{R(h)}"'] \arrow[r,"{\phi_L}", ""{name=W}]  \& {F(L)} \arrow[dd,"{F(g\c h)}",""{name=J}, bend left =60] \arrow[d,"{F(h)}"]\\
		{R(E)} \arrow[r,"{\phi_E}", ""{name=E}, bend left= 60]  \arrow[ru,"{\phi_h}", shift left= 2.8ex, Rightarrow ,shorten <=3.2ex, shorten >= 2.5ex] \arrow[r,"{\theta_E}"', ""{name=S}] \arrow[from=S, to=E,"{m_E}", Rightarrow ,shorten <=0.5ex, shorten >= 0.5ex] \arrow[d,"{R(g)}"']\& {F(E)} \arrow[to=J,"{\eta_{g,h}}", Rightarrow ,shorten <=1ex, shorten >= 1ex] \arrow[d,"{F(g)}"] \\
		{R(D)} \arrow[ru,"{\theta_g}"', Rightarrow ,shorten <=2.5ex, shorten >= 2.5ex]  \arrow[r,"{\theta_D}"'] \& {F(D)} 
	\end{cdN}
	\vspace*{-1.8ex}
\end{eqD*}
\begin{eqD*}\vspace*{-1.8ex}\hspace*{-0.7cm}
	\begin{cdsN}{3}{3}
		\hphantom{.}\arrow[d,triple,"{\theta\hspace*{0.7ex}}"']\\
		\hphantom{.}
	\end{cdsN}\hspace{6.5cm}
	\begin{cdsN}{3}{3}
		\hphantom{.}\arrow[d,triple,"{\hspace*{0.4ex}m_g}"]\\
		\hphantom{.}
	\end{cdsN}
\end{eqD*}
\begin{eqD*}
	\begin{cdN}
		{R(L)}  \arrow[dd,bend left=50,"{R(g\c h)}",""{name=J}] \arrow[d,"{R(h)}"'] \arrow[r,"{\theta_L}"', ""{name=W}] \arrow[r,"{\phi_L}", ""{name=E}, bend left= 60] \arrow[from=W, to=E,"{m_L}", Rightarrow ,shorten <=0.5ex, shorten >= 0.5ex] \& {F(L)}  \arrow[dd,"{F(g\c h)}", bend left =40]\\
		{R(E)} \arrow[to=J,"{\eta_{g,h}}"{pos=0.4}, Rightarrow ,shorten <=-0.2ex, shorten >= 0.4ex] \arrow[d,"{R(g)}"'] \\
		{R(D)} \arrow[ruu,shift right=5ex,"{\theta_{g\c h}}"', Rightarrow ,shorten <=8.5ex, shorten >= 8ex] \arrow[r,"{\theta_D}"'] \& {F(D)} 
	\end{cdN}\h[2]\hphantom{C\aM{}}\h[2]
	\begin{cdN}
		{R(L)} \arrow[d,"{R(h)}"'] \arrow[r,"{\phi_L}"] \& {F(L)} \arrow[dd,"{F(g\c h)}",""{name=J}, bend left =60] \arrow[d,"{F(h)}"]\\
		{R(E)} \arrow[ru,"{\phi_h}"', Rightarrow ,shorten <=2.5ex, shorten >= 2.5ex] \arrow[r,"{\phi_E}"] \arrow[d,"{R(g)}"']\& {F(E)} \arrow[to=J,"{\eta_{g,h}}", Rightarrow ,shorten <=1ex, shorten >= 1ex] \arrow[d,"{F(g)}"] \\
		{R(D)}  \arrow[r,"{\theta_D}"', ""{name=E}, bend right= 60]  \arrow[ru,"{\phi_g}"', Rightarrow ,shorten <=2.5ex, shorten >= 2.5ex]  \arrow[r,"{\phi_D}", ""{name=W}] \arrow[from=W, to=E,"{m_D}", Leftarrow ,shorten <=0.5ex, shorten >= 0.5ex] \& {F(D)} 
	\end{cdN}
	\vspace*{-2ex}
\end{eqD*}
\begin{eqD*}\vspace*{-2ex}\hspace*{-0.7cm}
	\begin{cdsN}{3}{3}
		\hphantom{.}\arrow[rd,triple,"{m_{g\c h}\hspace*{0.4ex}}"']\\
		\&\hphantom{.}
	\end{cdsN}\hspace{4cm}
	\begin{cdsN}{3}{3}
		\&\hphantom{.}\arrow[ld,triple,"{\phi\hspace*{0.4ex}}"]\\
		\hphantom{.}
	\end{cdsN}
\end{eqD*}
\begin{eqD*}
	\begin{cdN}
		{R(L)}  \arrow[dd,bend left=50,"{R(g\c h)}",""{name=J}] \arrow[d,"{R(h)}"'] \arrow[r,"{\phi_L}"', ""{name=W}]  \& {F(L)} \arrow[dd,"{F(g\c h)}", bend left =40] \\
		{R(E)} \arrow[to=J,"{\eta_{g,h}}"{pos=0.4}, Rightarrow ,shorten <=-0.2ex, shorten >= 0.4ex] \arrow[d,"{R(g)}"'] \\
		{R(D)} \arrow[ruu,shift right=5ex,"{\phi_{g\c h}}"', Rightarrow ,shorten <=8.5ex, shorten >= 8ex] \arrow[r,"{\phi_D}", ""{name=A}] \arrow[r,"{\theta_D}"', ""{name=E}, bend right= 60] \arrow[from=A, to=E,"{m_D}", Leftarrow ,shorten <=0.5ex, shorten >= 0.5ex] \& {F(D)} 
	\end{cdN}
\end{eqD*}
\end{samepage}
(where the names of the modifications of type $m$ are given without taking into account the relative whiskerings and the modifications named $\phi$ and $\theta$ are relative to the  of $\phi$ and $\theta$).

Given $D\in \K$, the following equalities of modifications hold:

\begin{eqD*}
	\begin{cdN}
		{R(D)} \arrow[r,"{\theta_D}"',""{name=V}] \arrow[r,"{\phi_D}",""{name=W}, bend left= 60] \arrow[from=V,to=W,"{m_D}", Rightarrow ,shorten <=0.5ex, shorten >= 0.5ex] \arrow[d,"{R(\id{D})}",""{name=D}] \arrow[d,"{}", equal, bend right=60,""{name=C}] \arrow[from=C, to=D,Rightarrow,"{\zeta}",shorten <=0.5ex, shorten >= 1ex]\& {F(D)} \arrow[d,"{F(\id{D})}"] \\
		{R(D)} \arrow[ru,"{\theta_{\id{D}}}"', Rightarrow,shorten <=2.5ex, shorten >= 2.5ex] \arrow[r,"{\theta_D}"'] \& {F(D)}
	\end{cdN}\h[2]\aM{m_{\id{D}}}\relax\h[2]
	\begin{cdN}
		{R(D)} \arrow[r,"{\phi_D}"]  \arrow[d,"{R(\id{D})}",""{name=D}] \arrow[d,"{}", equal, bend right=60,""{name=C}] \arrow[from=C, to=D,"{\zeta}",Rightarrow,shorten <=0.5ex, shorten >= 1ex]\& {F(D)} \arrow[d,"{F(\id{D})}"] \\
		{R(D)} \arrow[ru,"{\phi_{\id{D}}}"', Rightarrow,shorten <=2.5ex, shorten >= 2.5ex] \arrow[r,"{\theta_D}"',""{name=W}, bend right=60] \arrow[r,"{\phi_D}", ""{name=V}] \arrow[from=W,to=V,"{m_D}"', Rightarrow ,shorten <=0.5ex, shorten >= 0.5ex]\& {F(D)}
	\end{cdN}
	\vspace*{-1.8ex}
\end{eqD*}
\begin{eqD*}\vspace*{-1.8ex}\hspace*{-0.7cm}
	\begin{cdsN}{3}{3}
		\hphantom{.}\arrow[d,triple,"{\theta\hspace*{0.7ex}}"']\\
		\hphantom{.}
	\end{cdsN}\hspace{6cm}
	\begin{cdsN}{3}{3}
		\hphantom{.}\arrow[d,triple,"{\hspace*{0.7ex}\phi}"]\\
		\hphantom{.}
	\end{cdsN}
\end{eqD*}
\begin{eqD*}
	\begin{cdN}
		{R(D)} \arrow[r,"{\theta_D}"',""{name=V}] \arrow[r,"{\phi_D}",""{name=W}, bend left= 60] \arrow[from=V,to=W,"{m_D}", Rightarrow ,shorten <=0.5ex, shorten >= 0.5ex] \arrow[d,equal,"",""{name=D}]  \& {F(D)} \arrow[d,"{F(\id{D})}", bend left=60, ""{name=X}]  \arrow[d,"{}"{name=Z}, equal]\\
		{R(D)} \arrow[ru,equal,""',shorten <=5ex, shorten >= 5ex] \arrow[r,"{\theta_D}"'] \arrow[from=Z, to=X,"{\zeta}",Rightarrow,shorten <=0.5ex, shorten >= 1ex]\& {F(D)}
	\end{cdN}\h[2]\aM[]\relax\h[2]
	\begin{cdN}
		{R(D)} \arrow[r,"{\phi_D}"]  \arrow[d,equal,""{name=D}] \& {F(D)} \arrow[d,"{F(\id{D})}",""{name=T}, bend left=60]  \arrow[d,"{}"{name=S}, equal]\\
		{R(D)} \arrow[ru,equal,""',shorten <=5ex, shorten >= 5ex] \arrow[r,"{\theta_D}"',""{name=W}, bend right=60] \arrow[r,"{\phi_D}", ""{name=V}] \arrow[from=S, to=T,"{\zeta}",Rightarrow,shorten <=0.5ex, shorten >= 1ex] \arrow[from=W,to=V,"{m_D}"', Rightarrow ,shorten <=0.5ex, shorten >= 0.5ex]\& {F(D)}
	\end{cdN}
\end{eqD*}

(where the names of the modifications of type $m$ are given without taking into account the relative whiskerings and the modifications named $\phi$ and $\theta$ are relative to the  of $\phi$ and $\theta$).

\end{esempio}

\begin{esempio}[perturbation] \label{ourpert}
	Let $R, F\:  \K\op \to \Bicat$ be trihomomorphisms, let $\theta,\phi\: R \Rightarrow F$ be tritransformations and let $m,n\: \theta \aM{} \phi$ be trimodifications. A perturbation $p\: m \aP{} n$ is given by a family of modifications $p_D\: m_D \aM{} n_D$ indexed by the objects of $\K$ such that the following equality of modifications hold for every morphism $g\: E \to D$ in $R(D)$:
	
	\begin{samepage}
	\begin{eqD*}
		\begin{cdN}
			{R(D)} \arrow[r,"{\theta_D}"',""{name=P}] \arrow[d,"{R(g)}"'] \arrow[r,"{\phi_D}", ""{name=O}, bend left= 60] \arrow[from=O,to=P,"{m_D}", Leftarrow ,shorten <=0.5ex, shorten >= 0.5ex]\& {F(D)}  \arrow[d,"{F(g)}"]\\
			{R(E)} \arrow[ru,"{\theta_g}"' ,shorten <=3.5ex, shorten >= 3.5ex, Rightarrow]  \arrow[r,"{\theta_E}"']\& {F(E)} 
		\end{cdN}
		\begin{cdN}
			{\hphantom{.}}  \arrow[r,"{{m_g}}"'{inner sep= 1ex}, triple]\& {\hphantom{.}}
		\end{cdN}
		\begin{cdN}
			{R(D)}  \arrow[d,"{R(g)}"'] \arrow[r,"{\phi_D}", ""{name=O}] \& {F(D)}  \arrow[d,"{F(g)}"]\\
			{R(E)} \arrow[ru,"{\phi_g}"' ,shorten <=3.5ex, shorten >= 3.5ex, Rightarrow]  \arrow[r,"{\theta_E}"', ""{name=R}, bend right= 60] \arrow[r,"{\phi_E}",""{name=Q}]\arrow[from=R,to=Q,"{m_E}", Rightarrow ,shorten <=0.5ex, shorten >= 0.5ex]\& {F(E)} 	\end{cdN}
		\vspace*{-1.8ex}
	\end{eqD*}
	\begin{eqD*}\vspace*{-1.8ex}\hspace*{-0.7cm}
		\begin{cdsN}{3}{3}
			\hphantom{.}\arrow[d,triple,"{p_D\hspace*{0.7ex}}"']\\
			\hphantom{.}
		\end{cdsN}\hspace{5.4cm}
		\begin{cdsN}{3}{3}
			\hphantom{.}\arrow[d,triple,"{\hspace*{0.4ex}p_E}"]\\
			\hphantom{.}
		\end{cdsN}
	\end{eqD*}
	\begin{eqD*}
		\begin{cdN}
			{R(D)} \arrow[r,"{\theta_D}"',""{name=P}] \arrow[d,"{R(g)}"'] \arrow[r,"{\phi_D}", ""{name=O}, bend left= 60] \arrow[from=O,to=P,"{n_D}", Leftarrow ,shorten <=0.5ex, shorten >= 0.5ex]\& {F(D)}  \arrow[d,"{F(g)}"]\\
			{R(E)} \arrow[ru,"{\theta_g}"' ,shorten <=3.5ex, shorten >= 3.5ex, Rightarrow]  \arrow[r,"{\theta_E}"']\& {F(E)} 
		\end{cdN}
		\begin{cdN}
			{\hphantom{.}}  \arrow[r,"{{n_g}}"'{inner sep= 1ex}, triple]\& {\hphantom{.}}
		\end{cdN}
		\begin{cdN}
			{R(D)}  \arrow[d,"{R(g)}"'] \arrow[r,"{\phi_D}", ""{name=O}] \& {F(D)}  \arrow[d,"{F(g)}"]\\
			{R(E)} \arrow[ru,"{\phi_g}"' ,shorten <=3.5ex, shorten >= 3.5ex, Rightarrow]  \arrow[r,"{\theta_E}"', ""{name=R}, bend right= 60] \arrow[r,"{\phi_E}",""{name=Q}]\arrow[from=R,to=Q,"{n_E}", Rightarrow ,shorten <=0.5ex, shorten >= 0.5ex]\& {F(E)} 	\end{cdN}
	\end{eqD*}
\end{samepage}
\end{esempio}

We present the tricategorical version of the Yoneda lemma that has been proven by Buhn\'e's in their PhD thesis \cite{Buhneithesis}. This result will be useful in section \ref{sec2stacks}.
\begin{teor}[Tricategorical Yoneda lemma, \cite{Buhneithesis}]\label{Yonedatricat}
	Let $T$ be a tricategory and let $F\: T\op \to \Bicat$ be a trihomomorphism. For every $C\in T$ the evaluation  of the identity at $C$ induces a biequivalence 
	\vspace{-1.5mm}
	$$ \Tricat(T\op, \Bicat)(y(C), F) \simeq F(C)$$
	which is natural in $C$. 
\end{teor}

We now recall the explicit action of the biequivalence of Theorem \ref{Yonedatricat}. 

\begin{oss}\label{assYonedatricat}
	The action of the biequivalence given by the Tricategorical Yoneda lemma is the following:
	\begin{itemize}
		\item an object $X\in F(C)$ corresponds to a tritransformation $\sigma^X\: y(C) \Rightarrow F$, whose component $(\sigma^X)_D$ on $D\in \K$ is a pseudofunctor sending $E\ar{f}D$ in $y(D)$ to $F(f)(X)$;
		\item a morphism $a\: X\to Y$ corresponds to a trimodification $m^a\: \sigma^X \aM{} \sigma^Y$ whose component $(m^a)_D$ on $D\in \K$ is a pseudonatural transformation that has component $((m^a)_D)_f$  relative to $E\ar{f}D$ equal to $F(f)(a)$;
		\item a 2-cell $\alpha\: a\Rightarrow b$ in $F(C)$ corresponds to a perturbation $p^{\alpha}\: m^a \aP{} m^b$ whose component $(p^{\alpha})_D$ on $D\in K$ is a modification that has component $((p^{\alpha})_D)_f$ relative to $E\ar{f}D$ equal to $F(f)(\alpha)$.
	\end{itemize}
	\end{oss}
\section{Bisites} \label{secbisites}

Let $\K$ be a small (strict) 2-category.  

\begin{oss}
	The definitions of this paper could be given more in general for a bicategory and all the results proven in this more general context. The decision to do everything in the less general context of 2-categories simplifies the calculations involved. One should also keep in mind that every bicategory is biequivalent to a 2-category. Moreover, every bicategory with finite bilimits is biequivalent to a 2-category with finite flexible limits. This result has been proven by Power in \cite{Powercohresult}.
\end{oss}

We recall the notion of bi-iso-comma object (called bipullbacks in \cite{Streetcharbicatstacks}) that is a bicategorical analogue of the notion of pullback.

\begin{defne}
	Let $f\: A \to C$ and $g\: B\to C$ be morphisms in $\K$. The \dfn{bi-iso-comma object} of $f$ and $g$ is an object $A\xp{C}{f,g}B$ of $\K$ (sometimes simply denoted by $A\xp{C}{}B$) equipped with morphisms $f\st g\: A\xp{C}{f,g}B \to A$ and $g\st f\: A\xp{C}{f,g}B \to B$ and an invertible 2-cell
	\biisocomma{A\xp{C}{f,g}B }{A}{B}{C}{f\st g}{g\st f}{f}{g}{\lambda^{f,g}}{3}{3}
	such that the following universal properties are satisfied:
	\begin{itemize}
		\item given an object $D\in K$, morphisms $v\: D\to A$ and $w\: D\to B$ and an isomorphic 2-cell 
		\biisocomma{D}{A}{B}{C}{v}{w}{f}{g}{\delta}{3}{3}
		there exists a morphism $u\: D \to A\xp{C}{f,g}B$ together with 2-cells $\delta_1$ and $\delta_2$ such that the following equality holds
			\begin{eqD*}
			\bicommaunivvN{D}{v}{w}{u}{A\xp{C}{f,g}B}{A}{B}{C}{f\st g}{g\st f}{f}{g}{\lambda^{f,g}}{\delta_1}{\delta_2}
			\qquad = \qquad 
			\csq*[l][7][7][\delta]{D}{A}{B}{C}{v}{w}{f}{g}
		\end{eqD*}
	\item given morphisms $u,t\: D\to A\xp{C}{f,g}B$ and isomorphic 2-cells
	\begin{eqD*}
		\csq*[l][7][7][\alpha]{D}{A\xp{C}{f,g}B}{A\xp{C}{f,g}B}{A}{u}{f\st g}{t}{f\st g}
		\qquad \text{and} \qquad 
		\csq*[l][7][7][\beta]{D}{A\xp{C}{f,g}B}{A\xp{C}{f,g}B}{B}{u}{g\st f}{t}{g\st f}
	\end{eqD*}
such that 
\begin{eqD*}
	\begin{cdN}
		{D} \arrow[r,"{u}"] \arrow[d,"{t}"'] \& {A\xp{C}{f,g}B} \arrow[ld,"{\alpha}",Rightarrow ,shorten <=2.5ex, shorten >= 2.5ex ] \arrow[d,"{f\st g}"] \\
		{A\xp{C}{f,g}B} \arrow[r,"{f\st g}"] \arrow[d,"{g\st f}"']\& {A} \arrow[d,"{f}"] \arrow[ld,"{\lambda^{f,g}}",Rightarrow ,shorten <=3.5ex, shorten >= 3.5ex ] \\
		{B} \arrow[r,"{g}"'] \& {C} 
	\end{cdN}
	\h[3]= \h[3]
	\begin{cdN}
		{D} \arrow[r,"{u}"] \arrow[d,"{t}"']\& {A\xp{C}{f,g}B} \arrow[ld,"{\beta}",Rightarrow ,shorten <=2.5ex, shorten >= 2.5ex ] \arrow[r,"{f\st g}"] \arrow[d,"{g\st f}"]\& {A} \arrow[ld,"{\lambda^{f,g}}",Rightarrow ,shorten <=3.5ex, shorten >= 3.5ex ] \arrow[d,"{f}"] \\
		{A\xp{C}{f,g}B} \arrow[r,"{g\st f}"']\& {B} \arrow[r,"{g}"']\& {C}
	\end{cdN}
\end{eqD*}
there exists a unique 2-cell 
$$\Gamma\: u \Rightarrow t$$
such that 
$$f\st g \star \Gamma=\alpha \qquad \text{and} \qquad g\st f \star \Gamma = \beta$$
	\end{itemize}
	\end{defne}

In what follows, we will also need to consider a stricter version of the bi-iso-comma object that is called \predfn{iso-comma object}.

\begin{defne}
	Let $f\: A \to C$ and $g\: B\to C$ be morphisms in $\K$. The \dfn{iso-comma object} of $f$ and $g$ is an object $A\xp{C}{f,g}B$ of $\K$ (sometimes simply denoted by $A\xp{C}{}B$) equipped with morphisms $f\st g\: A\xp{C}{f,g}B \to A$ and $g\st f\: A\xp{C}{f,g}B \to B$ and an invertible 2-cell
	\biisocomma{A\xp{C}{f,g}B }{A}{B}{C}{f\st g}{g\st f}{f}{g}{\lambda^{f,g}}{3}{3}
	such that the following universal properties are satisfied:
	\begin{itemize}
		\item given an object $D\in K$, morphisms $v\: D\to A$ and $w\: D\to B$ and an isomorphic 2-cell 
		\biisocomma{D}{A}{B}{C}{v}{w}{f}{g}{\delta}{3}{3}
		there exists a unique morphism $u\: D \to A\xp{C}{f,g}B$ such that the following equality holds
		\begin{eqD*}
			\commaunivvN{D}{v}{w}{u}{A\xp{C}{f,g}B}{A}{B}{C}{f\st g}{g\st f}{f}{g}{\lambda^{f,g}}
			\qquad = \qquad 
			\csq*[l][7][7][\delta]{D}{A}{B}{C}{v}{w}{f}{g}
		\end{eqD*}
		\item given morphisms $u,t\: D\to A\xp{C}{f,g}B$ and isomorphic 2-cells
		\begin{eqD*}
			\csq*[l][7][7][\alpha]{D}{A\xp{C}{f,g}B}{A\xp{C}{f,g}B}{A}{u}{f\st g}{t}{f\st g}
			\qquad \text{and} \qquad 
			\csq*[l][7][7][\beta]{D}{A\xp{C}{f,g}B}{A\xp{C}{f,g}B}{B}{u}{g\st f}{t}{g\st f}
		\end{eqD*}
		such that 
		\begin{eqD*}
			\begin{cdN}
				{D} \arrow[r,"{u}"] \arrow[d,"{t}"'] \& {A\xp{C}{f,g}B} \arrow[ld,"{\alpha}",Rightarrow ,shorten <=2.5ex, shorten >= 2.5ex ] \arrow[d,"{f\st g}"] \\
				{A\xp{C}{f,g}B} \arrow[r,"{f\st g}"] \arrow[d,"{g\st f}"']\& {A} \arrow[d,"{f}"] \arrow[ld,"{\lambda^{f,g}}",Rightarrow ,shorten <=3.5ex, shorten >= 3.5ex ] \\
				{B} \arrow[r,"{g}"'] \& {C} 
			\end{cdN}
			\h[3]= \h[3]
			\begin{cdN}
				{D} \arrow[r,"{u}"] \arrow[d,"{t}"']\& {A\xp{C}{f,g}B} \arrow[ld,"{\beta}",Rightarrow ,shorten <=2.5ex, shorten >= 2.5ex ] \arrow[r,"{f\st g}"] \arrow[d,"{g\st f}"]\& {A} \arrow[ld,"{\lambda^{f,g}}",Rightarrow ,shorten <=3.5ex, shorten >= 3.5ex ] \arrow[d,"{f}"] \\
				{A\xp{C}{f,g}B} \arrow[r,"{g\st f}"']\& {B} \arrow[r,"{g}"']\& {C}
			\end{cdN}
		\end{eqD*}
		there exists a unique 2-cell 
		$$\Gamma\: u \Rightarrow t$$
		such that 
		$$f\st g \star \Gamma=\alpha \qquad \text{and} \qquad g\st f \star \Gamma = \beta$$
	\end{itemize}
\end{defne}

\begin{oss}
	We notice that the iso-comma-object, if it exists, is a particular representative of the bi-iso-comma object. In fact, the universal properties of the iso-comma object imply those of the bi-iso-comma object.
\end{oss}

We want to endow $\K$ with an appropriate two-dimensional Grothendieck topology. We could simply consider a Grothendieck topology on the underlying $1$-category of $\K$. This choice is quite common in the literature, but it is not the right choice for us since we will need a topology that really takes into account the two-dimensional nature of $\K$. For this reason, we endow $\K$ with a bitopology in the sense of Street. This notion has been introduced in greater generality for a bicategory in \cite{Streetcharbicatstacks}.

We start by introducing the notion of sieve in this two-categorical context. 

\begin{defne}[\cite{Streetcharbicatstacks}]
	A \dfn{bisieve} $S$ over $C\in\K$ is a fully faithful morphism \linebreak[4] ${S\: R\Rightarrow y(C)}$ in $\Bicat(\K\op, \Cat)$.
\end{defne}

We now make some useful remarks about the previous definition.

\begin{oss}
	Notice that a bisieve is not required to be injective on objects. But every bisieve is equivalent to one that is injective on objects. Since all the notions that we will use are invariant under equivalence of bisieves, we will only consider bisieves that are injective on objects. This will allow us to think of a bisieve as a collection of morphisms in $\K$ with a common target and 2-cells between them. 
\end{oss}

\begin{oss}
	We observe that, since a bisieve is required to be full, every 2-cell between morphisms of the bisieve is in it. 
\end{oss}
\begin{oss}[normal bisieves]
We will always consider bisieves whose source strictly preserves identities. We call \dfn{normal bisieves} the ones of this form. This choice is not restrictive as every bisieve is equivalent to a normal one.
\end{oss}
\begin{oss}
	Notice that, since the transformation $S$ is just pseudonatural, the bisieve is required to be closed under precomposition only up to isomorphism. 
\end{oss}

\begin{notazione}
	Given a morphism $f\: D\to C$ in $S$ and a morphism $g\: E \to D$ in $\K$, we will denote by $\wt{f\c g}$ the morphism which is isomorphic to $f\c g$ and is in $S$. Notice that this coincides with $R(g)(f)$. Moreover, we will call $\sigma_{f,g}$ the isomorphism $(S_g)_f$ between $\wt{f\c g}$ and $f\c g$ given by the pseudonaturality of $S$.
\end{notazione}

The following construction gives an alternative way to encode the data of a bisieve that will be useful for us. This is probably well-known, but it does not seem to appear in the literature. The construction will use the 2-category of elements construction introduced by Street in \cite{Streetlimits}. This is a 2-categorical generalization of the usual category of elements construction (also called Grothendieck construction).

\begin{costr}\label{bisieveGroth}
Let $S\: R \Rightarrow y(C)$ be a bisieve over $C\in \K$.  Consider the \linebreak 2-category of elements $\Groth{R}$  of the pseudofunctor $R\: \K\op \to \Cat$. Then $\Groth{R}$  is a 2-category that has:
\begin{itemize}
	\item[-] as objects the pairs $(D, D\ar{f} C)$, where $D\in \K$ and $f\in R(D)$;
	\item[-] as morphisms from $(E, E\ar{h} C)$ to $(D, D\ar{f} C)$ the pairs $(g, \alpha)$ with $g\: E\to D$ is a morphism in $\K$ and $\alpha\: h \Rightarrow \wt{f\c g}$ is a 2-cell in $\K$;
	\item[-] as 2-cells from $(g, \alpha)\: (E, h) \to (D,f)$ to $(g', \beta)\: (E, h) \to (D,f)$ the 2-cells $\delta\: g \Rightarrow g'$ in $\K$ such that 
	\begin{eqD*}
		\begin{cdsN}{12}{12}
			{E} \arrow[r,"{h}", ""'{name=A}, bend left=60] \arrow[r,"{\wt{f\c g}}"{pos=0.3},""{name=B},""'{name=D}] \arrow[r,"{\wt{f\c g'}}"',""{name=C}, bend right=60] \arrow[from=A, to=B,"{\alpha}", Rightarrow] \arrow[from=D, to=C,"{R(\delta)_{f}}", Rightarrow] \& {C}
		\end{cdsN}
		\h = \h
		\begin{cdsN}{12}{12}
			{E} \arrow[r,"{h}", ""'{name=A}, bend left=60]  \arrow[r,"{\wt{f\c g'}}"',""{name=C}, bend right=60] \arrow[from=A, to=C,"{\beta}", Rightarrow ,shorten <=1.5ex, shorten >= 1.5ex] \& {C.}
		\end{cdsN}
	\end{eqD*}
\end{itemize}
The 2-category $\Groth{R}$ encodes the data of the bisieve $S$. Indeed, the objects of $\Groth{R}$  correspond to the morphisms in $S$ and the 2-cells correspond to the 2-cells of $S$. The morphisms of $\Groth{R}$ allows one to relate morphisms of $S$ with different sources and encode the pseudonaturality of $S$.
\end{costr}

\begin{oss}\label{splitmorGr}
	Given a morphism $(g, \alpha)\: (E, h) \to (D,f)$  in $\Groth{R}$ we can write it as the composite of $(\id{E}, \alpha) \: (E, h) \to (E, \wt{f\c g})$ and $(g, \id{\wt{f\c g}})\: (E, \wt{f\c g}) \to (D,f)$. 
\end{oss}

We are now ready to recall the notion of  bitopology.

\begin{defne}[\cite{Streetcharbicatstacks}] \label{bitopology}
	A \dfn{Grothendieck bitopology} $\tau$ on $\K$ is an assignment for each object $C\in \K$ of a collection $\tau(C)$ of bisieves on $C$, called \dfn{covering bisieves}, in a way such that
	\begin{itemize}
		\item[(T1)] the identity of $y(C)$  is in $\tau(C)$;
		\item[(T2)] for all $S\: R\Rightarrow y(C)$ in $\tau(\cY)$ and all arrows $f\: D\to C$ in $\K$, the bi-iso-comma object 
		\vspace{-2.5mm}
		{\biisocomma{ P}{y(D)}{R}{y(C)}{}{}{-\c f}{S}{}{2}{2}}
		has the top arrow in $\tau(D)$;
		
		\item[(T3)] if $S'\: R'\Rightarrow y(C)$ is in $\tau(C)$ and $S\: R\Rightarrow y(C)$ is a bisieve such that for every $f\: D \to C$ in $R'(D)$ there exists a bi-iso-comma object 
		\vspace{-2.5mm}
		{\biisocomma{ P}{y(D)}{R}{y(C)}{}{}{-\c f}{S}{}{2}{2}}
		such that the top arrow is in $\tau(D)$ then $S$ is equivalent to a covering bisieve in $\tau(C)$.
	\end{itemize}
	The pair $(\K, \tau)$ with $\tau$ a Grothendieck bitopology is called \dfn{bisite}.
\end{defne}

\begin{oss}\label{bisievefstarS}
	Axiom (T2) of Definition \ref{bitopology} ensures that the collection
	$$f\st S=\{E\ar{g} D|f\c g \text{ is isomorphic to a morphism in } S\}$$
	is a covering bisieve.  Given a morphism $E\ar{g} D$ in $f\st S$, we will denote by $\overline{f\c g}$ the morphism isomorphic to $f\c g$ that is in $S$.
	
	\noindent Notice that every morphism $E\ar{g} D\in f\st S$ is isomorphic to a composite of the form 
	$$E \ar{m_g} E\x[C] D \ar{f\st ({\wt{f\c g}})} D,$$
	where $m_g$ is induced by the universal property of the bi-iso-comma object $E\x[C] D$ as in the following diagram
	\begin{eqD*}
		\bicommaunivvN{E}{\id{E}}{g}{m_g}{E\x[C] D}{E}{D}{C}{(\wt{f\c g})\st f}{(\wt{f\c g})\st f}{\wt{f\c g}}{f}{\lambda^{\wt{f\c g},f}}{\sigma_1}{\sigma_2}
		\qquad = \qquad 
		\csq*[l][7][7][\sigma_{f,g}]{E}{E}{D}{C}{\id{E}}{g}{\wt{f\c g}}{f}
	\end{eqD*}
	This means that, up to isomorphism and precomposition with an appropriate morphism, the morphisms of $f\st S$ are of the form $f\st h$ with $h\in S$. This is the appropriate higher dimensional analogue of the closure under pullback of a Grothendieck topology. 
\end{oss}

\begin{oss}
	Axiom (T3) is essentially saying that being a covering bisieve can be checked locally with respect to a covering bisieve. This is the analogue of one of the axioms required for a Grothendieck topology.
\end{oss}

Among Grothendieck bitopologies we can distinguish the ones that make representables satisfy the appropriate gluing conditions. In this two-dimensional context we will ask that the representables are stacks, that are the two-dimensional analogues of sheaves. Since the usual notion of stack involves a pseudofunctor from a category to $\Cat$ and the representables have as source the 2-category $\K$, we will need to consider the following generalized notion of stack defined by Street in \cite{Streetcharbicatstacks}.

\begin{defne}[\cite{Streetcharbicatstacks}] \label{defstack2cat}
		Let $F\: \K\op \to \Cat$ be a pseudofunctor. $F$ is a \dfn{stack} if for every $C\in \K$ and every covering bisieve $S\: R\Rightarrow y(C)$ in $\tau(C)$ the functor
			$$-\c S\: \HomC{\Psm{\K\op}{\Cat}}{\HomC{\K}{-}{C}}{ F} \longrightarrow \HomC{\Psm{\K\op}{ \Cat}}{R}{ F}\vspace{-2.5mm}$$
			is an equivalence of categories.
	\end{defne}

\begin{oss}\label{charstack}
	The previous definition in the case of a pseudofunctor that has a category as source is equivalent to the usual definition of stack. The gluing conditions of stacks arise from the conditions of essential surjectivity, fullness and faithfulness of the equivalence of Definition \ref{defstack2cat}. To see this, one has to use the Yoneda lemma and unpack the data of the category $\HomC{\Psm{\K\op}{ \Cat}}{R}{ F}$. In section \ref{sec2stacks}, we will use an analogous procedure one dimension higher.
\end{oss}
\begin{defne}\label{subcanbitopology}
	A bitopology $\tau$ on $\K$ is said \dfn{subcanonical} if all the representable prestacks are stacks with respect to $\tau$.
\end{defne}

We now aim at proving that, if the topology is subcanonical, the objects of $\K$ are some kind of colimits of the covering bisieves over them. This will generalize to the two-categorical context the well-known result that the objects of a subcanonical site are colimits of the sieves over them. 

The right kind of colimits that we need to consider in order to obtain the desired result are called \predfn{sigma-bicolimits}. They have been introduced for the first time by Gray in \cite{Graybook} in a more general setting and  then studied and applied by Descotte, Dubuc and Szyld in \cite{DescotteDubucSzyld}. The work of \cite{DescotteDubucSzyld} has been inspired by Street's paper \cite{Streetlimits} in which he presents the strict case. For our purposes we will need to consider an oplax version for colimits also used by Mesiti in \cite{Mesiti2setenriched}.

	\begin{defne}[\cite{DescotteDubucSzyld}]
	Let $W\:\A\op\to \Cat$ be a pseudofunctor with $\A$ small, and consider $2$-functors $M,N\:{\left(\Groth{W}\right)}\op\to \D$. A  \dfn{sigma natural transformation} $\gamma\: M \Rightarrow N$ is an oplax natural transformation such that the structure $2$-cell on every morphism $(f,\alpha)$ in $\Groth{W}$ with $\alpha$ an isomorphism is isomorphic.
\end{defne}

\begin{oss}
	In the language of \cite{DescotteDubucSzyld}, we are taking as class $\Sigma$ of morphisms in $\Groth{W}$ the class of all morphisms of type $(f,\alpha)$ where $\alpha$ is an isomorphism. Moreover, we are considering the case in which $W$ is a pseudofunctor rather then a strict 2-functor. This is discussed in Appendix A of \cite{DescotteDubucSzyld}.
\end{oss}

\begin{defne}[\cite{DescotteDubucSzyld}]\label{defsigmabicolim}
	Let $W\:\A\op\to \Cat$ be a pseudofunctor with $\A$ small, and let $F\:\Groth{W}\to \K$ be a pseudofunctor. The \dfn{sigma-bicolimit of $F$}, denoted as $\sigmabicolim{F}$, is (if it exists) an object $C\in \K$ together with a pseudonatural equivalence of categories
	$$\HomC{\K}{C}{D}\simeq \HomC{\msigma{{\left(\Grothdiag{W}\right)}\op}{\Cat}}{\Delta 1}{\HomC{\K}{F(-)}{D}}$$
	where the right-hand side is the category of sigma-natural transformations and modifications.
	\noindent When $\sigmabicolim{F}$ exists, the identity on $C$ provides a sigma natural transformation $\mu\:\Delta 1 \Rightarrow \HomC{\K}{F(-)}{C}$
	called the \dfn{universal sigma-bicocone}.
\end{defne}

\begin{oss}[universal properties of a sigma-bicolimit]
We can extract from the equivalence of categories required by Definition \ref{defsigmabicolim} the explicit universal properties satisfied by $K=\sigmabicolim{F}$. 

The essential surjectivity means that, given a sigma-bicocone $\rho$ of shape $\Groth{W}$ over an object $U\in \K$, there exists a morphism $r\: K \to U$ together with 2-cells 
\begin{cd}
	{F(A)}  \arrow[r,"{\mu_{(A,h)}}"]  \arrow[rr,"{\rho_{(A,h)}}"', bend right=45 ,""{name=B}]\& {K} \arrow[to=B,"{\theta_{(A,h)}}"{pos=0.2},twoiso] \arrow[r,"{r}"]\& {U} 
\end{cd}
for every $(A,h)\in \Groth{W}$, such that, given a morphism $(g,\alpha)\: (A,h) \to (B,t)$ in $\Groth{W}$, the following equality holds:
\begin{eqD*}
	\begin{cdN}
		|[alias=D]|{F((A,h))}  \arrow[d,"{F((g,\alpha))}"'] \arrow[r,"{\mu_{(A,h)}}"]  \& |[alias=G]|{K}  \arrow[r,"{r}"] \& {U} \\
		{F((B,t))}  \arrow[rru,"{\rho_{(B,t)}}"', bend right=30 ,""{name=F}] \arrow[from=G,to=F,"{\theta_{(B,t)}}"{pos=0.7},twoiso ,shorten <=1.5ex, shorten >= 1.5ex] \arrow[ru,"{\mu_{(B,t)}}"'{inner sep=0.2ex},""{name=S}] \arrow[from=D,to=S,"{\mu_{(g,\alpha)}}",Rightarrow ,shorten <=0.5ex, shorten >= 0.5ex] \& {} \& {}
	\end{cdN}
	\h[3]=\h[3]
	\begin{cdN}
		{F((A,h))}  \arrow[r,"{\mu_{(A,h)}}"]  \arrow[d,"{F((g,\alpha))}"'] \arrow[rr,"{\rho_{(A,h)}}"', bend right=45 ,""{name=B},""'{name=C}]\& {K} \arrow[to=B,"{\theta_{(A,h)}}"{pos=0.2},twoiso] \arrow[r,"{r}"]\& {U} \\
		{F((B,t))} \arrow[from=C,"{\rho_{(g,\alpha)}}",Rightarrow ,shorten <=1.5ex, shorten >= 1.5ex, shift left=1.5ex] \arrow[rru,"{\rho_{(B,t)}}"', bend right=40 ,""{name=F}] \& {} \& {}
	\end{cdN}
\end{eqD*}
The fully-faithfulness means that, given morphisms $r,s\: K \to U$ and for every $(A,h)\in \Groth{W}$ a 2-cell
\begin{cd}
	{F((A,h))} \arrow[r,"{\mu_{(A,h)}}"] \arrow[rd,"{\mu_{(A,h)}}"'] \& {K} \arrow[d,"{m_{(A,h)}}"{inner sep=0.4ex,pos=0.45},Rightarrow ,shorten <=1.5ex, shorten >= 1.5ex]  \arrow[r,"{r}"]\& {U} \\
	{} \& {K} \arrow[ru,"{s}"'] \& {}
\end{cd}
such that, given a morphism $(g,\alpha)\: (A,h) \to (B,t)$ in $\Groth{W}$, the following equality holds:
\begin{eqD*}
	\begin{cdN}
		{F((A,h))} \arrow[d,"{F((g,\alpha))}"'] \arrow[r,"{\mu_{(A,h)}}"] \arrow[rd,"{\mu_{(A,h)}}"{inner sep=0.1ex} ,""{name=E}] \& {K} \arrow[d,"{m_{(A,h)}}"{inner sep=0.4ex,pos=0.45},Rightarrow ,shorten <=1.5ex, shorten >= 1.5ex]  \arrow[r,"{r}"]\& {U} \\
		{F((B,t))} \arrow[from=E,"{\mu_{(g,\alpha)}}",Rightarrow ,shorten <=1.5ex, shorten >= 1.5ex] \arrow[r,"{\mu_{(B,h)}}"'] \& {K} \arrow[ru,"{s}"'] \& {}
	\end{cdN}
	\h[3]=\h[3]
	\begin{cdN}
		|[alias=Z]|{F((A,h))} \arrow[d,"{F((g,\alpha))}"'] \arrow[r,"{\mu_{(A,h)}}"] \& {K} \arrow[d,"{m_{(B,t)}}"{pos=0.4},Rightarrow ,shorten <=1.5ex, shorten >= 1.5ex]  \arrow[r,"{r}"]\& {U} \\
		{F((B,t))} \arrow[ru,"{\mu_{(B,t)}}"'{inner sep=0.2ex},""{name=V}]  \arrow[from=Z,to=V,"{\mu_{(g,\alpha)}}",Rightarrow ,shorten <=0.5ex, shorten >= 0.5ex] \arrow[r,"{\mu_{(B,h)}}"'] \& {K} \arrow[ru,"{s}"'] \& {}
	\end{cdN}
\end{eqD*}
there exists a unique 2-cell
$$\Gamma\: r \Rightarrow s$$
such that $\Gamma \star \mu_{(A,h)} =m_{(A,h)}$ for every $(A,h)\in \Groth{W}$.
\end{oss}

We now introduce a notion of \predfn{sigma-bicolim bisieve} that will be the two-dimensional analogue of the notion of colim sieve (called effectively epimorphic sieve in \cite{Elephant}).

\begin{defne}\label{sigmabicolimsieve}
	 A bisieve $S\: R \Rightarrow y(C)$ on $C\in \K$ is called a \dfn{sigma-bicolim bisieve} if $C=\sigmabicolim{F}$, where $F\: \Groth{R}\to \K$ is the 2-functor of  projection to the first component.
\end{defne}

In order to prove that every covering bisieve of a subcanonical bitopology is a sigma-bicolim bisieve, we will need the following result of conicalization of weighted colimits. This is the restriction of a result of \cite{Streetlimits} and it was proved in \cite{DescotteDubucSzyld}. We will present an oplax version of it.

\begin{prop} [\cite{Streetlimits}, \cite{DescotteDubucSzyld}]\label{conicaliz}
	Let $W\:\A\op\to \Cat$ be a pseudofunctor with $\A$ small and let $F\: \Groth{W} \to \K$ be a 2-functor. There is an equivalence of categories
	$$\HomC{\Psm{\A}{\Cat}}{W}{U}\simeq \HomC{\msigma{(\Grothdiag{W})\op}{\Cat}}{\Delta 1}{U\c F}$$
	that is pseudonatural in $U$.
\end{prop}

We now prove the main result of this section.

\begin{teor}\label{teorsigmabicolimbisieves}
	Let $\tau$ be a subcanonical bitopology and let $S\: R \Rightarrow y(C)$ be a covering bisieve over $C$. Then $S$ is a sigma-bicolim bisieve.
\end{teor}

\begin{proof}
	We need to prove that there exists a pseudonatural equivalence of categories 
	$$\HomC{\K}{C}{D}\simeq \HomC{\msigma{{\left(\Grothdiag{W}\right)}\op}{\Cat}}{\Delta 1}{\HomC{\K}{F(-)}{D}}.$$
	We construct this equivalence of categories as composite of three equivalences of categories. 
	
	\noindent Since the Yoneda embedding is fully faithful, there exists an isomorphism of categories
	\begin{eqD*}
		\HomC{\K}{C}{D}\iso \HomC{\Psm{\K\op}{\Cat}}{\HomC{\K}{-}{C}}{\HomC{\K}{-}{D}}
		\end{eqD*}
Moreover, since the bitopology is subcanonical, the definition of  stack  (Definition \ref{defstack2cat}) applied to the representable $\HomC{\K}{-}{C}$ yields an equivalence of categories
\begin{eqD*}
\HomC{\Psm{\K\op}{\Cat}}{\HomC{\K}{-}{C}}{ \HomC{\K}{-}{D}} \simeq \HomC{\Psm{\K\op}{ \Cat}}{R}{  \HomC{\K}{-}{D}}
\end{eqD*}
Finally, by Proposition \ref{conicaliz}, there exists an equivalence of categories
\begin{eqD*}
	\HomC{\Psm{\K\op}{ \Cat}}{R}{  \HomC{\K}{-}{D}}\simeq \HomC{\msigma{(\Grothdiag{W})\op}{\Cat}}{\Delta 1}{\HomC{\K}{F(-)}{D}}
	\end{eqD*}
Since all three equivalences are pseudonatural in $D$, the composite of them yields the desired pseudonatural equivalence of categories.
\end{proof}

We explicitly write the universal sigma-bicocone that presents an object as the sigma-bicolim of a covering bisieve over it.

\begin{oss}[universal sigma-bicocone of a sigma-bicolim bisieve] \label{sigmabicocone} 
The universal sigma-bicocone $\lambda$ of the sigma-bicolimit of Theorem \ref{teorsigmabicolimbisieves} is obtained applying to the identity of $C$ the chain of equivalences of categories of the proof. 

Given $(D, D\ar{f} C)\in \Groth{R}$, we have that the component $\lambda_f$ of $\lambda$ is  the morphisms $f$ itself. 
 Given a morphism of the form $(g, \id{})\: (E, \wt{f\c g}) \to (D,f)$ in $\Groth{R}$, the associated structure 2-cell is $\lambda_{(g,\id{})}= \sigma_{f,g}$ (the isomorphism between $\wt{f\c g}$ and $f\c g$ given by the pseudonaturality of the bisieve S). Moreover, given a morphism of the form $(\id{}, \alpha)$ in $\Groth{R}$, the structure 2-cell $\lambda_{(\id{}, \alpha)}$ is simply equal to $\alpha$. So, given a generic morphism $(g,\alpha)\: (E, h) \to (D,f)$ the structure 2-cell $\lambda_{(g,\alpha)}$ is given by the following pasting diagram
 \begin{cd}
 	{E} \arrow[rd,"{h}",""{name=A}] \arrow[d,"{}",equal] \&[10ex] {} \\
 	{E}  \arrow[from=A,"{\alpha}",Rightarrow ,shorten <=2.5ex, shorten >= 2.5ex] \arrow[r,"{\wt{f\c g}}"'] \arrow[d,"{g}",""{name=C}]\& {C} \\
 	{D} \arrow[ru,"{f}"',""{name=D}] \arrow[from=C,to=D,"{\sigma_{f,g}}"{inner sep=1.3ex},iso, ]\& {} 
 \end{cd}
 (This can be seen using the factorization of $(g,\alpha)$ as in Remark  \ref{splitmorGr}).
\end{oss}

We now prove a result that ensures that every time we consider a covering bisieve of the form $f\st S$ on $Z\in \K$ (see Remark \ref{bisievefstarS}) for a subcanonical bitopology, the object $D$ can be expressed as a sigma-bicolimit with diagram parametrized only by the morphisms of the form $f\st g$ with $g\in S$. This result will be very useful for us in \cite{Prin2bunquo2stacks} when considering the analogues of quotient prestacks one dimension higher.

\begin{prop}\label{coconofstar}
	Let $\K$ be a small 2-category with iso-comma objects and let $\tau$ be a subcanonical bitopology on it. Let then $S\: R \Rightarrow \K(-,Y)$ be a covering bisieve over $Y\in\K$ and $f\: X\to Y$ be a morphism in $\K$. Then $X$ is the sigma-bicolimit of the 2-functor 
	$$F\:  \Grothdiag{R} \ar{\opn{inc}} \laxslice{\K}{Y} \ar{f\st} \laxslice{\K}{X}  \ar{\opn{dom}} \K$$
	where $\opn{inc}=\Groth{S}\: \Groth{R} \to \Groth{\K(-,Y)}$ is the inclusion 2-functor  and $f\st$ is the 2-functor of iso-comma object along the morphism $f$. 
\end{prop}

\begin{proof}
	Let $(g,\alpha)\: (W, W\ar{h} Y) \to (T, T\ar{l} Y)$ be a morphism in $\Groth{R}$. The structure 2-cell of the universal sigma-bicocone that presents $Y$ as sigma-bicolimit of $S$ corresponding to the morphism $(g,\alpha)$ (see Remark \ref{sigmabicocone}) is sent by $F$ to the following identical 2-cell
	\begin{eqD}{2cells}
	\begin{cdN}
		{X\xp{Y}{h}W} \arrow[rd,"{f\st h}",""{name=A}] \arrow[d,"{F((\id{},\alpha))}"',] \&[10ex] {} \\
		{X\xp{Y}{\wt{l\c g}}W}   \arrow[r,"{f\st(\wt{l\c g})}"'{pos=0.4}] \arrow[d,"{F((g,\id{}))}"',""{name=C}]\& {X} \\
		{X\xp{Y}{l}T} \arrow[ru,"{l}"',""{name=D}] \& {} 
	\end{cdN}
\end{eqD}
	where $F((\id{},\alpha))$ is the morphism induced by the universal property of the iso-comma object ${X\xp{Y}{\wt{l\c g}}W}$ as in the following diagram
	\begin{eqD*}
		\commaunivvN{X\xp{Y}{h} W}  {f\st h}{h\st f}{F((\id{},\alpha))}{X\xp{Y}{\wt{l\c g}} W}{X}{W}{Y}{f\st (\wt{l\c g})} {(\wt{l\c g})\st f}{f}{\wt{l\c g}}{\lambda^{f,\wt{l\c g}}}
		\h[3]= \h[3]
		\begin{cdN}
			{X\xp{Y}{h}W} \arrow[r,"{{f\st h}}"]  \arrow[d,"{{h\st f}}"']\& {X} \arrow[d,"{f}"] \arrow[ld,"{{\lambda^{f,h}}}",Rightarrow ,shorten <=3.5ex, shorten >= 3.5ex, shift left=-2ex]\\
			{W} \arrow[r,"{h}", bend left=30 ,""'{name=A}] \arrow[r,"{\wt{l\c g}}"', bend right=30 ,""{name=B}] \arrow[from=A,to=B,"{\alpha}",Rightarrow ,shorten <=0.5ex, shorten >= 0.2ex]\& {Y}
		\end{cdN}
	\end{eqD*}
and $F((g,\id{}))$ is the morphism induced by the universal property of the iso-comma object $X\xp{Y}{l}T$ as in the following diagram
	\begin{eqD*}
	\commaunivvN{X\xp{Y}{\wt{l\c g}} W}  {f\st(\wt{l\c g})}{g\c ((\wt{l\c g})\st f)}{F((g,\id{}))}{X\xp{Y}{l} T}{X}{T}{Y}{f\st l} {l\st f}{f}{l}{\lambda^{f,l}}
	\h[3]= \h[3]
	\begin{cdN}
		{X\xp{Y}{\wt{l\c g}} W} \arrow[rr,"{{f\st(\wt{l\c g})}}"]  \arrow[d,"{{(\wt{l\c g})\st f}}"']\& {} \& {X} \arrow[d,"{f}"] \arrow[lld,"{{\lambda^{f,\wt{l\c g}}}}"'{inner sep=0.2ex},Rightarrow ,shorten <=6.5ex, shorten >= 6.5ex, shift left=-2ex]\\
		{W} \arrow[rr,"{\wt{l\c g}}", bend left=30 ,""'{name=A}] \arrow[r,"{g}"', ""{name=B}] \& {T} \arrow[from=A,"{\sigma_{l,g}}"{inner sep=0.2ex,pos=0.9},Rightarrow,shorten <=-0.5ex, shorten >= -0.5ex] \arrow[r,"{l}"'] \& {Y}
	\end{cdN}
\end{eqD*}
Notice that we can see the explicit action of $F$ on a generic morphism $(g,\alpha)$ putting together $F((g,\id{}))$ and $F((\id{},\alpha))$.

We want to prove that the 2-cells like the one in diagram \ref{2cells} form a universal sigma-bicocone for the sigma-bicolimit of $F$.
To do so, we use the fact that $X$ is the sigma-bicolimit of the covering bisieve $f\st S\: R_f \Rightarrow \K(-,X)$ with structure 2-cells as described in Remark \ref{sigmabicocone}. 

We consider a sigma-bicocone $\beta$ on an object $U\in \K$ that has structure 2-cell associated to $(g,\alpha)$ given by 
\begin{eqD}{beta}
\begin{cdN}
	{X\xp{Y}{h}W} \arrow[rd,"{\beta_ h}",""{name=A}] \arrow[d,"{\widehat{\alpha}}"',] \&[10ex] {} \\
	{X\xp{Y}{\wt{l\c g}}W} \arrow[from=A,"{\beta_{(\id,\alpha)}}", Rightarrow ,shorten <=2.5ex, shorten >= 2.5ex]  \arrow[r,"{\beta_{\wt{l\c g}}}"'{pos=0.5},""{name=S}] \arrow[d,"{\widehat{g}}"',""{name=C}]\& {U} \\
	{X\xp{Y}{l}T} \arrow[from=S,"{\beta_{(g,id)}}",twoiso ,shorten <=4.5ex, shorten >= 4.5ex, shift left=-2ex] \arrow[ru,"{\beta_l}"',""{name=D}] \& {} 
\end{cdN}
\end{eqD}
We want to use $\beta$ to construct a sigma-bicocone for the diagram that expresses $X$ as sigma-bicolimit of the covering bisieve $f\st S$, in order to use the universal property of the sigma-bicolimit.
Given a morphism $(g,\alpha)\: (Z,Z\ar{s} X) \to (Z', Z'\ar{t} X)$ in $\Groth{R_f}$, we consider the following 2-cell:
\begin{eqD}{A}
\begin{cdN}
	{Z} \arrow[d,"{}"',equal] \arrow[r,"{m_s}"] \& {X\xp{Y}{f\c s}Z} \arrow[d,"{F((\id{},f\star \alpha))}"] \arrow[r,"{}",aiso]\&	{X\xp{Y}{\overline{f\c s}}Z}  \arrow[rd,"{\beta_ {\overline{f\c s}}}",""{name=A}] \arrow[d,"{F((\id{},{\overline{\alpha}}))}"',] \&[10ex] {} \\
	{Z}  \arrow[r,"{m_{\wt{t\c g}}}"] \arrow[d,"{g}"']\& {X\xp{Y}{f\c (\wt{t\c g})}Z} \arrow[ld,"{\zeta_{(g,\alpha)}}",twoiso ,shorten <=4.5ex, shorten >= 4.5ex] \arrow[d,"{\widehat{g}}"]  \arrow[r,"{}",aiso]\&	{X\xp{Y}{\overline{f\c (\wt{t\c g})}}Z}  \arrow[from=A,"{\beta_{(\id{},\overline{\alpha})}}", Rightarrow ,shorten <=2.5ex, shorten >= 2.5ex]  \arrow[r,"{\beta_{\overline{f \c (\wt{t\c g})}}}"'{pos=0.5},""{name=S}] \arrow[d,"{F((g,\psi))}"',""{name=C}]\& {U} \\
	{Z'} \arrow[r,"{m_t}"']\& {X\xp{Y}{f\c t}Z'} \arrow[r,"{}",aiso] \& {X\xp{Y}{\overline{f\c t}}T} \arrow[from=S,"{\beta_{(g,\psi)}}",twoiso ,shorten <=5.5ex, shorten >= 5.5ex, shift left=-2ex] \arrow[ru,"{\beta_{\overline{f\c t}}}"',""{name=D}] \& {} 
\end{cdN}
\end{eqD}
where:
\begin{itemize}
	\item[-] $m_s, m_{\wt{t\c g}}$ and $m_t$ are the morphisms given by the factorization of $s,\wt{t\c g}$ and $t$ as in Remark \ref{bisievefstarS}.
	\item[-] $\overline{\alpha}$ is the following 2-cell
	\begin{cd}
		{Z} \arrow[d,"{\alpha}"{pos=0.35},Rightarrow ,shorten <=1.5ex, shorten >= 2.5ex, shift left=3ex] \arrow[r,"{s}"]  \arrow[d,"{}", equal] \arrow[rr,"{\overline{f\c s}}", bend left=40,""'{name=A}]\& {X} \arrow[d,"{}"{pos=0.35},twoiso ,shorten <=2.5ex, shorten >= 2.5ex, shift left=2ex] \arrow[from=A,"{}",twoiso ] \arrow[r,"{f}"] \& {Y} \\
		{Z} \arrow[ru,"{\wt{t\c g}}"'] \arrow[rru,"{\overline{f\c (\wt{t\c g})}}"', bend right=30,""{name=G}]\& {} \& {} 
	\end{cd}
\item[-] the morphism $(g,\psi)$ corresponds to the following equal 2-cells in the lax slice $\laxslice{\K}{Y}$ 
\begin{eqD*}
	\begin{cdN}
		{Z} \arrow[d,"{\sigma_{t,g}}"{pos=0.35},twoiso ,shorten <=1.5ex, shorten >= 2.5ex, shift left=3ex] \arrow[r,"{\wt{t\c g}}"]  \arrow[d,"{g}"', ] \arrow[rr,"{\overline{f\c (\wt{t\c g}})}", bend left=40,""'{name=A}]\& {X} \arrow[d,"{}"{pos=0.35},twoiso ,shorten <=2.5ex, shorten >= 2.5ex, shift left=2ex] \arrow[from=A,"{}",twoiso ] \arrow[r,"{f}"] \& {Y} \\
		{Z'} \arrow[ru,"{\wt{t\c g}}"'] \arrow[rru,"{\overline{f\c t}}"', bend right=30,""{name=G}]\& {} \& {} 
	\end{cdN}
	\h[3]=\h[3]
	\begin{cdN}
		|[alias=G]|{Z}  \arrow[d,"{g}"'] \arrow[r,"{\overline{f\c (\wt{t\c g}})}", bend left=40,""'{name=A}] \arrow[r,"{\wt{\overline{f\c t}\c g}}"',""{name=B}] \&[8ex] {Y} \\
		{Z'} \arrow[ru,"{\overline{f\c t}}"',""{name=C}, bend right=30]  \arrow[from=A, to=B,"{}",twoiso] \arrow[to=C,from=G,"{\sigma_{\overline{f\c t},g}}"'{inner sep=1ex},twoiso ,shorten <=1.5ex, shorten >= 1.5ex]\& {}
	\end{cdN}
\end{eqD*}
\item[-] the morphism $\widehat{g}$ is induced by the universal property of the iso-comma object $X\xp{Y}{f\c t}Z'$ using the isomorphic 2-cell
\begin{eqD*}
	\begin{cdN}
		{Z} \arrow[d,"{\sigma_{t,g}}"{pos=0.35},twoiso ,shorten <=1.5ex, shorten >= 2.5ex, shift left=3ex] \arrow[r,"{\wt{t\c g}}"]  \arrow[d,"{g}"', ] \arrow[rr,"{\overline{f\c (\wt{t\c g}})}", bend left=40,""'{name=A}]\& {X} \arrow[from=A,"{}",twoiso ] \arrow[r,"{f}"] \& {Y} \\
		{Z'} \arrow[ru,"{\wt{t\c g}}"'] \& {} \& {} 
	\end{cdN}
\end{eqD*}
\item[-]the 2-cell $\zeta_{(g,\alpha)}$ is induced using the two-dimensional universal property of the iso-comma object $X\xp{Y}{\overline{f\c t}} T$ (starting from the identity 2-cell and a 2-cell essentially given by $\sigma_{t,g}$ ).
\end{itemize}
One can then prove that these 2-cells define a sigma-bicocone over $U$ for the diagram that expresses X as a sigma-bicolimit of the covering bisieve $f\st S$.
So this sigma-bicocone factors through the universal one. This means that there exist a unique morphism $u\: X\to U$ and isomorphic 2-cells $\iota$ and $\kappa$ such that the diagram
\begin{eqD}{B}
	\begin{cdN}
		{Z} \arrow[rrd,"{s}",""{name=A}] \arrow[d,"{}"',equal] \arrow[r,"{m_s}"] \& {X\xp{Y}{f\c s}Z}  \arrow[r,"{}",aiso]\&	{X\xp{Y}{\overline{f\c s}}Z}  \arrow[rd,"{\beta_{\overline{f\c s}}}"] \arrow[d,"{\iota_s}",twoiso ,shorten <=1.5ex, shorten >= 1.5ex]\&[10ex] {} \\
		{Z} \arrow[r,"{\sigma_{l,g}}"{inner sep=1.2ex},iso,shift left=-4.5ex] \arrow[from=A,"{\alpha}",Rightarrow ,shorten <=4.5ex, shorten >= 4.5ex] \arrow[rr,"{\wt{t\c g}}"]  \arrow[d,"{g}"']\& {}    \&	{X} \arrow[r,"{u}"] \arrow[d,"{\iota_t^{-1}}",twoiso ,shorten <=1.5ex, shorten >= 1.5ex]  \& {U} \\
		{Z'} \arrow[rru,"{t}"'] \arrow[r,"{m_t}"']\& {X\xp{Y}{f\c t}Z'} \arrow[r,"{}",aiso] \& {X\xp{Y}{\overline{f\c t}}T}  \arrow[ru,"{\beta_{\overline{f\c t}}}"',""{name=D}] \& {} 
	\end{cdN}
\end{eqD}
is equal to diagram (\ref{A}).
It is then straightforward to prove that the morphism $u\: X\to U$ together with the isomorphic 2-cell
\begin{cd}
	{X\xp{Y}{h}W} \arrow[ddd,"{f\st h}"', bend right=70] \arrow[rrd,"{\beta_h}", bend left=20]  \arrow[d,"m_{f\st h}"'] \arrow[rd,"{}",equal, bend left=20]\& {} \& {}\\
	{(X\xp{Y}{h}W)\xp{Y}{f\c f\st h}} \arrow[d,"{}",aiso]\& {X\xp{Y}{h}W} \arrow[r,"{\beta_h}"]  \& {U}\\
	{(X\xp{Y}{h}W)\xp{Y}{\overline{f\c f\st h}}X} \arrow[d,"{\iota_{f\st h}}",twoiso ,shorten <=3.5ex, shorten >= 3.5ex] \arrow[rru,"{\beta_{\overline{f \c f\st h}}}"', bend right=70] \arrow[r,"{}",iso, shift left=-5ex] \arrow[r,"{}",aiso] \& {(X\xp{Y}{h}W)\xp{Y}{\wt{h\c h\st f}}X } \arrow[ru,"{\beta_{\wt{h\c h\st f}}}"'] \arrow[u,"{\widehat{h\st f}}"]  \arrow[ru,"{}",iso, shift left=2.6ex]\& {}\\[5ex]
	{X} \arrow[rruu,"{u}"', bend right=70] \& {} \& {}
\end{cd}
and the analogous one associated to the morphism $l$, gives the desired factorization of the 2-cell of diagram (\ref{beta}). Notice that here the key idea is to apply the universality of the sigma-bicocone that presents $X$ as sigma-bicolimit of $f\st S$ to the structure 2-cell that corresponds to the morphism $(a,\sigma_{f\st l, a})$ with $a=F((g,\id{}))\c F((\id{},\alpha))$. 

Let now $u,v\: X \to U$ and consider a collection of  compatible 2-cells of the form
\begin{cd}
	{} \&[5ex] {X} \arrow[dd,"{n_h}", Rightarrow ,shorten <=2.5ex, shorten >= 2.5ex] \arrow[rd,"{u}", bend left=20] \&[5ex] {} \\[-4ex]
	{X\xp{Y}{h} W} \arrow[ru,"{f\st h}", bend left=20] \arrow[rd,"{f\st h}"', bend right=20] \& {} \& {U} \\[-4ex]
	{} \& {X}  \arrow[ru,"{v}"', bend right=20]\& {}
\end{cd}
for every $(W, W \ar{h} Y)\in \Groth{R}$. 
We can use this collection to construct another collection of compatible 2-cells of the form 
\begin{eqD}{V}
\begin{cdN}
	{} \& {} \&{}  \&[5ex] {X} \arrow[dd,"{n_{\overline{f\c s}}}", Rightarrow ,shorten <=2.5ex, shorten >= 2.5ex] \arrow[rd,"{u}", bend left=20] \&[5ex] {} \\[-4ex]
	{Z} \arrow[r,"{m_s}"] \& {X\xp{Y}{{f \c s}} Z} \arrow[r,"{}",aiso]\&	{X\xp{Y}{\overline{f \c s}} Z} \arrow[ru,"{f\st (\overline{f\c s})}", bend left=20] \arrow[rd,"{f\st (\overline{f\c s})}"', bend right=20] \& {} \& {U} \\[-4ex]
	{} \& 	{} \& {} \& {X}  \arrow[ru,"{v}"', bend right=20]\& {}
\end{cdN}
\end{eqD}
indexed by objects $(Z,Z\ar{s} X)\in \Groth{R_f}$.
Thanks to the universality of the universal sigma-bicocone that expresses $X$ as sigma-bicolimit of $f\st S$, we induce a 2-cell $\Gamma\: u \Rightarrow v$ such that for every $(Z,Z\ar{s} X)\in \Groth{R_f}$, the whiskering of $\Gamma$ with the morphism 
\begin{cd}
	{Z} \arrow[r,"{m_s}"] \& {X\xp{Y}{{f \c s}} Z} \arrow[r,"{}",aiso]\&	{X\xp{Y}{\overline{f \c s}} Z} \arrow[r,"{f\st (\overline{f\c s})}"] \& {X}
\end{cd}
is equal to diagram (\ref{V}). Applying this equality to $s=f\st h$, it is straightforward to prove that $\Gamma \star f\st h=n_h$. This concludes the proof that $X=\sigmabicolim{F}$.
\end{proof}

\begin{oss}
	The result of Proposition \ref{coconofstar} allows one to present $Z$ as the sigma-bicolimit of a diagram indexed on the morphisms of the bisieve $S$ avoiding to consider the morphisms in $f\st S$ that are not iso-commas of a morphism of $S$ along $f$. This offers substantial benefits in practice.
\end{oss}
\section{2-stacks} \label{sec2stacks}

Let $\K$ be a small (strict) 2-category with bi-iso-comma objects and let $\tau$ be a Grothendieck bitopology on it. We introduce a notion of 2-stack that is suitable for a trihomomorphism from $\K\op$ to $\Bicat$. This notion is a natural generalization of Street's notion of stack (Definition \ref{defstack2cat}) one dimension higher. The main motivation for the author to introduce this new notion has been the lack of a notion of higher dimensional stack suitable for the higher dimensional analogue of the quotient stack studied in \cite{Prin2bunquo2stacks}.

\begin{defne}\label{2stack}
	A trihomomorphism $F:\K\op \to \Bicat$ is a \dfn{2-stack} if for every object $C\in \K$ and every bisieve $S\: R\Rightarrow y(C)$ in $\tau(C)$ the pseudofunctor
		$$-\c S\: \Tricat(\K\op, \Bicat)(y(C), F) \longrightarrow \Tricat(\K\op, \Bicat)(R, F)$$
		is a biequivalence. 
\end{defne}

\begin{oss}
	Similar higher dimensional generalizations of Street's notion of stack have been introduced in the literature in different contexts. The notion of $\infty$-stack studied by Lurie in \cite{Lurie} applies the same idea to a functor with domain an \linebreak $(\infty,1)$-category that takes values in $(\infty,1)$-categories. When truncated to dimension 3 the notion of $\infty$-stack yields a notion of $(2,1)$-stack for a functor with domain a $(2,1)$-category that takes values in $(2,1)$-categories. 
	
	\noindent Moreover, Campbell in their PhD thesis \cite{Campbellthesis} proposed a notion of $2$-stack for a trihomomorphism from a one dimensional category to $\Bicat$ that generalizes Street's definition in the same fashion as our notion.
	
	\noindent None of these two notions present in the literature is suitable for the higher dimensional generalization of quotient prestacks presented in \cite{Prin2bunquo2stacks}.
\end{oss}

Since Definition \ref{2stack} is quite abstract and very difficult to use in practice, we aim at a characterization of the notion of 2-stack given by explicit conditions that can be checked more easily to prove that a certain trihomomorphism is a 2-stack.

\begin{oss}
	Notice that it would be hard to choose the right gluing conditions to ask for a 2-stack trying to generalize the classical definition of stack one dimension higher. This is because one would not know what are the right coherences to ask, especially for the gluing condition on objects. For this reason, we introduced the notion of 2-stack as a generalization of Street's definition and we now use our definition in order to find the explicit gluing conditions. This will automatically give us the right coherence conditions to require. 
\end{oss}

As first step towards a characterization given by explicit conditions, we give the following preliminary characterization of the notion of 2-stack.

\begin{prop}\label{charYoneda}
	Let $(\K,\tau)$ be a bisite. A trihomomorphism $F:\K\op \to \Bicat$ is a \dfn{2-stack} if and only if for every object $\cC\in \K$ and every bisieve $S\: R\Rightarrow \K(-,\cC)$ in $\tau(\cC)$ the pseudofunctor
	$$(-\c S)\c \Gamma\: F(C) \longrightarrow \Tricat(\K\op, \Bicat)(R, F),$$
	where $\Gamma$ is the biequivalence given by the tricategorical Yoneda lemma (Theorem \ref{Yonedatricat}), is a biequivalence. 
\end{prop}

\begin{proof}
	Straightforward using the two out of three property of biequivalences of bicategories.
\end{proof}

This characterization will allow us to extract from the abstract definition of 2-stack the explicit gluing conditions. We will need to recall the following basic characterization of biequivalence between bicategories. This result is considered folklore, but it appeared for the first time in \cite{GabberRamero}.

\begin{prop}[characterization of biequivalences,\cite{GabberRamero}] \label{charbieq}
	Assuming the axiom of choice, a pseudofunctor $G\: \C \to \D$ between bicategories is a biequivalence if and only if the following conditions hold:
		\begin{itemize}
		\item[(1)]  $G$ surjective on equivalence classes of objects;
		\item [(2)]  $G$ is essentially surjective on morphisms;
		\item [(3)]  $G$ is fully-faithful on 2-cells.
	\end{itemize}
	\end{prop}

The key idea behind the characterization that we will achieve is that the three conditions of Proposition \ref{charbieq} correspond to effectiveness conditions of the appropriate kind of data of descent on objects, morphisms and 2-cells. These effectiveness conditions are given by the biequivalence of Proposition \ref{charYoneda}. We will present the explicit conditions on objects, morphisms and 2-cells separately and then we will put them together to obtain the characterization of 2-stacks.

We start by considering 2-cells. Since stacks, that are two-dimensional analogues of sheaves, are sheaves on morphisms, we expect 2-stacks to be sheaves on 2-cells. And this is exactly the explicit condition that we find unpacking what it means for the pseudofunctors as in Proposition \ref{charYoneda} to be fully-faithful on 2-cells. 

We give the following natural definition.
\begin{defne}
	Let $F\: \K\op \to \Bicat$ be a trihomomorphism. Let then $S$ be a bisieve on $C\in \K$ and let $a,b\: X\to Y$ be two morphisms in $F(C)$ .  A \dfn{matching family  for $S$ of 2-cells of $F$} is a function which assigns to each $f\: D \to C$ in $S$ a 2-cell $w_f \: f\st a \Rightarrow f\st b$ in $F(D)$, in such a way that the following conditions are satisfied.
	For every $g\: E \to D$ in $\C$, the following is an equality of pasting diagrams
		\begin{eqD*}
			\begin{cdsN}{7}{7}
				{g\st f \st X} \arrow[d,"{}", simeq] \arrow[r,"{g\st f\st a}"] \& {{g\st f \st Y}} \arrow[ld, Rightarrow,twoiso,"{}", shorten <=2.5ex, shorten >=3ex,pos=0.4] \arrow[d,"{}", simeq]\\
				{(\wt{f\c g})\st X} \arrow[r,"{(\wt{f\c g})\st a}", ""'{name=A}] \arrow[r,"{(\wt{f\c g})\st b}"', bend right=65,""{name=B}] \arrow[Rightarrow,from=A, to=B,"{w_{\widetilde{f\c g}}}"{}, shorten >=-0.6ex, shorten <=-0.3ex]\& {(\wt{f\c g})\st Y.}
			\end{cdsN}
			= 
			\begin{cdsN}{7}{7}
				{g\st f\st X} \arrow[d,"{}", simeq] \arrow[r,"{g\st f\st a}",""'{name=D}, bend left=65]  \arrow[r,"{g\st f\st b}"',""'{name=E}] \arrow[Rightarrow,from=D, to=E,"{g\st w_{f}}"{}, shorten >=0.3ex]\& {g\st f\st Y} \arrow[ld, Rightarrow,twoiso,"{}", shorten <=2.5ex, shorten >=3ex,pos=0.4] \arrow[d,"{}", simeq] \\
				{(\wt{f\c g})\st X} \arrow[r,"{(\wt{f\c g})\st b}"']\& {(\wt{f\c g})\st Y}
		\end{cdsN}
	\end{eqD*}
For every 2-cell $\Gamma\: f \Rightarrow f'$, the following is an equality of pasting diagrams
\begin{eqD}{match2cells}
	\begin{cdsN}{7}{7}
		{f\st X} \arrow[d,"{F(\Gamma)_X}"',] \arrow[r,"{f\st a}",""'{name=D}, bend left=65]  \arrow[r,"{f\st b}"',""'{name=E}] \arrow[Rightarrow,from=D, to=E,"{w_{f}}"{}, shorten >=0.3ex]\& {\st f\st Y} \arrow[ld, Rightarrow,twoiso,"{}", shorten <=2.5ex, shorten >=3ex,pos=0.4] \arrow[d,"{F(\Gamma)_Y}"] \\
		{{f'}\st X} \arrow[r,"{{f'}\st b}"']\& {{f'}\st Y}
	\end{cdsN}
	\h[3]=\h[5] 
	\begin{cdsN}{7}{7}
		{f \st X} \arrow[d,"{F(\Gamma)_X}"',] \arrow[r,"{f\st a}"] \& {{f \st Y}} \arrow[ld, Rightarrow,twoiso,"{}", shorten <=2.5ex, shorten >=3ex,pos=0.4]  \arrow[d,"{F(\Gamma)_Y}"]\\
		{{f'}\st X} \arrow[r,"{{f'}\st a}", ""'{name=A}] \arrow[r,"{{f'}\st b}"', bend right=65,""{name=B}] \arrow[Rightarrow,from=A, to=B,"{w_{f'}}"{}, shorten >=-0.6ex, shorten <=-0.3ex]\& {{f'}\st Y}
	\end{cdsN}
\end{eqD}

An \dfn{amalgamation} of such a matching family is a 2-cell $w\: a\Rightarrow b$ in $F(C)$ such that, for every $f\: D\to C$ in $S$,  we have $f\st w=w_f.$
\end{defne}

\begin{oss}
	The condition of diagram \ref{match2cells} is given by the fact that the bisieve is two-dimensional and so it contains 2-cells. For this reason, it does not have a counterpart in the usual one-dimensional notion of matching family.
\end{oss}

\begin{prop}\label{sheaveson2cells}
	Let $C\in \K$ and let $S\: R \Rightarrow \HomC{\K}{-}{C}$ be a bisieve on it. The pseudofunctor
	$$(-\c S)\c \Gamma\: F(C) \longrightarrow \Tricat(\K\op, \Bicat)(S, F),$$
	is fully-faithful on 2-cells if and only if every matching family for $S$ of  2-cells of $F$ has a unique amalgamation.
	\end{prop}

\begin{proof}
	Let $a,b\: X\to Y$ be morphisms in $F(C)$. The images of $a$ and $b$ under the pseudofunctor $(-\c S)\c \Gamma$ are respectively $m^a\star S\: \sigma^X \c S \aM{} \sigma^Y \c S$ and \linebreak $m^b\star S\: \sigma^X \c S \aM{} \sigma^Y \c S$, where the trimodifications $m^a\: \sigma^X \Rightarrow \sigma^{Y}$ and $m^b\: \sigma^X \Rightarrow \sigma^{Y}$ are the ones given by the biequivalence of the Tricategorical Yoneda lemma (see Remark \ref{assYonedatricat}). 
	
	Let then $q\: m^a\star S \aP{} m^b\star S$ be a perturbation. This is a perturbation of the kind described in Example \ref{ourpert}. Then $q$ is given by a family of modifications $q_D\:(m^a\star S)_D \aM{} (m^b\star S)_D$. Notice that the component of $(m^a\star S)_D$ on a morphism $D\ar{f}C$ in $R(D)$ is equal to $((m^a)_D)_f$ and so it is $f\st a$ and the same holds for $m^b$. So the component of $q_D$ on $D\ar{f}C$ in $S$ is a 2-cell 
	$$(q_D)_f\: f\st a \Rightarrow f\st b$$ 
	Moreover, the axiom of perturbation for $q$ for the morphism $E\ar{g}D$ computed on  $D\ar{f} C\in R(D)$ gives the following equality of 2-cells:
	\begin{eqD*}
		\begin{cdsN}{7}{7}
			{g\st f \st X} \arrow[d,"{}", simeq] \arrow[r,"{g\st f\st a}"] \& {{g\st f \st X}} \arrow[ld, Rightarrow,twoiso,"{((m^a)_g)_f}"{inner sep=0.2ex}, shorten <=2.5ex, shorten >=3ex,pos=0.4] \arrow[d,"{}", simeq]\\
			{(\wt{f\c g})\st X} \arrow[r,"{(\wt{f\c g})\st a}", ""'{name=A}] \arrow[r,"{(\wt{f\c g})\st b}"', bend right=65,""{name=B}] \arrow[Rightarrow,from=A, to=B,"{(q_D)_{\widetilde{f\c g}}}"{}, shorten >=-0.6ex, shorten <=-0.3ex]\& {(\wt{f\c g})\st Y.}
		\end{cdsN}
		= 
		\begin{cdsN}{7}{7}
			{g\st f\st X} \arrow[d,"{}", simeq] \arrow[r,"{g\st f\st a}",""'{name=D}, bend left=65]  \arrow[r,"{g\st f\st b}"',""'{name=E}] \arrow[Rightarrow,from=D, to=E,"{g\st (q_D)_{f}}"{}, shorten >=0.3ex]\& {g\st f\st Y} \arrow[ld, Rightarrow,twoiso,"{((m^b)_g)_f}"{inner sep=0.2ex}, shorten <=2.5ex, shorten >=3ex,pos=0.4] \arrow[d,"{}", simeq] \\
			{(\wt{f\c g})\st X} \arrow[r,"{(\wt{f\c g})\st b}"']\& {(\wt{f\c g})\st Y}
		\end{cdsN}
	\end{eqD*}
	where the 2-cells $((m^a)_g)_f$ and $((m^b)_g)_f$ coincide with the structure 2-cells given by the fact that $F$ is a trihomomorphism by the Tricategorical Yoneda lemma. So the data of the perturbation $q$ are exactly the same data of a matching family for $S$ of 2-cells of $F$, where the assignment $w_f$ is the 2-cell $(q_D)_f$, the compatibility condition with respect to composition is given by the axiom of perturbation and the compatibility condition with respect to 2-cells is given by the fact that $(q_D)$ is a modification for every $D\in \K$. 
	
	It suffice, then, to prove that the perturbation $q$ is the image under $(-\c S)\c \Gamma$ of a unique 2-cell $w\: a \Rightarrow b$ in $F(C)$ exactly when the corresponding matching family of 2-cells has a unique amalgamation (that is the same 2-cell $w\: a \Rightarrow b$).
	
	If $q$ is the image under $(-\c S)\c \Gamma$ of $w$, then there exists a 2-cell $w\: a \Rightarrow b$ such that  for every $D\in K$ and every $D\ar{f}C\in R(D)$ we have $(q_D)_f=f\st w$. This happens exactly when the corresponding matching family has an amalgamation. And the uniqueness of such a 2-cell whose image is $q$ corresponds to the uniqueness of the amalgamation. 
\end{proof}

We now consider the gluing condition on morphisms. We expect to obtain a condition similar to the effectiveness of descent data that we have on objects for usual stacks and this is what happens.

We give the following definition.

\begin{defne}\label{descentdatumonmorphisms}
	Let $F\: \K\op \to \Bicat$ be a trihomomorphism and let $S$ be a bisieve on $C\in \K$. 
	
	A \dfn{descent datum for $S$ of morphisms of $F$} is an assignment for every morphism $D\ar{f} C$ in $S$ of a morphism  $w_f\: f\st X \ar{} f\st Y $ in $F(D)$, for every pair of composable morphisms $D'\ar{g}D\ar{f}C$
	with $f\in S$, of an invertible 2-cell
	\begin{eqD*}
			\begin{cdN}
				{g\st f\st X}  \arrow[r,"{g\st w_f}"] \arrow[d,"{}",simeq] \&[5ex] {g\st f\st Y} \arrow[d,"{}",simeq] \arrow[ldd,"{\phi^{f,g}}",twoiso, shorten <=3.5ex, shorten >=3.5ex]\\[-5ex]
				{(f\c g)\st X} \arrow[d,"{}",simeq]\&[5ex]{(f\c g)\st Y} \arrow[d,"{}",simeq] \\[-5ex]
				{(\widetilde{f\c g})\st X} \arrow[r,"{w_{\widetilde{f\c g}}}"'] \&[5ex]{(\widetilde{f\c g})\st Y} 
		\end{cdN}
	\end{eqD*}
and for every 2-cell $\gamma\: f\Rightarrow f'$ with $f,f'\: D\to C$ in $S$ of an invertible 2-cell in $F(D)$
\begin{eqD*}\scalebox{0.9}{
		\begin{cdN}
			{f\st X} \arrow[r,"{w_f}"] \arrow[d,"{F(\gamma) _X}"']\& {f \st Y} \arrow[d,"{{F(\gamma) _Y}}"] \arrow[ld,"{\eta_{\gamma}}",twoiso, shorten <= 2ex, shorten >= 2ex]\\
			{{f'}\st X} \arrow[r,"{w_{f'}}"'] \& {{f'}\st Y}
	\end{cdN}}
\end{eqD*}
These data need to satisfy the following compatibility conditions.

Given $f\: D\to C\in S$, the isomorphic 2-cell $\phi^{f,\id{D}}$ coincides with the isomorphic 2-cell 
\begin{cd}
	{f\st X} \arrow[r,"{}",aeq] \arrow[r,"{}",iso, shift left=-5ex]\arrow[d,"{w_f}"']\& {\id{D}\st (f\st X)}\arrow[d,"{\id{D}\st w_f}"]\\
	{f\st Y} \arrow[r,"{}",aeq]\& {\id{D}\st (f\st X)}
\end{cd}
given by the fact that $F$ is a trihomomorphism.

Given morphisms $D''\ar{h} D' \ar{g}  D\ar{f} C$
with $f\in S$ the given 2-cells need to satisfy the following equality, called \dfn{cocycle condition on morphisms}:
\begin{samepage}
\begin{eqD*}\scalebox{0.8}{
		\begin{cdsN}{7}{3}
			{(\wt{\wt{f\c g}\c h})\st X} \arrow[rr,"{}",aiso] \arrow[d,"{}",aiso] \arrow[rd,"{w_{\wt{\wt{f\c g}\c h}}}"]\& {} \& {(\wt{f\c g}\c h)\st X} \arrow[ld,"{\phi^{\wt{f\c g},h}}", Rightarrow ,shorten <=2.5ex, shorten >= 2.5ex, shift left=2ex] \arrow[r,"{}",aeq] \& {h\st(\wt{f\c g})\st X} \arrow[d,"{h\st w_{\wt{f\c g}}}"] \arrow[r,"{}",aiso] \& {h\st({f\c g})\st X}  \arrow[r,"{}",aeq] \& {h\st g\st f\st X} \arrow[dd,"{h\st g\st w_f}"] \arrow[lld,"{h\st \phi^{f,g}}", Rightarrow ,shorten <=8.5ex, shorten >= 8.5ex, shift left=4ex]\\
			{(\wt{{f\c g}\c h})\st X} \arrow[r,"{}",iso , color=blue, shift left=-2ex] \arrow[rdd,"{w_{\wt{f\c g\c h}}}"']\& {(\wt{\wt{f\c g}\c h})\st Y} \arrow[dd,"{}", aiso]\arrow[rd,"{}", aiso] \& {} \& {h\st({f\c g})\st Y} \arrow[rd,"{}",aiso,""{name=A}]  \& {} \& {}\\
			{} \& {} \& {(\wt{f\c g}\c h)\st Y} \arrow[rr,"{}",iso, color=blue]\arrow[ru,"{}",aeq] \arrow[rd,"{}", aeq] \& {} \& {h\st (f\c g)\st Y} \arrow[r,"{}",aeq]\& {h\st g\st f\st Y} \arrow[d,"{}", aeq]\\
			{} \& {{(\wt{f\c g}\c h)\st Y}} \arrow[rr,"{}",aeq] \& {} \& {(f\c g\c h)\st Y} \arrow[ru,"{}", aeq] \arrow[rr,"{}", aeq]\& {} \& {(g\c h)\st f\st Y}
	\end{cdsN}}
\end{eqD*}
\begin{eqD}{cocyclemor}
	\vspace*{-1.8ex}\hspace*{5.7cm}
	\begin{cdsN}{3}{3}
		\hphantom{.}\arrow[d,equal]\\
		\hphantom{.}
	\end{cdsN}\hspace{5.4cm}
\end{eqD}
\begin{eqD*}\scalebox{0.8}{
		\begin{cdsN}{9}{7}
			{(\wt{\wt{f\c g}\c h})\st X} \arrow[r,"{}",aiso] \arrow[d,"{}",aiso] \& {(\wt{f\c g}\c h)\st X} \arrow[rr,"{}",iso ,shift left=-5ex, color=blue] \arrow[rd,"{}",aiso] \arrow[r,"{}",aeq] \& {h\st(\wt{f\c g})\st X}  \arrow[r,"{}",aiso] \& {h\st({f\c g})\st X}  \arrow[r,"{}",aeq] \& {h\st g\st f\st X} \arrow[d,"{h\st g\st w_f}"] \\
			{(\wt{f\c g\c h})\st X} \arrow[rr,"{}",aeq] \arrow[rd,"{w_{\wt{f\c g\c h}}}"']\& {} \& {({f\c g\c h})\st X} \arrow[ru,"{}",aeq] \arrow[r,"{}",aiso]\& {(g\c h)\st f \st X} \arrow[lld,"{\phi^{f,g\c h}}", Rightarrow ,shorten <=10.5ex, shorten >= 10.5ex]\arrow[ru,"{}",aeq] \arrow[rd,"{(g\c h)\st w_f}"'] \arrow[r,"{}",iso, color=blue] \& {h\st g\st f\st Y} \arrow[d,"{}", aiso] \\
			{} \& {(\wt{f\c g\c h})\st Y} \arrow[r,"{}",aiso]  \& {({f\c g\c h})\st Y} \arrow[rr,"{}", aeq] \& {} \& {(g\c h)\st f\st Y} 
	\end{cdsN}}
\end{eqD*}
\end{samepage}
where the unnamed arrows are the canonical isomorphism or equivalences and the isomorphic 2-cells in blue are given by the interchange law.

Given $D\ar{f} C\in S$, the following equality of 2-cells holds:
\begin{eqD*}
	\csq*[l][7][7][\eta_{id_f}][2.5]{f\st X}{f\st Y}{f\st X}{f\st Y}{w_f}{F(\id{f})_X}{F(\id{f})_Y}{w_f}
	\h[3]= \h[3]
	\begin{cdN}
		{f\st X} \arrow[r,"{w_f}"] \arrow[d,"{F(\id{f})_X}"', ""{name=A},bend right=60]\arrow[d,"{}",equal,""'{name=B}] \arrow[from=A,to=B,"{}",iso]\& {f\st Y} \arrow[d,"{}",equal,""{name=C}] \arrow[d,"{F(\id{f})_Y}",""'{name=D}, bend left=60] \arrow[from=C, to=D,"{}", iso]\\
		{f\st X} \arrow[r,"{w_f}"]\& {f\st Y}
	\end{cdN}
\end{eqD*}

Given morphisms $f,f',f''\: D\to C$ in $S$ and 2-cells $\gamma\: f\Rightarrow f'$ and $\delta\: f'\Rightarrow f''$, the following equality of 2-cells holds:
	\begin{eqD*}
		\csq*[l][7][7][\eta_{\delta \c \gamma}][2.5]{f\st X}{f\st Y}{{f''}\st X}{{f''}\st Y}{w_f}{F(\delta\c \gamma)_X}{F(\delta\c \gamma)_Y} {w_{f''}}
		\h[3] = \h[3]
		\begin{cdN}
			{f\st X}\arrow[dd,"{F(\delta \c \gamma)_X}"', bend right=60,""{name=G}] \arrow[r,"{w_f}"] \arrow[d,"{F(\gamma)_X}"']\& {f\st Y} \arrow[d,"{F(\gamma)_Y}"] \arrow[dd,"{F(\delta \c \gamma)_Y}", bend left=60, ""'{name=F}] \arrow[ld,"{\eta_{\gamma}}", Rightarrow ,shorten <=3.5ex, shorten >= 3.5ex]\\
			{{f'}\st X} \arrow[from=G,"{}",iso] \arrow[d,"{F(\delta)_X}"']\arrow[r,"{w_{f'}}"'] \& {{f'}\st Y} \arrow[d,"{F(\delta)_Y}"] \arrow[to=F,"{}", iso] \arrow[ld,"{\eta_{\delta}}", Rightarrow ,shorten <=3.5ex, shorten >= 3.5ex]\\
			{{f''}\st X} \arrow[r,"{w_{f''}}"'] \& {{f''}\st Y}
		\end{cdN}
	\end{eqD*}

Given morphisms $f,f'\: D\to C$ in $S$, a 2-cell $\gamma\: f \Rightarrow f'$ and a morphism $g\: E\to D$ in $\K$, the following equality of 2-cells holds:

\begin{samepage}
	\begin{eqD*}
	\begin{cdsN}{6}{6}
		{} \& {(f\c g)\st X} \arrow[d,"{\phi^{f,g}}", Rightarrow ,shorten <=1.5ex, shorten >= 1.5ex, shift left=9ex] \arrow[r,"{}",aeq]\& {g\st f\st X}\arrow[rd,"{g\st w_f}"] \& {} \\
		{(\wt{f\c g})\st X} \arrow[d,"{F(R(g)(\gamma))_X}"'] \arrow[ru,"{}",aiso] \arrow[r,"{w_{\wt{f\c g}}}"']\& {(\wt{f\c g})\st Y} \arrow[ld,"{\eta_{R(g)(\gamma)}}"', Rightarrow,shorten <=3.5ex, shorten >= 3.5ex]\arrow[d,"{F(R(g)(\gamma))_Y}"{description},""'{name=R}] \arrow[r,"{}",aiso] \& {(f\c g)\st Y} \arrow[d,"{F(\gamma \star g)_Y}"{description},""'{name=S}] \arrow[r,"{}",aeq] \& {g\st f \st Y} \arrow[d,"{g\st (F(\gamma)_Y)}",""'{name=P}] \\
		{(\wt{f'\c g})\st X} \arrow[r,"{w_{\wt{f'\c g}}}"']\& {(\wt{f'\c g})\st Y} \arrow[r,"{}",aiso] \& {(f'\c g)\st Y} \arrow[r,"{}",aeq]\& {g\st {f '}\st Y} \arrow[from=R, to=S,"{}",iso] \arrow[from=S, to=P,"{}",iso]
	\end{cdsN}
\end{eqD*}
\vspace{-9.5mm}
\begin{eqD}{phicon2cella}
	\hspace{4ex}
	\begin{cdN}
		{\hphantom{.}} \arrow[d,"{}", equal ,shorten <=2.5ex, shorten >= 2.5ex]\\
		{\hphantom{.}}
	\end{cdN}
\end{eqD}
\vspace{-9.5mm}
\begin{eqD*}
	\begin{cdsN}{6}{6}
		{(\wt{f\c g})\st X} \arrow[r,"{}",aiso] \arrow[d,"{F(R(g)(\gamma))_X}"',""{name=A}] \& {(f\c g)\st X} \arrow[r,"{}",aeq] \arrow[d,"{F(\gamma \star g)_X}"'{description},""'{name=B}] \arrow[from=A,to=B,"{}",iso]\& {g\st f\st X} \arrow[r,"{g\st w_f}"] \arrow[d,"{g\st(F(\gamma)_X)}"{description},""{name=C}] \& {g\st f \st Y} \arrow[ld,"{g\st \eta_{\gamma}}", Rightarrow ,shorten <=3.5ex, shorten >= 3.5ex] \arrow[d,"{g\st(F(\gamma)_Y)}"]\\
		{(\wt{f'\c g})\st X} \arrow[rd,"{w_{\wt{f'\c g}}}"'] \arrow[r,"{}",aeq]\& {(f'\c g)\st X} \arrow[d,"{\phi^{f',g}}", Rightarrow ,shorten <=1.5ex, shorten >= 1.5ex, shift left=9ex]\arrow[r,"{}",aeq] \& {g\st {f'}\st X} \arrow[r,"{}",aeq] \& {g\st {f'} \st Y} \\
		{} \& {(\wt{f'\c g})\st Y} \arrow[r,"{}",aiso] \arrow[from=B,to=C,"{}",iso]\& {(f'\c g)\st Y} \arrow[ru,"{}",aeq] \& {}
	\end{cdsN}
\end{eqD*}
\end{samepage}

 Given $f\: D\to C$ in $S$, morphisms $g,g'\: E \to D$ in $\K$ and a 2-cell $\delta\: g\Rightarrow g'$, the following equality of 2-cells holds:
\begin{samepage}
	\begin{eqD*}
	\begin{cdsN}{6}{6}
		{} \& {(f\c g)\st X} \arrow[d,"{\phi^{f,g}}", Rightarrow ,shorten <=1.5ex, shorten >= 1.5ex, shift left=9ex] \arrow[r,"{}",aeq]\& {g\st f\st X}\arrow[rd,"{g\st w_f}"] \& {} \\
		{(\wt{f\c g})\st X} \arrow[d,"{F(\wt{f\star \delta})_X}"'] \arrow[ru,"{}",aiso] \arrow[r,"{w_{\wt{f\c g}}}"']\& {(\wt{f\c g})\st Y} \arrow[ld,"{\eta_{\wt{f\star \delta}}}"', Rightarrow,shorten <=3.5ex, shorten >= 3.5ex]\arrow[d,"{F(\wt{f\star \delta})_Y}"{description},""'{name=R}] \arrow[r,"{}",aiso] \& {(f\c g)\st Y} \arrow[d,"{F(f\star \delta)_Y}"{description},""'{name=S}] \arrow[r,"{}",aeq] \& {g\st f \st Y} \arrow[d,"{F(\delta)(f\st Y)}",""'{name=P}] \\
		{(\wt{f\c g'})\st X} \arrow[r,"{w_{\wt{f\c g'}}}"']\& {(\wt{f\c g'})\st Y} \arrow[r,"{}",aiso] \& {(f\c g')\st Y} \arrow[r,"{}",aeq]\& {{g'}\st {f }\st Y} \arrow[from=R, to=S,"{}",iso] \arrow[from=S, to=P,"{}",iso]
	\end{cdsN}
\end{eqD*}
\vspace{-4.5mm}
\begin{eqD}{phigg'}
	\hspace{4ex}
	\begin{cdN}
		{\hphantom{.}} \arrow[d,"{}", equal ,shorten <=2.5ex, shorten >= 2.5ex]\\
		{\hphantom{.}}
	\end{cdN}
\end{eqD}
\vspace{-4.5mm}
\begin{eqD*}
	\begin{cdsN}{6}{6}
		{(\wt{f\c g})\st X} \arrow[r,"{}",aiso]  \arrow[d,"{F(\wt{f\star \delta})_X}"',""{name=A}] \& {(f\c g)\st X} \arrow[r,"{}",aeq] \arrow[d,"{F(\wt{f\star \delta})_Y}"'{description},""{name=B}]  \arrow[from=A,to=B,"{}",iso]\& {g\st f\st X} \arrow[r,"{g\st w_f}"] \arrow[d,"{F(\delta)(f\st X)}"{description},""{name=C}] \& {g\st f \st Y}  \arrow[d,"{F(\delta)(f\st Y)}",""{name=D}]\\
		{(\wt{f\c g'})\st X} \arrow[rd,"{w_{\wt{f\c g'}}}"'] \arrow[r,"{}",aeq]\& {(f\c g')\st X} \arrow[d,"{\phi^{f,g'}}", Rightarrow ,shorten <=1.5ex, shorten >= 1.5ex, shift left=9ex]\arrow[r,"{}",aeq] \& {{g'}\st {f}\st X} \arrow[r,"{}",aeq] \& {{g'}\st {f} \st Y} \\
		{} \& {(\wt{f\c g'})\st Y} \arrow[r,"{}",aiso] \arrow[from=B,to=C,"{}",iso] \arrow[from=C,to=D,"{}",iso]\& {(f\c g')\st Y} \arrow[ru,"{}",aeq] \& {}
	\end{cdsN}
\end{eqD*}
\end{samepage}
\vspace{4.5mm}

	This descent datum on morphisms is  \dfn{effective} if there exists a morphism \linebreak[4] $w\: X\ar{} Y$ in $F(C)$ and, for every morphism $D\ar{f}C\in S$, an invertible 2-cell
\begin{eqD*}
	\begin{cdN}
		f\st X \arrow[r,"{f\st w}", bend left=30, ""'{name=A}] \arrow[r,"{w_f}"', bend right=30,""{name=B}] \arrow[twoiso,from=A, to=B,"\psi^{f}", shorten >=-0.6ex, shorten <=-0.3ex]\&  f\st Y
	\end{cdN}
\end{eqD*}
such that given $f,f'\: D\to C$ in $S$ and a 2-cell $\gamma\: f\Rightarrow f'$, the following 2-cells are equal:
\begin{eqD}{comp2cells}
	\begin{cdN}
		{f\st X} \arrow[r,"{}"' ,iso, shift left=-6ex] \arrow[r,"{f\st w}"',""{name=P}] \arrow[d,"{F(\gamma)_X}"'] \arrow[r,"{w_f}", ""{name=O}, bend left= 60] \arrow[from=O,to=P,"{\psi^f}", Leftarrow ,shorten <=0.5ex, shorten >= 0.5ex]\& {f\st Y}  \arrow[d,"{F(\gamma)_Y}"]\\
		{{f'}\st Y} \arrow[r,"{{f'}\st w}"']\& {{f'}\st X}
	\end{cdN}
\h[-3]
	\begin{cdN}
		{\hphantom{.}}  \arrow[r,""'{inner sep= 1ex}, equal ,shorten <=2.5ex, shorten >= 2.5ex]\& {\hphantom{.}}
	\end{cdN}
\h[-3]
	\begin{cdN}
		{f\st X} \arrow[r,"{}"' ,iso, shift left=-6ex]  \arrow[d,"{F(\gamma)_X}"'] \arrow[r,"{w_f}", ""{name=O}] \& {f\st Y}  \arrow[d,"{F(\gamma)_Y}"]\\
		{{f'}\st X} \arrow[r,"{{f'}\st w}"', ""{name=R}, bend right= 60] \arrow[r,"{w_{f'}}",""{name=Q}]\arrow[from=R,to=Q,"{\psi^{f'}}"', Rightarrow ,shorten <=0.5ex, shorten >= 0.5ex]\& {{f'}\st Y} 	\end{cdN}
	\end{eqD}
Moreover, given morphisms $E\ar{g} D \ar{f} C$ with $f\in S$, the following equality of 2-cells holds:
\begin{eqD*}
	\begin{cdN}
		{} \& {(f\c g)\st X} \arrow[r,"{}",aeq] \& {g\st f\st X} \arrow[dd,"{\phi^{f,g}}", Rightarrow ,shorten <=6.5ex, shorten >= 6.5ex,shift left=-8ex] \arrow[rd,"{g\st f\st w}",""'{name=A}, bend left=40] \arrow[rd,"{g\st w_f}"',""{name=B}, bend right=40] \arrow[from=A,to=B,"{\psi^f}", Rightarrow ,shorten <=1.5ex, shorten >= 1.5ex] \& {} \\
		{(\wt{f\c g})\st X} \arrow[rd,"{w_{\wt{f\c g}}}"] \arrow[ru,"{}",aiso] \& {} \& {} \& {g\st f\st Y} \\
		{} \& {(\wt{f\c g})\st Y} \arrow[r,"{}",aiso] \& {(f\c g)\st Y} \arrow[ru,"{}",aeq] \& {} 
	\end{cdN}
\end{eqD*}
\vspace{-6.5mm}
\begin{eqD}{condessmor}
	\begin{cdN}
		{\hphantom{.}} \arrow[d,"{}",equal ,shorten <=2.5ex, shorten >= 2.5ex, shift left=3ex]\\
		{\hphantom{.}}
	\end{cdN}
\end{eqD}
\vspace{-6.5mm}
\begin{eqD*}
	\begin{cdN}
		{} \& {(f\c g)\st X} \arrow[r,"{}",aeq] \& {g\st f\st X} \arrow[rd,"{g\st f\st w}"] \& {} \\
		{(\wt{f\c g})\st X} \arrow[rrr,"{}",iso] \arrow[rd,"{{(\wt{f\c g})\st w}}",""'{name=A}, bend left=40] \arrow[rd,"{w_{\wt{f\c g}}}"',""{name=B}, bend right=40] \arrow[from=A, to=B,"{\psi^{\wt{f\c g}}}", Rightarrow ,shorten <=1.5ex, shorten >= 1.5ex] \arrow[ru,"{}",aiso] \& {} \& {} \& {g\st f\st Y} \\
		{} \& {(\wt{f\c g})\st Y} \arrow[r,"{}",aiso] \& {(f\c g)\st Y} \arrow[ru,"{}",aeq] \& {} 
	\end{cdN}
\end{eqD*}
\end{defne}

\vspace{4.5mm}
We now make some comments about the previous definition, in order to compare it with the usual definition of descent datum.

\begin{oss}[cocycle condition on morphisms]
	Notice that the cocycle condition on morphisms (diagram (\ref{cocyclemor}) of the previous definition) relates $\phi^{\wt{f\c g},h}, h\st \phi^{f,g}$ and $\phi^{f,g\c h}$. Taking into account the pseudonaturality of the bisieve, this is the same data related by the usual cocycle of a descent datum (on objects). The only differences are that the cocycle condition on morphisms involves 2-cells instead of morphisms and that the canonical isomorphism given by the pseudofunctoriality becomes a canonical isomorphic 2-cell given by the fact that $F$ is a trihomomorphism.
\end{oss}

\begin{oss}[descent along the identity]
	In Definition \ref{descentdatumonmorphisms} we require a condition of normality of the isomorphisms of type $\phi^{f,\id{D}}$. This condition was not necessary in the case of descent data for objects as it was implied by the other conditions.
	\end{oss}

\begin{oss}[conditions on 2-cells]
	All the conditions regarding 2-cells were not involved in the notion of descent datum on objects because it was given for a 1-category. One could also give a definition of descent datum on objects in a 2-categorical context that involves conditions about 2-cells analogous to the ones of the previous definition. We are not interested in this notion as it is not involved in the theory of 2-stacks.
	\end{oss}

We are ready to prove our second characterization result.

\begin{prop}\label{effonmorphisms}
	Let $C\in \K$ and let $S\: R \Rightarrow \HomC{\K}{-}{C}$ be a bisieve on it. The pseudofunctor
	$$(-\c S)\c \Gamma\: F(C) \longrightarrow \Tricat(\K\op, \Bicat)(R, F)$$
	is essentially surjective on morphisms if and only if every descent datum for $S$ of morphisms of $F$ is effective.
\end{prop}
\begin{proof}
	Let $X,Y\in F(C)$. The images of $X$ and $Y$ under $(-\c S)\c \Gamma$ are respectively $\sigma^X \c S\: R \Rightarrow F$ and $\sigma^Y \c S\: R \Rightarrow F$, where $\sigma^X$ and $\sigma_Y$ are given as in Remark \ref{assYonedatricat}. 
	Let then $m\: \sigma^X \c S \aM{} \sigma^Y \c S$ be a trimodification. This is a trimodification of the kind described in Example \ref{ourtrimod}. Then $m$ is given by a family of pseudonatural transformations $m_D\: (\sigma^X \c S)_D \Rightarrow (\sigma^Y \c S)_D$ indexed by the objects of $\K$.  Notice that the image of a morphism $D\ar{f}C$ in $R(D)$ under the pseudofunctor $(\sigma^X \c S)_D $ is equal to $\sigma^X(f)$ and so it is $f\st X$ and the same holds for $(\sigma^Y \c S)_D$. So the component of $m_D$ on $f\: D \to C$ in $R(D)$ is a morphism
	$$(m_D)_f\: f\st X \to f\st Y$$
	The pseudonaturality of $m_D$ implies that, given morphisms $f,f'\in R(D)$ and a 2-cell $\gamma\: f \Rightarrow f'$, there is an invertible 2-cell
	\begin{eqD*}
			\begin{cdN}
				{f\st X} \arrow[r,"{(m_D)_f}"] \arrow[d,"{F(\gamma) _X}"']\& {f \st Y} \arrow[d,"{{F(\gamma) _Y}}"] \arrow[ld,"{{(m_D)}_{\gamma}}",twoiso, shorten <= 2ex, shorten >= 2ex]\\
				{{f'}\st X} \arrow[r,"{(m_D)_{f'}}"'] \& {{f'}\st Y}
		\end{cdN}
	\end{eqD*}
(Here we used that, by the tricategorical Yoneda lemma, $(\sigma^X)_D(\gamma)=F(\gamma)_X$ and analogously for $Y$.)
	And these structure 2-cells are pseudofunctorial so given \linebreak[4] $f\: D\to C$ in $R(D)$ we have 
	\begin{eqD}{norm}
		\csq*[l][7][7][(m_D)_{id_f}][2.5]{f\st X}{f\st Y}{f\st X}{f\st Y}{(m_D)_f}{F(\id{f})_X}{F(\id{f})_Y}{(m_D)_f}
		\h[3]= \h[3]
		\begin{cdN}
			{f\st X} \arrow[r,"{(m_D)_f}"] \arrow[d,"{F(\id{f})_X}"', ""{name=A},bend right=60]\arrow[d,"{}",equal,""'{name=B}] \arrow[from=A,to=B,"{}",iso]\& {f\st Y} \arrow[d,"{}",equal,""{name=C}] \arrow[d,"{F(\id{f})_Y}",""'{name=D}, bend left=60] \arrow[from=C, to=D,"{}", iso]\\
			{f\st X} \arrow[r,"{(m_D)_f}"']\& {f\st Y}
		\end{cdN}
	\end{eqD}

and given morphisms $f,f',f''\in R(D)$ and 2-cells $\gamma\: f \Rightarrow f'$ and $\delta\: f' \Rightarrow f''$ we have
	\begin{eqD*}
		\csq*[l][7][7][(m_D)_{\delta \c \gamma}][2.5]{f\st X}{f\st Y}{{f''}\st X}{{f''}\st Y}{(m_D)_f}{F(\delta\c \gamma)_X}{F(\delta\c \gamma)_Y} {(m_D)_{f''}}
		\h[3] = \h[3]
		\begin{cdN}
			{f\st X}\arrow[dd,"{F(\delta \c \gamma)_X}"', bend right=60,""{name=G}] \arrow[r,"{(m_D)_f}"] \arrow[d,"{F(\gamma)_X}"']\& {f\st Y} \arrow[d,"{F(\gamma)_Y}"] \arrow[dd,"{F(\delta \c \gamma)_Y}", bend left=60, ""'{name=F}] \arrow[ld,"{(m_D)_{\gamma}}", Rightarrow ,shorten <=3.5ex, shorten >= 3.5ex]\\
			{{f'}\st X} \arrow[from=G,"{}",iso] \arrow[d,"{F(\delta)_X}"']\arrow[r,"{(m_D)_{f'}}"'] \& {{f'}\st Y} \arrow[d,"{F(\delta)_Y}"] \arrow[to=F,"{}", iso] \arrow[ld,"{(m_D)_{\delta}}", Rightarrow ,shorten <=3.5ex, shorten >= 3.5ex]\\
			{{f''}\st X} \arrow[r,"{(m_D)_{f''}}"'] \& {{f''}\st Y}
		\end{cdN}
	\end{eqD*}
Furthermore, given  a morphism $g\: E \to D$ in $\K$, we have an invertible modification 
		\begin{eqD*}
		\begin{cdN}
			{S(D)} \arrow[r,"{(\sigma^Y\c S)_D}"',""{name=P}] \arrow[d,"{S(g)}"'] \arrow[r,"{(\sigma^X \c S)_D}", ""{name=O}, bend left= 60] \arrow[from=O,to=P,"{m_D}", Leftarrow ,shorten <=0.5ex, shorten >= 0.5ex]\& {F(D)}  \arrow[d,"{F(g)}"]\\
			{R(E)} \arrow[ru,"{(\sigma^Y\c S)_g}"' ,shorten <=3.5ex, shorten >= 3.5ex, Rightarrow] \arrow[r,"{(\sigma^Y\c S)_E}"']\& {F(E)} 
		\end{cdN}
		\begin{cdN}
			{\hphantom{.}}  \arrow[r,"{{m_g}}"'{inner sep= 1ex}, triple]\& {\hphantom{.}}
		\end{cdN}
		\begin{cdN}
			{S(D)}  \arrow[d,"{S(g)}"'] \arrow[r,"{(\sigma^X \c S)_D}", ""{name=O}] \& {F(D)}  \arrow[d,"{F(g)}"]\\
			{S(E)} \arrow[ru,"{(\sigma^X \c S)_g}"' ,shorten <=3.5ex, shorten >= 3.5ex, Rightarrow] \arrow[r,"{(\sigma^Y\c S)_E}"', ""{name=R}, bend right= 60] \arrow[r,"{(\sigma^X \c S)_E}",""{name=Q}]\arrow[from=R,to=Q,"{m_E}", Rightarrow ,shorten <=0.5ex, shorten >= 0.5ex]\& {F(E)} 	\end{cdN}
	\end{eqD*}
This means that, given morphisms $f,f'\: D\to C$ in $S$, a 2-cell $\gamma\: f \Rightarrow f'$ and a morphism $g\: E\to D$ in $\K$, the following equality holds:
\begin{samepage}
	\begin{eqD*}
	\begin{cdsN}{6}{6}
		{} \& {(f\c g)\st X} \arrow[d,"{(m_g)_f}", Rightarrow ,shorten <=1.5ex, shorten >= 1.5ex, shift left=9ex] \arrow[r,"{}",aeq]\& {g\st f\st X}\arrow[rd,"{g\st (m_D)_f}"] \& {} \\
		{(\wt{f\c g})\st X} \arrow[d,"{F(R(g)(\gamma))_X}"'] \arrow[ru,"{}",aiso] \arrow[r,"{m_E({\wt{f\c g}})}"]\& {(\wt{f\c g})\st Y} \arrow[ld,"{(m_E)_{R(g)(\gamma)}}"', Rightarrow,shorten <=3.5ex, shorten >= 3.5ex]\arrow[d,"{F(R(g)(\gamma))_Y}"{description},""'{name=R}] \arrow[r,"{}",aiso] \& {(f\c g)\st Y} \arrow[d,"{F(\gamma \star g)_Y}"{description},""'{name=S}] \arrow[r,"{}",aeq] \& {g\st f \st Y} \arrow[d,"{g\st (F(\gamma)_Y)}",""'{name=P}] \\
		{(\wt{f'\c g})\st X} \arrow[r,"{m_E({\wt{f'\c g}})}"']\& {(\wt{f'\c g})\st Y} \arrow[r,"{}",aiso] \& {(f'\c g)\st Y} \arrow[r,"{}",aeq]\& {g\st {f '}\st Y} \arrow[from=R, to=S,"{}",iso] \arrow[from=S, to=P,"{}",iso]
	\end{cdsN}
\end{eqD*}
\vspace{-4.5mm}
\begin{eqD*}
	\hspace{4ex}
	\begin{cdN}
		{\hphantom{.}} \arrow[d,"{}", equal ,shorten <=2.5ex, shorten >= 2.5ex]\\
		{\hphantom{.}}
	\end{cdN}
\end{eqD*}
\vspace{-4.5mm}
\begin{eqD*}
	\begin{cdsN}{6}{6}
		{(\wt{f\c g})\st X} \arrow[r,"{}",aiso] \arrow[d,"{F(R(g)(\gamma))_X}"',""{name=A}] \& {(f\c g)\st X} \arrow[r,"{}",aeq] \arrow[d,"{F(\gamma \star g)_X}"'{description},""'{name=B}] \arrow[from=A,to=B,"{}",iso]\& {g\st f\st X} \arrow[r,"{g\st (m_D)_f}"] \arrow[d,"{g\st(F(\gamma)_X)}"{description},""{name=C}] \& {g\st f \st Y} \arrow[ld,"{g\st (m_D)_{\gamma}}", Rightarrow ,shorten <=3.5ex, shorten >= 3.5ex] \arrow[d,"{g\st(F(\gamma)_Y)}"]\\
		{(\wt{f'\c g})\st X} \arrow[rd,"{m_E({\wt{f'\c g}})}"'] \arrow[r,"{}",aeq]\& {(f'\c g)\st X} \arrow[d,"{(m_g)_f'}", Rightarrow ,shorten <=1.5ex, shorten >= 1.5ex, shift left=9ex]\arrow[r,"{}",aeq] \& {g\st {f'}\st X} \arrow[r,"{}",aeq] \& {g\st {f'} \st Y} \\
		{} \& {(\wt{f'\c g})\st Y} \arrow[r,"{}",aiso] \arrow[from=B,to=C,"{}",iso]\& {(f'\c g)\st Y} \arrow[ru,"{}",aeq] \& {}
	\end{cdsN}
\end{eqD*}
\end{samepage}

The $m_g$'s for $g\: E\to D$ form a modification $\wt{m}$. So given $g,g'\: E \to D$ and a 2-cell $\delta\: g \Rightarrow g'$ the following equality holds:

\begin{samepage}
	\begin{eqD*}
		\begin{cdsN}{6}{6}
			{} \& {(f\c g)\st X} \arrow[d,"{(m_g)_f}", Rightarrow ,shorten <=1.5ex, shorten >= 1.5ex, shift left=9ex] \arrow[r,"{}",aeq]\& {g\st f\st X}\arrow[rd,"{g\st (m_D(f))}"] \& {} \\
			{(\wt{f\c g})\st X} \arrow[d,"{F(\wt{f\star \delta})_X}"'] \arrow[ru,"{}",aiso] \arrow[r,"{m_E({\wt{f\c g}})}"]\& {(\wt{f\c g})\st Y} \arrow[ld,"{(m_E)_{\wt{f\star \delta}}}"', Rightarrow,shorten <=3.5ex, shorten >= 3.5ex]\arrow[d,"{F(\wt{f\star \delta})_Y}"{description},""'{name=R}] \arrow[r,"{}",aiso] \& {(f\c g)\st Y} \arrow[d,"{F(f\star \delta)_Y}"{description},""'{name=S}] \arrow[r,"{}",aeq] \& {g\st f \st Y} \arrow[d,"{F(\delta)(f\st Y)}",""'{name=P}] \\
			{(\wt{f\c g'})\st X} \arrow[r,"{m_E({\wt{f\c g'}})}"']\& {(\wt{f\c g'})\st Y} \arrow[r,"{}",aiso] \& {(f\c g')\st Y} \arrow[r,"{}",aeq]\& {{g'}\st {f }\st Y} \arrow[from=R, to=S,"{}",iso] \arrow[from=S, to=P,"{}",iso]
		\end{cdsN}
	\end{eqD*}
	\vspace{-4.5mm}
	\begin{eqD*}
		\hspace{4ex}
		\begin{cdN}
			{\hphantom{.}} \arrow[d,"{}", equal ,shorten <=2.5ex, shorten >= 2.5ex]\\
			{\hphantom{.}}
		\end{cdN}
	\end{eqD*}
	\vspace{-4.5mm}
	\begin{eqD*}
		\begin{cdsN}{6}{6}
			{(\wt{f\c g})\st X} \arrow[r,"{}",aiso]  \arrow[d,"{F(\wt{f\star \delta})_X}"',""{name=A}] \& {(f\c g)\st X} \arrow[r,"{}",aeq] \arrow[d,"{F(\wt{f\star \delta})_Y}"'{description},""{name=B}]  \arrow[from=A,to=B,"{}",iso]\& {g\st f\st X} \arrow[r,"{g\st( m_D(f))}"] \arrow[d,"{F(\delta)(f\st X)}"{description},""{name=C}] \& {g\st f \st Y} \arrow[d,"{F(\delta)(f\st Y)}",""{name=D}]\\
			{(\wt{f\c g'})\st X} \arrow[rd,"{m_E({\wt{f\c g'}})}"'] \arrow[r,"{}",aeq]\& {(f\c g')\st X} \arrow[d,"{(m_{g'})_f}", Rightarrow ,shorten <=1.5ex, shorten >= 1.5ex, shift left=9ex]\arrow[r,"{}",aeq] \& {{g'}\st {f}\st X} \arrow[r,"{}",aeq] \& {{g'}\st {f} \st Y} \\
			{} \& {(\wt{f\c g'})\st Y} \arrow[r,"{}",aiso] \arrow[from=B,to=C,"{}",iso]\& {(f\c g')\st Y} \arrow[from=C,to=D,"{}",iso] \arrow[ru,"{}",aeq] \& {}
		\end{cdsN}
	\end{eqD*}
\end{samepage}
Moreover, the first axiom of trimodification for $m$ (see Example \ref{ourtrimod}) computed on component $f\: D\to C \in R(D)$ yield the following condition.

Given $D''\ar{h} D' \ar{g} D \ar{f} C$ with $f\in R(D)$, the following equality holds:
\begin{samepage}
	\begin{eqD*}\scalebox{0.8}{
			\begin{cdsN}{7}{3}
				{(\wt{\wt{f\c g}\c h})\st X} \arrow[rr,"{}",aiso] \arrow[d,"{}",aiso] \arrow[rd,"{m_{D''}({\wt{\wt{f\c g}\c h}})}"]\& {} \& {(\wt{f\c g}\c h)\st X} \arrow[ld,"{(m_h)(\wt{f\c g})}", Rightarrow ,shorten <=2.5ex, shorten >= 2.5ex, shift left=2ex] \arrow[r,"{}",aeq] \& {h\st(\wt{f\c g})\st X} \arrow[d,"{h\st (m_{D'}({\wt{f\c g}}))}"] \arrow[r,"{}",aiso] \& {h\st({f\c g})\st X}  \arrow[r,"{}",aeq] \& {h\st g\st f\st X} \arrow[dd,"{h\st g\st (m_D(f))}"] \arrow[lld,"{h\st( (m_g)_f)}", Rightarrow ,shorten <=8.5ex, shorten >= 8.5ex, shift left=4ex]\\
				{(\wt{{f\c g}\c h})\st X} \arrow[r,"{}",iso, shift left=-2ex] \arrow[rdd,"{m_{D''}({\wt{f\c g\c h}})}"']\& {(\wt{\wt{f\c g}\c h})\st Y} \arrow[dd,"{}", aiso]\arrow[rd,"{}", aiso] \& {} \& {h\st({f\c g})\st Y} \arrow[rd,"{}",aiso,""{name=A}]  \& {} \& {}\\
				{} \& {} \& {(\wt{f\c g}\c h)\st Y} \arrow[rr,"{}",iso,]\arrow[ru,"{}",aeq] \arrow[rd,"{}", aeq] \& {} \& {h\st (f\c g)\st Y} \arrow[r,"{}",aeq]\& {h\st g\st f\st Y} \arrow[d,"{}", aeq]\\
				{} \& {{(\wt{f\c g}\c h)\st Y}} \arrow[rr,"{}",aeq] \& {} \& {(f\c g\c h)\st Y} \arrow[ru,"{}", aeq] \arrow[rr,"{}", aeq]\& {} \& {(g\c h)\st f\st Y}
		\end{cdsN}}
	\end{eqD*}
	\begin{eqD*}
		\vspace*{-1.8ex}\hspace*{5.7cm}
		\begin{cdsN}{3}{3}
			\hphantom{.}\arrow[d,equal]\\
			\hphantom{.}
		\end{cdsN}\hspace{5.4cm}
	\end{eqD*}
	\begin{eqD*}\scalebox{0.8}{
			\begin{cdsN}{9}{7}
				{(\wt{\wt{f\c g}\c h})\st X} \arrow[r,"{}",aiso] \arrow[d,"{}",aiso] \& {(\wt{f\c g}\c h)\st X} \arrow[rr,"{}",iso ,shift left=-5ex] \arrow[rd,"{}",aiso] \arrow[r,"{}",aeq] \& {h\st(\wt{f\c g})\st X}  \arrow[r,"{}",aiso] \& {h\st({f\c g})\st X}  \arrow[r,"{}",aeq] \& {h\st g\st f\st X} \arrow[d,"{h\st g\st (m_D(f))}"] \\
				{(\wt{f\c g\c h})\st X} \arrow[rr,"{}",aeq] \arrow[rd,"{m_{D''}({\wt{f\c g\c h}}}"']\& {} \& {({f\c g\c h})\st X} \arrow[ru,"{}",aeq] \arrow[r,"{}",aiso]\& {(g\c h)\st f \st X} \arrow[lld,"{(m_{g\c h})_f}", Rightarrow ,shorten <=10.5ex, shorten >= 10.5ex]\arrow[ru,"{}",aeq] \arrow[rd,"{(g\c h)\st (m_D(f))}"'] \arrow[r,"{}",iso] \& {h\st g\st f\st Y} \arrow[d,"{}", aiso] \\
				{} \& {(\wt{f\c g\c h})\st Y} \arrow[r,"{}",aiso]  \& {({f\c g\c h})\st Y} \arrow[rr,"{}", aeq] \& {} \& {(g\c h)\st f\st Y} 
		\end{cdsN}}
	\end{eqD*}
\end{samepage}
And the second axiom of trimodification computed on component $f$ gives again the condition of normalization of Diagram (\ref{norm}).
We observe that the data and the axioms of the trimodification $m$ and exactly the same of those of a descent datum on $S$ of morphisms of $F$, where the assignment $w_f$ on $f\: D\to C$ is the morphism $m_D(f)$ and the assignment $\eta_{\gamma}$ on a 2-cell $\gamma\: F \Rightarrow f'$ is the 2-cell $(m_D)_{\gamma}$. And the compatibility conditions are given by the functoriality and naturality of the structure cells involved and by the axioms of trimodification.

It suffices, then, to prove that the trimodification $m$ is isomorphic to a trimodification of the form $((-\c S)\c \Gamma)(w)$ for some  morphism $w\: X \to Y$ in $F(C)$ if and only if the corresponding descent datum of morphisms is effective.

If $m$ is isomorphic to $((-\c S)\c \Gamma)(w)$, this means that there exists an isomorphic perturbation $p\: n\c S \aP{} m$, where $n\: \sigma^X \aM{} \sigma^Y$ is the trimodification $m^w$ (image of $w$ under $\Gamma$, see Remark \ref{assYonedatricat}). So $p$ is given by a family of isomorphic modifications $p_D\: m_D \Rightarrow (n\c S)_D$ indexed on the objects of $\K$. Then, for every $f\: D\to C$ in $R(D)$, we have an isomorphic 2-cell 
\begin{eqD*}
	\begin{cdN}
		f\st X \arrow[r,"{f\st w}", bend left=30, ""'{name=A}] \arrow[r,"{m_D(f)}"', bend right=30,""{name=B}] \arrow[twoiso,from=A, to=B,"(p_D)_f", shorten >=-0.6ex, shorten <=-0.3ex]\&  f\st Y
	\end{cdN}
\end{eqD*}
These 2-cells are such that, given $f,f'\: D\to C$ in $S$ and a 2-cell $\gamma\: f\Rightarrow f'$, the following 2-cells are equal:
\begin{eqD*}
	\begin{cdN}
		{f\st X} \arrow[r,"{}"' ,iso, shift left=-6ex] \arrow[r,"{f\st w}"',""{name=P}] \arrow[d,"{F(\gamma)_X}"'] \arrow[r,"{m_D(f)}", ""{name=O}, bend left= 60] \arrow[from=O,to=P,"{(p_D)_f}", Leftarrow ,shorten <=0.5ex, shorten >= 0.5ex]\& {f\st Y}  \arrow[d,"{F(\gamma)_Y}"]\\
		{{f'}\st Y} \arrow[r,"{{f'}\st w}"']\& {{f'}\st X}
	\end{cdN}
	\begin{cdN}
		{\hphantom{.}}  \arrow[r,""'{inner sep= 1ex}, equal ,shorten <=2.5ex, shorten >= 2.5ex]\& {\hphantom{.}}
	\end{cdN}
	\begin{cdN}
		{f\st X} \arrow[r,"{}"' ,iso, shift left=-6ex]  \arrow[d,"{F(\gamma)_X}"'] \arrow[r,"{m_D(f)}", ""{name=O}] \& {f\st Y}  \arrow[d,"{F(\gamma)_Y}"]\\
		{{f'}\st X} \arrow[r,"{{f'}\st w}"', ""{name=R}, bend right= 60] \arrow[r,"{m_D({f'})}",""{name=Q}]\arrow[from=R,to=Q,"{(p_D)_{f'}}"', Rightarrow ,shorten <=0.5ex, shorten >= 0.5ex]\& {{f'}\st Y} 	\end{cdN}
\end{eqD*}
Moreover, given morphisms $E\ar{g} D \ar{f} C$ with $f\in S$, the following equality of 2-cells holds:
\begin{eqD*}
	\begin{cdN}
		{} \& {(f\c g)\st X} \arrow[r,"{}",aeq] \& {g\st f\st X} \arrow[dd,"{(m_g)_f}", Rightarrow ,shorten <=6.5ex, shorten >= 6.5ex,shift left=-8ex] \arrow[rd,"{g\st f\st w}",""'{name=A}, bend left=40] \arrow[rd,"{g\st m_D(f)}"',""{name=B}, bend right=40] \arrow[from=A,to=B,"{(p_D)_f}", Rightarrow ,shorten <=1.5ex, shorten >= 1.5ex] \& {} \\
		{(\wt{f\c g})\st X} \arrow[rd,"{m_E({\wt{f\c g}})}"'] \arrow[ru,"{}",aiso] \& {} \& {} \& {g\st f\st Y} \\
		{} \& {(\wt{f\c g})\st Y} \arrow[r,"{}",aiso] \& {(f\c g)\st Y} \arrow[ru,"{}",aeq] \& {} 
	\end{cdN}
\end{eqD*}
\vspace{-6.5mm}
\begin{eqD*}
	\begin{cdN}
		{\hphantom{.}} \arrow[d,"{}",equal ,shorten <=2.5ex, shorten >= 2.5ex, shift left=3ex]\\
		{\hphantom{.}}
	\end{cdN}
\end{eqD*}
\vspace{-6.5mm}
\begin{eqD*}
	\begin{cdN}
		{} \& {(f\c g)\st X} \arrow[r,"{}",aeq] \& {g\st f\st X} \arrow[rd,"{g\st f\st w}"] \& {} \\
		{(\wt{f\c g})\st X} \arrow[rrr,"{}",iso] \arrow[rd,"{{(\wt{f\c g})\st w}}",""'{name=A}, bend left=40] \arrow[rd,"{m_E({\wt{f\c g}})}"',""{name=B}, bend right=40] \arrow[from=A, to=B,"{(p_E)_{\wt{f\c g}}}", Rightarrow ,shorten <=1.5ex, shorten >= 1.5ex] \arrow[ru,"{}",aiso] \& {} \& {} \& {g\st f\st Y} \\
		{} \& {(\wt{f\c g})\st Y} \arrow[r,"{}",aiso] \& {(f\c g)\st Y} \arrow[ru,"{}",aeq] \& {} 
	\end{cdN}
\end{eqD*}
This happens exactly when the descent datum that corresponds to $m$ is effective with isomorphic 2-cell $\psi^f$ relative to $f\: D\to C$ equal to $(p_D)_f$.
\end{proof}

We now consider the gluing condition on objects. This is the newest condition and we expect it to be the weakest one. 

We give the following definition.

\begin{defne}\label{weakdescentdatum}
	Let $F\: \K\op \to \Bicat$ be a trihomomorphism and let $S$ be a bisieve on $C\in \K$. 
	
	A \dfn{weak descent datum for $S$ of elements of $F$} is an assignment for every morphism $D\ar{f} C$ in $S$ of an object $W_f \in F(D)$, for every 2-cell $\gamma\: f \Rightarrow f'$ with $f,f'\: D\to C$ in $S$ of a morphism 
	$$W_f \ar{\eta_{\gamma}} W_{f'}$$
	and for every pair of composable morphisms 
	$D'\ar{g}D\ar{f}C$
	with $f\in S$, of an equivalence
	$${\phi^{f,g}} \: W_{\wt{{f\c g}}}\aequi g\st W_f .$$
	
	Moreover, for every morphism $D\ar{f} C$ in $S$ we have an isomorphic 2-cell
		\begin{cd}
		{W_f}  \arrow[r,"{}",aiso, bend right=40,""{name=B}] \arrow[r,"{\phi^{f,\id{D}}}", bend left =40,""'{name=A}]  \arrow[twoiso,from=A, to=B,"\rho_f", shorten >=1ex, shorten <=1ex]\& {\id{D}\st W_f} \\
	\end{cd}
	and for every chain of morphisms $L \ar{h} E \ar{g} D\ar{f} C$ with $f\in S$ we have an isomorphism
		\begin{eqD}{weakcocycle}
		\begin{cdN}
			{W_{\wt{\wt{f\c g}\c h}}} \arrow[d,"{}",aiso]\arrow[r,"{\phi^{\wt{f\c g},h}}"] \& {g\st W_{\wt{f\c g}}} \arrow[r,"{h\st \phi^{f,g}}"]\& {h\st g\st W_f} \arrow[d,"{}",aeq] \arrow[lld,"{{\beta}^{g,h}_f}", Rightarrow ,shorten <=9.5ex, shorten >= 9.5ex]\\
			{W_{\wt{f\c g\c h}}} \arrow[rr,"{\phi^{f,g\c h}}"'] \& {} \& {(g\c h)\st W_f}
		\end{cdN}
	\end{eqD}
	
	\noindent Finally, given $f,f'\: D\to C$ in $R(D)$ and a 2-cell $\gamma\: f \Rightarrow f'$, we have an isomorphic 2-cell
	\begin{cd}
		{W_{\wt{f\c g}}} \arrow[r,"{\phi^{f,g}}"] \arrow[d,"{\eta_{R(g)(\gamma)}}"']\& {g\st W_f} \arrow[d,"{g\st \eta_{\gamma}}"] \arrow[ld,"{\rho^{g}_{\gamma}}",twoiso, shorten <= 2ex, shorten >= 2ex]\\
		{W_{\wt{f'\c g}}} \arrow[r,"{\phi^{f',g}}"'] \& {g\st W_{f'}}
	\end{cd}
and given $f\: D\to C$ in $R(D)$, morphisms $g,g'\: E \to D$ and a 2-cell $\delta\: g \Rightarrow g'$, we have an isomorphic 2-cell 
		\begin{cd}
		{W_{\wt{f\c g}}} \arrow[r,"{\phi^{f,g}}"] \arrow[d,"{W_{R(\delta)_f}}"']\& {g\st W_f} \arrow[d,"{F(\delta)_{W_f}}"] \arrow[ld,"{(\alpha_{\delta})_f}",twoiso, shorten <= 2ex, shorten >= 2ex]\\
		{W_{\wt{f\c g'}}} \arrow[r,"{\phi^{f,{g'}}}"'] \& {{g'}\st W_{f}}
	\end{cd}
These data need to satisfy the following compatibility conditions.

Given $f\: D\to C$ in $R(D)$, there exists an isomorphic 2-cell 
	\begin{cd}
		{W_f}  \arrow[r,"{}",equal, bend right=40,""{name=B}] \arrow[r,"{\eta_{\id{f}}}", bend left =40,""'{name=A}]  \arrow[r,"{}",iso]\& {W_f}
	\end{cd}
and the following equalities of 2-cells hold:
\begin{eqD*}
	\begin{cdN}
	{W_{\wt{f\c g}}} \arrow[d,"{}",equal, bend right=60,""{name=A}]\arrow[r,"{\phi^{f,g}}"] \arrow[d,"{\eta_{R(g)(\id{f})}}"'{description},""{name=B}] \arrow[from=A,to=B,"{}",iso, shift left=2ex]\& {g\st W_f} \arrow[d,"{}",equal, bend left=60,""{name=D}]\arrow[d,"{g\st \eta_{\id{f}}}"{description},""'{name=C}] \arrow[ld,"{\rho_{\id{f}}}",twoiso, shorten <= 2ex, shorten >= 2ex]\\
	{W_{\wt{f\c g}}} \arrow[r,"{\phi^{f,g}}"'] \arrow[from=C,to=D,"{}",iso, shift left=2ex]\& {g\st W_{f}}
\end{cdN}
\h[7]=\h[2]
\begin{cdN}
	{W_{\wt{f\c g}}} \arrow[d,"{}",equal] \arrow[r,"{\phi^{f,g}}"] \& {g\st W_f} \arrow[d,"{}",equal]\\
	{W_{\wt{f\c g}}}  \arrow[r,"{\phi^{f,g}}"']\& {g\st W_f}
\end{cdN}
\end{eqD*}
and 
\begin{eqD*}
	\begin{cdN}
	{W_{{f}}} \arrow[r,"{\phi^{f,\id{D}}}",""'{name=A}, bend left=40] \arrow[r,"{}",aeq,""{name=B}, bend right=40] \arrow[from=A,to=B,"{\rho_f}",Rightarrow ,shorten <=1.5ex, shorten >= 1.5ex] \arrow[d,"{\eta_{\gamma}}"',""{name=C}]\& {{\id{D}}\st W_f} \arrow[d,"{{\id{D}}\st \eta_{\gamma}}",""'{name=D}]\arrow[from=C,to=D,"{}",iso,shift left=-1ex]\\
	{W_{{f'}}} \arrow[r,"{}"',aeq] \& {\id{D}\st W_{f'}}
\end{cdN}
\h[3]=\h[3]
\begin{cdN}
	{W_{{f}}} \arrow[r,"{\phi^{f,\id{D}}}"] \arrow[d,"{\eta_{\gamma}}"']\& {{\id{D}}\st W_f} \arrow[d,"{{\id{D}}\st \eta_{\gamma}}"] \arrow[ld,"{\rho^{{\id{D}}}_{\gamma}}"{inner sep=0.001ex},twoiso,shift left=-2ex, shorten <= 4ex, shorten >= 4ex]\\
	{W_{{f'}}} \arrow[r,"{\phi^{f',\id{D}}}"', bend left=40,""{name=E}] \arrow[r,"{}",aeq,""{name=F}, bend right=40] \arrow[from=E,to=F,"{\rho_{f'}}",Rightarrow ,shorten <=1.5ex, shorten >= 1.5ex]\& {\id{D}\st W_{f'}}
\end{cdN}
\end{eqD*}
\noindent Given $f\: D\to C$ in $R(D)$ and $g\: E \to D$ in $\K$, the following equality of 2-cells holds:
		\begin{eqD*}
		\begin{cdN}
			{W_{\wt{f\c g}}} \arrow[r,"{\phi^{f,g}}"] \arrow[d,"{}", equal, bend right=60 ,""{name=A}] \arrow[d,"{{R(\id{g})_f}}"'{description},""'{name=B}]\& {g\st W_f} \arrow[d,"{F(\id{g})_{W_f}}"{description},""'{name=C}] \arrow[d,"{}", equal, bend left=60 ,""'{name=D}] \arrow[from=A,to=B,"{}",iso,shift left=2ex] \arrow[from=C, to=D,"{}",iso,shift left=2ex] \arrow[ld,"{(\alpha_{\id{g}})_f}",twoiso, shorten <= 2ex, shorten >= 2ex]\\
			{W_{\wt{f\c g}}} \arrow[r,"{\phi^{f,{g}}}"'] \& {{g}\st W_{f}}
		\end{cdN}
		\h[7]=\h[2]
		\begin{cdN}
			{W_{\wt{f\c g}}} \arrow[d,"{}",equal] \arrow[r,"{\phi^{f,g}}"] \& {g\st W_f} \arrow[d,"{}",equal]\\
			{W_{\wt{f\c g}}}  \arrow[r,"{\phi^{f,g}}"']\& {g\st W_f}
		\end{cdN}
	\end{eqD*}

\noindent Given morphisms $f,f',f''\: D\to C$ in $R(D)$ and 2-cells $\gamma\: f \Rightarrow f'$ and $\delta\: f' \Rightarrow f''$, there exists an invertible 2-cell

\begin{cd}
	{W_f} \arrow[rr,"{}",iso,shift left=-4ex] \arrow[rr, bend right =40, "{\eta_{\delta \c \gamma}}"'] \arrow[r,"{\eta_{\gamma}}"] \& {W_{f'}} \arrow[r,"{\eta_{\delta}}"] \& {W_{f''}}
\end{cd}
and the following equalities of 2-cells hold:
	\begin{eqD*}
	\begin{cdN}
		{W_{\wt{f\c g}}} \arrow[r,"{\phi^{f,g}}"] \arrow[d,"{\eta_{R(g)(\delta \c \gamma)}}"']\& {g\st W_f} \arrow[d,"{g\st \eta_{\delta \c \gamma}}"] \arrow[ld,"{\rho_{\delta \c \gamma}}",twoiso, shorten <= 2ex, shorten >= 2ex]\\
		{W_{\wt{f''\c g}}} \arrow[r,"{\phi^{f'',g}}"'] \& {g\st W_{f''}}
	\end{cdN}
	\h[3] = \h[3]
	\begin{cdN}
		{W_{\wt{f\c g}}} \arrow[dd,"{\eta_{R(g)(\delta \c \gamma)}}"', bend right=60,""{name=A}]\arrow[r,"{\phi^{f,g}}"] \arrow[d,"{\eta_{R(g)(\gamma)}}"']\& {g\st W_f} \arrow[d,"{g\st \eta_{\gamma}}"] \arrow[ld,"{\rho_{\gamma}}",twoiso, shorten <= 2ex, shorten >= 2ex]\arrow[dd,"{g\st \eta_{\delta \c \gamma}}", bend left=60,""'{name=B}]\\
		{W_{\wt{f'\c g}}} \arrow[from=A,"{}",iso]\arrow[d,"{\eta_{R(g)(\delta)}}"'] \arrow[r,"{\phi^{f',g}}"'] \& {g\st W_{f'}} \arrow[d,"{g\st \eta_{\delta}}"] \arrow[ld,"{\rho_{\delta}}",twoiso, shorten <= 2ex, shorten >= 2ex] \arrow[to=B,"{}",iso]\\
		{W_{\wt{f''\c g}}} \arrow[r,"{\phi^{f'',g}}"'] \& {g\st W_{f''}}
	\end{cdN}
\end{eqD*}
and 
\begin{eqD*}
	\begin{cdN}
		{} \& {h\st W_{\wt{f\c g}}} \arrow[d,"{\beta^{g,h}_f}", Rightarrow ,shorten <=1.5ex, shorten >= 1.5ex, shift left=7ex]\arrow[r,"{h\st \phi^{f,g}}"]\& {h\st g\st W_f} \arrow[rd,"{}",aeq] \& {}\\
		{W_{\wt{\wt{f\c g}\c h}}}  \arrow[r,"{}",iso, shift left=-6ex]\arrow[r,"{}",aiso] \arrow[d,"{\eta_{R(g)(\gamma \star g)}}"'] \arrow[ru,"{\phi^{\wt{f\c g},h}}"] \& {W_{\wt{f \c g \c h}}} \arrow[d,"{\eta_{R(g\c h)(\gamma)}}"{description}]  \arrow[rr,"{\phi^{f,g\c h}}"]\& {} \& {(g\c h)\st W_f} \arrow[d,"{(g\c h)\st \eta_{\gamma}}"] \arrow[lld,"{\rho^{g\c h}_{\gamma}}", Rightarrow ,shorten <=10.5ex, shorten >= 10.5ex]\\
		{W_{\wt{\wt{f'\c g}\c h}}} \arrow[r,"{}",aiso]\& {W_{\wt{f '\c g \c h}}} \arrow[rr,"{\phi^{f',g\c h}}"'] \& {} \& {(g\c h)\st W_{f'}}\\
	\end{cdN}
\end{eqD*}
\vspace{-15.5mm}
\begin{eqD*}
	\begin{cdN}
		{\hphantom{.}} \arrow[d,"{}",equal ,shorten <=2.5ex, shorten >= 2.5ex, shift left=3ex]\\
		{\hphantom{.}}
	\end{cdN}
\end{eqD*}
\vspace{-6.5mm}
\begin{eqD*}
	\begin{cdN}
		{W_{\wt{\wt{f\c g}\c h}}} \arrow[r,"{}",iso, shift left=-6ex]\arrow[d,"{\eta_{R(g)(\gamma \star g)}}"'] \arrow[r,"{\phi^{\wt{f\c g},h}}"] \& {h\st W_{\wt{f\c g}}} \arrow[d,"{h\st \eta_{R(g)(\gamma)}}"{description}]\arrow[r,"{h\st \phi^{f,g}}"] \& {h\st g\st W_f} \arrow[ld,"{h\st \rho^{g}_{\gamma}}", Rightarrow ,shorten <=4ex, shorten >= 4ex] \arrow[r,"{}",iso, shift left=-6ex] \arrow[d,"{h\st g\st \eta_{\gamma}}"{description}] \arrow[r,"{}",aeq] \& {(g\c h)\st W_f} \arrow[d,"{(g\c h)\st \eta_{\gamma}}"]\\
		{W_{\wt{\wt{f'\c g}\c h}}} \arrow[r,"{\phi^{\wt{f\c g},h}}",] \arrow[rd,"{}",aiso] \& {h\st W_{\wt{f'\c g}}} \arrow[d,"{\beta^{g,h}_{f'}}", Rightarrow ,shorten <=1.5ex, shorten >= 1.5ex, shift left=7ex] \arrow[r,"{h\st \phi^{f',g}}"] \& {h\st g\st W_{f'}} \arrow[r,"{}",aeq]\& {(g\c h)\st W_{f'}}\\
		{} \& {W_{\wt{f '\c g \c h}}} \arrow[rru,"{\phi^{f',g\c h}}"', bend right=15] \& {} \& {}
	\end{cdN}
\end{eqD*}

Given morphisms $f\: D\to C$ in $R(D)$, $g,g',g''\: E\to D$ in $\K$ and 2-cells $\delta\: g \Rightarrow g'$ and $\varepsilon\: g' \Rightarrow g''$, the following equality of 2-cells holds:
\begin{eqD*}
	\begin{cdN}
		{W_{\wt{f\c g}}} \arrow[r,"{\phi^{f,g}}"] \arrow[d,"{{R(\varepsilon \c \delta)_f}}"']\& {g\st W_f} \arrow[d,"{F(\varepsilon\c \delta)_{W_f}}"] \arrow[ld,"{(\alpha_{\varepsilon \c \delta})_f}",twoiso, shorten <= 2ex, shorten >= 2ex]\\
		{W_{\wt{f\c g''}}} \arrow[r,"{\phi^{f,g''}}"'] \& {{g''}\st W_{f}}
	\end{cdN}
	\h[3] = \h[3]
	\begin{cdN}
		{W_{\wt{f\c g}}} \arrow[dd,"{R(\varepsilon \c \delta)_f}"', bend right=60,""{name=A}]\arrow[r,"{\phi^{f,g}}"] \arrow[d,"{R(\delta)_f}"']\& {g\st W_f} \arrow[d,"{F(\delta)_{W_f}}"] \arrow[ld,"{(\alpha_{\delta})_f}",twoiso, shorten <= 2ex, shorten >= 2ex]\arrow[dd,"{F(\varepsilon\c \delta)_{W_f}}", bend left=60,""'{name=B}]\\
		{W_{\wt{f\c g'}}} \arrow[from=A,"{}",iso]\arrow[d,"{R(\varepsilon)_f}"'] \arrow[r,"{\phi^{f,g'}}"'] \& {{g'}\st W_{f}} \arrow[d,"{F(\varepsilon)_{W_f}}"] \arrow[ld,"{(\alpha_{\varepsilon})_f}",twoiso, shorten <= 2ex, shorten >= 2ex] \arrow[to=B,"{}",iso]\\
		{W_{\wt{f\c g''}}} \arrow[r,"{\phi^{f,g''}}"'] \& {{g''}\st W_{f}}
	\end{cdN}
\end{eqD*}

Given a chain of morphisms $M\ar{t} L \ar{h} E \ar{g} D \ar{f} C$ with $f\in S$, the following equality of 2-cells holds:
\begin{samepage}
\begin{eqD*}\scalebox{0.9}{
	\begin{cdN}
		{W_{\wt{\wt{\wt{f\c g}\c h}\c t}}} \arrow[rd,"{}",aiso] \arrow[dd,"{}",aiso,""{name=C,pos=0.8}] \arrow[r,"\phi^{\wt{\wt{f\c g}\c h},t}"] \& {t\st W_{\wt{\wt{f\c g}\c h}}}\arrow[rr,"{t\st \phi^{\wt{f\c g},h}}"]\& {} \& {t\st h\st W_{\wt{f\c g}}} \arrow[d,"{}",aeq,""{name=A}] \arrow[rr,"{t\st h\st \phi^{f,g}}"] \arrow[lld,"{\beta^{h,t}_{\wt{f\c g}}}"'{pos=0.4},Rightarrow ,shorten <=7.5ex, shorten >= 8.5ex]\& {} \& {t\st h\st g\st W_f} \arrow[d,"{}",aeq,""{name=B}] \arrow[from=A,to=B,"{}",iso]\\
		{} \& {W_{\wt{\wt{f\c g}\c h \c t}}} \arrow[d,"{}",aiso,""{name=D}] \arrow[from=C, to=D,iso,"{}",shift left=2ex]\arrow[rr,"{\phi^{\wt{f\c g},h \c t}}"] \& {} \& {(h\c t)\st W_{\wt{f\c g}}} \arrow[rr,"{(h\c t)\st \phi^{f,g}}"]\& {} \& {(h\c t)\st g \st W_f} \arrow[d,"{}",aeq] \arrow[lllld,"{\beta^{g, h\c t}_f}",Rightarrow ,shorten <=21.5ex, shorten >= 21.5ex] \\
		{W_{\wt{\wt{f\c g\c h}\c t}}} \arrow[r,"{}",aiso]\& {W_{\wt{f\c g\c h \c t}}} \arrow[rrrr,"{\phi^{f,g\c h\c t}}"']\& {} \& {} \& {} \& {(g\c h\c t)\st W_f} 
	\end{cdN}}
\end{eqD*}
\vspace{-6.5mm}
\begin{eqD}{cocycle4mor}
	\begin{cdN}
		{\hphantom{.}} \arrow[d,"{}",equal ,shorten <=2.5ex, shorten >= 2.5ex, shift left=3ex]\\
		{\hphantom{.}}
	\end{cdN}
\end{eqD}
\vspace{-6.5mm}
\begin{eqD*}\scalebox{0.9}{
	\begin{cdN}
		{W_{\wt{\wt{\wt{f\c g}\c h}\c t}}} \arrow[dd,"{}",aiso,""{name=A}] \arrow[r,"\phi^{\wt{\wt{f\c g}\c h},t}"] \& {t\st W_{\wt{\wt{f\c g}\c h}}} \arrow[rd,"{}",aiso,""{name=P,pos=0.95}]\arrow[rr,"{t\st \phi^{\wt{f\c g},h}}"]\&[-4ex] {} \&[-6ex] {t\st h\st W_{\wt{f\c g}}} \arrow[rr,"{t\st h\st \phi^{f,g}}"]\&[-6ex] {} \&[-2ex] {t\st h\st g\st W_f} \arrow[d,"{}",aeq] \arrow[ld,"{}",aeq] \arrow[to=P,"{h\st \beta^{g,h}_f}",Rightarrow,shorten <=18.5ex, shorten >= 18.5ex]\\
		{} \& {} \& {t\st W_{\wt{f\c g\c h}}} \arrow[from=A,"{}",iso] \arrow[rr,"{t\st \phi^{f,g\c h}}",""'{name=F}] \& {} \& {t\st((g\c h)\st W_f)} \arrow[rd,"{}",aeq] \arrow[r,"{}",iso]\& {(h\c t)\st g \st W_f} \arrow[d,"{}",aeq]  \\
		{W_{\wt{\wt{f\c g\c h}\c t}}} \arrow[r,"{}",aiso]\& {W_{\wt{f\c g\c h \c t}}} \arrow[ru,"{\phi^{\wt{f\c g\c h},t}}"] \arrow[rrrr,"{\phi^{f,g\c h\c t}}"',""'{name=G,pos=0.425}] \arrow[from=F,to=G,"{\beta^{g\c h,t}_f}",Rightarrow,shorten <=3.5ex, shorten >= 3.5ex]\& {} \& {} \& {} \& {(g\c h\c t)\st W_f} 
	\end{cdN}}
\end{eqD*}
\end{samepage}
Given $f\: D\to C$ in $R(D)$ and $g\: E\to D$ in $\K$, the following equalities of 2-cells hold:
\begin{eqD*}
	\begin{cdN}
		{W_{\wt{f\c g}}} \arrow[d,"{}",equal] \arrow[r,"{\phi^{\wt{f\c g},\id{E}}}"]\& {\id{E}\st W_{\wt{f\c g}}} \arrow[r,"{\id{E}\st \phi^{f,g}}"]\& {\id{E}\st g \st W_f} \arrow[d,"{}",aeq] \arrow[lld,"{\beta^{g,\id{E}}_f}",Rightarrow,shorten <=8.5ex, shorten >= 10.5ex]\\
		{W_{\wt{f\c g}}} \arrow[rr,"{\phi^{f,g}}"'] \& {} \& {g\st W_f}
	\end{cdN}
\end{eqD*}
\vspace{-8.5mm}
\begin{eqD}{idE}
	\begin{cdN}
		{\hphantom{.}} \arrow[d,"{}",equal ,shorten <=2.5ex, shorten >= 2.5ex]\\
		{\hphantom{.}}
	\end{cdN}
\end{eqD}
\vspace{-8.5mm}
\begin{eqD*}
	\begin{cdN}
		{W_{\wt{f\c g}}} \arrow[r,"{}",aeq, bend right=50,""'{name=B,pos=0.46}] \arrow[d,"{}",equal,""{name=C}] \arrow[r,"{\phi^{\wt{f\c g},\id{E}}}",""'{name=A}] \arrow[from=A, to=B,"{\rho_{\wt{f\c g}}}"{pos=0.3}, Rightarrow ,shorten <=0.2ex, shorten >= 1ex, shift left=-1ex]\& {\id{E}\st W_{\wt{f\c g}}} \arrow[r,"{\id{E}\st \phi^{f,g}}"]\& {\id{E}\st g \st W_f} \arrow[d,"{}",aeq,""{name=D}] \arrow[from=C, to=D,"{}",iso]\\
		{W_{\wt{f\c g}}} \arrow[rr,"{\phi^{f,g}}"'] \& {} \& {g\st W_f}
	\end{cdN}
\end{eqD*}

and 

\begin{eqD*}
	\begin{cdN}
		{W_{\wt{f\c g}}} \arrow[d,"{}",equal] \arrow[r,"{\phi^{f,g}}"]\& {g\st W_f} \arrow[r,"{g\st \phi^{f,\id{D}}}"]\& { g \st\id{D}\st  W_f} \arrow[d,"{}",aeq] \arrow[lld,"{\beta^{\id{D},g}_f}",Rightarrow,shorten <=8.5ex, shorten >= 10.5ex]\\
		{W_{\wt{f\c g}}} \arrow[rr,"{\phi^{f,g}}"'] \& {} \& {g\st W_f}
	\end{cdN}
\end{eqD*}
\vspace{-6.5mm}
\begin{eqD}{idD}
	\begin{cdN}
		{\hphantom{.}} \arrow[d,"{}",equal ,shorten <=2.5ex, shorten >= 2.5ex]\\
		{\hphantom{.}}
	\end{cdN}
\end{eqD}
\vspace{-6.5mm}
\begin{eqD*}
	\begin{cdN}
		{W_{\wt{f\c g}}}  \arrow[d,"{}",equal,""{name=C}] \arrow[r,"{\phi^{f,g}}"]\& {g\st W_f} \arrow[r,"{}",aeq,bend right=50,""'{name=B,pos=0.46}] \arrow[r,"{g\st \phi^{f,\id{D}}}",""'{name=A}] \arrow[from=A, to=B,"{g\st \rho_{f}}"{pos=0.3}, Rightarrow ,shorten <=0.2ex, shorten >= 1ex, shift left=-1ex]\& { g \st\id{D}\st  W_f} \arrow[d,"{}",aeq,""{name=D}] \arrow[from=C,to=D,"{}",iso] \\
		{W_{\wt{f\c g}}} \arrow[rr,"{\phi^{f,g}}"'] \& {} \& {g\st W_f}
	\end{cdN}
\end{eqD*}

\vspace{3.5mm}

	\noindent This weak descent datum is  \dfn{weakly effective} if there exists an object $W\in F(C)$, for every morphism $D\ar{f}C\in S$ an equivalence
	$$\psi^f\: W_f \aequi{} f\st W$$
	for every $D\ar{f}C$ in $R(D)$ and every morphism $g\: E\to D$ an isomorphic 2-cell
	\begin{eqD}{2cella}
	\begin{cdN}
		{W_{\wt{f\c g}}} \arrow[rr,"{\phi^{f,g}}"] \arrow[d,"{\psi^{\wt{f\c g}}}"']\& {} \& {g\st W_f} \arrow[d,"{g\st \psi^f}"] \arrow[lld,"{(\varepsilon_g)_f}",Rightarrow ,shorten <=10.5ex, shorten >= 12.5ex]\\
		{(\wt{f\c g})\st W} \arrow[r,"{}",aiso]\& {(f\c g)\st W}\arrow[r,"{}",aeq] \& {g\st f\st W} 
	\end{cdN}
\end{eqD}
and for every 2-cell $\gamma\: f\Rightarrow f'$ with $f,f'\: D\to C$ in $R(D)$ an isomorphic 2-cell
\begin{eqD*}
	\begin{cdN}
		{W_f} \arrow[r,"{\psi^f}"] \arrow[d,"{\eta_{\gamma}}"']\& {f\st W} \arrow[d,"{{F(\gamma) _W}}"] \arrow[ld,"{\psi_{\gamma}}",twoiso, shorten <= 2ex, shorten >= 2ex]\\
		{W_{f'}} \arrow[r,"{\psi^{f'}}"'] \& {{f'}\st W}
	\end{cdN}
\end{eqD*}

\noindent These data need to satisfy the following conditions.
Given $f\: D\to C$ in $R(D)$, the following equality of 2-cells holds:
\begin{eqD*}
	\csq*[l][7][7][\psi_{id_f}][2.5]{W_f}{f\st W}{W_f}{f\st W}{\psi^f}{\eta_{\id{f}}}{F(\id{f})_W}{\psi^f}
	\h[3]= \h[3]
	\begin{cdN}
		{W_f} \arrow[r,"{\psi^f}"] \arrow[d,"{\eta_{\id{f}}}"', ""{name=A},bend right=60]\arrow[d,"{}",equal,""'{name=B}] \arrow[from=A,to=B,"{}",iso]\& {f\st W} \arrow[d,"{}",equal,""{name=C}] \arrow[d,"{F(\id{f})_W}",""'{name=D}, bend left=60] \arrow[from=C, to=D,"{}", iso]\\
		{W_f} \arrow[r,"{\psi^f}"']\& {f\st W}
	\end{cdN}
\end{eqD*}

Given $f,f',f'':D\to C$ in $R(D)$ and 2-cells $\gamma\: f \Rightarrow f'$ and $\delta\: f' \Rightarrow f''$, the following equality of 2-cells holds:
\begin{eqD*}
	\csq*[l][7][7][\psi_{\delta \c \gamma}][2.5]{W_f}{f\st W}{W_{f''}}{{f''}\st W}{\psi^f}{\eta{\delta\c \gamma}}{F(\delta\c \gamma)_W} {\psi^{f''}}
	\h[3] = \h[3]
	\begin{cdN}
		{W_f}\arrow[dd,"{\eta_{\delta \c \gamma}}"', bend right=60,""{name=G}] \arrow[r,"{\psi^f}"] \arrow[d,"{\eta_\gamma}"']\& {f\st W} \arrow[d,"{F(\gamma)_W}"] \arrow[dd,"{F(\delta \c \gamma)_W}", bend left=60, ""'{name=F}] \arrow[ld,"{\psi_{\gamma}}", Rightarrow ,shorten <=3.5ex, shorten >= 3.5ex]\\
		{W_{f'}} \arrow[from=G,"{}",iso] \arrow[d,"{\eta_\delta}"']\arrow[r,"{\psi^{f'}}"'] \& {{f'}\st W} \arrow[d,"{F(\delta)_W}"] \arrow[to=F,"{}", iso] \arrow[ld,"{\psi_{\delta}}", Rightarrow ,shorten <=3.5ex, shorten >= 3.5ex]\\
		{W_{f''}} \arrow[r,"{\psi^{f''}}"'] \& {{f''}\st W}
	\end{cdN}
\end{eqD*}

Given morphisms $f\: D\to C$ in $R(D)$ and $g,g'\: E \to D$ in n$\K$ and a 2-cell $\delta\: g \Rightarrow g'$, the following equality of 2-cells holds:

\begin{eqD*}
	\begin{cdN}
		{} \& {} \&[-5ex] {g\st W_f} \arrow[d,"{(\varepsilon_g)_f}",Rightarrow ,shorten <=2.5ex, shorten >= 2.5ex]\arrow[rrd,"{g\st \psi^f}", bend left=25]  \&[-5ex] {} \& {} \\
		{W_{\wt{f\c g}}} \arrow[r,"{\psi^{\wt{f\c g}}}"] \arrow[d,"{\eta_{R(g)(\delta)}}"',""'{name=A}] \arrow[rru,"{\phi^{f,g}}", bend left=25] \& {(\wt{f\c g})\st W}  \arrow[d,"{F(R(g)(\delta))_W}"{description},""'{name=B}] \arrow[rr,"{}",aiso,]\& {} \& {(f\c g)\st W} \arrow[d,"{F(f\star \delta)_W}"{description},""{name=C}] \arrow[r,"{}",aeq] \& {g\st f\st W} \arrow[d,"{g\st (F(\delta)_W)}",""{name=D}] \arrow[from=A,to=B,"{}",iso] \arrow[from=C,to=D,"{}",iso] \arrow[from=B,to=C,"{}",iso]\\
		{W_{\wt{f\c g'}}} \arrow[r,"{\psi^{\wt{f\c g'}}}"'] \& {(\wt{f\c g'})\st W} \arrow[rr,"{}",aiso] \& {} \& {(f\c g')\st W} \arrow[r,"{}",aeq] \& {{g'}\st f\st W} 
	\end{cdN}
\end{eqD*}
\vspace{-5.5mm}
\begin{eqD}{epsilongg'}
	\begin{cdN}
		{\hphantom{.}} \arrow[d,"{}",equal ,shorten <=2.5ex, shorten >= 2.5ex]\\
		{\hphantom{.}}
	\end{cdN}
\end{eqD}
\vspace{-5.5mm}
\begin{eqD*}
	\begin{cdN}
		{W_{\wt{f\c g}}}  \arrow[d,"{\eta_{R(g)(\delta)}}"',""{name=A}]\arrow[rr,"{\phi^{f,g}}"] \& {} \&[-5ex] {g\st W_f} \arrow[d,"{F(\delta)_{W_f}}",""{name=B}] \arrow[rr,"{g\st \psi^f}"] \&[-5ex] {} \& {g\st f\st W} \arrow[d,"{g\st (F(\delta)_W)}",""{name=C}] \arrow[from=A,to=B,"{}",iso] \arrow[from=B,to=C,"{}",iso]\\
		{W_{\wt{f\c g'}}}  \arrow[rr,"{\psi^{f,g'}}"]\arrow[rd,"{\psi^{\wt{f\c g'}}}"',bend right=25] \& {} \& {{g'}\st W_f}  \arrow[d,"{(\varepsilon_{g'})_f}",Rightarrow ,shorten <=2.5ex, shorten >= 2.5ex]\arrow[rr,"{{g'}\st \psi^f}"] \& {} \& {{g'}\st f\st W} \\
		{} \& {(\wt{f\c g'})\st W} \arrow[rr,"{}",aiso]\& {} \& {(f\c g')\st W} \arrow[ru,"{}",aeq, bend right =25]\& {}
	\end{cdN}
\end{eqD*}

\pagebreak[10]

\vspace{1.5mm}

Given morphisms $D''\ar{h} D' \ar{g}  D\ar{f} C$
with $f\in S$, the following equality of 2-cells holds:
\begin{samepage}
	\begin{eqD*}\scalebox{0.8}{
			\begin{cdsN}{7}{3}
				{W_{\wt{\wt{f\c g}\c h}}} \arrow[rrr,"{\phi^{h,\wt{f\c g}}}",] \arrow[d,"{}",aiso] \arrow[rd,"{\psi^{\wt{\wt{f\c g}\c h}}}"]\& {} \& {} \arrow[ld,"{(\varepsilon_h)_{\wt{f\c g}}}", Rightarrow ,shorten <=2.5ex, shorten >= 2.5ex, shift left=2ex]  \& {h\st W_{\wt{f\c g}}}  \arrow[d,"{h\st \psi^{\wt{f\c g}}}"] \arrow[rr,"{h\st \phi^{f,g}}"] \& {}  \& {h\st g\st W_f} \arrow[dd,"{h\st g\st \psi^f}"] \arrow[lld,"{h\st (\varepsilon_g)_f}", Rightarrow ,shorten <=8.5ex, shorten >= 8.5ex, shift left=4ex]\\
				{W_{\wt{{f\c g}\c h}}} \arrow[r,"{}",iso , shift left=-2ex] \arrow[rdd,"{\psi^{\wt{f\c g\c h}}}"']\& {(\wt{\wt{f\c g}\c h})\st W} \arrow[dd,"{}", aiso]\arrow[rd,"{}", aiso] \& {} \& {h\st({f\c g})\st W} \arrow[rd,"{}",aiso,""{name=A}]  \& {} \& {}\\
				{} \& {} \& {(\wt{f\c g}\c h)\st W} \arrow[rr,"{}",iso, color=blue]\arrow[ru,"{}",aeq] \arrow[rd,"{}", aeq] \& {} \& {h\st (f\c g)\st W} \arrow[r,"{}",aeq]\& {h\st g\st f\st W} \arrow[d,"{}", aeq]\\
				{} \& {{(\wt{f\c g}\c h)\st W}} \arrow[rr,"{}",aiso] \& {} \& {(f\c g\c h)\st W} \arrow[ru,"{}", aeq] \arrow[rr,"{}", aeq]\& {} \& {(g\c h)\st f\st W}
		\end{cdsN}}
	\end{eqD*}
	\begin{eqD}{axiom1trimod}
		\vspace*{-1.8ex}\hspace*{5.7cm}
		\begin{cdsN}{3}{3}
			\hphantom{.}\arrow[d,equal]\\
			\hphantom{.}
		\end{cdsN}\hspace{5.4cm}
	\end{eqD}
	\begin{eqD*}\scalebox{0.8}{
			\begin{cdsN}{9}{7}
				{W_{\wt{\wt{f\c g}\c h}}} \arrow[rr,"{\phi^{h,\wt{f\c g}}}"] \arrow[d,"{}",aiso] \& {}   \& {h\st W_{\wt{f\c g}}}  \arrow[rr,"{h\st \phi^{f,g}}"] \& {}   \& {h\st g\st W_f} \arrow[d,"{h\st g\st \psi^f}"] \arrow[ld,"{}",aeq] \arrow[lllld,"{\beta^{f,g}_h}",Rightarrow ,shorten <=31.5ex, shorten >= 31.5ex] \\
				{W_{\wt{f\c g\c h}}} \arrow[rd,"{\psi^{\wt{f\c g\c h}}}"'] \arrow[rrr,"{\phi^{f,g\c h}}"] \& {} \& {} \& {(g\c h)\st W_f} \arrow[lld,"{(\varepsilon_{g\c h})_f}", Rightarrow ,shorten <=10.5ex, shorten >= 10.5ex] \arrow[rd,"{(g\c h)\st \psi^f}"'] \arrow[r,"{}",iso] \& {h\st g\st f\st W} \arrow[d,"{}", aeq] \\
				{} \& {(\wt{f\c g\c h})\st W} \arrow[r,"{}",aiso]  \& {({f\c g\c h})\st W} \arrow[rr,"{}", aeq] \& {} \& {(g\c h)\st f\st W} 
		\end{cdsN}}
	\end{eqD*}
\end{samepage}
\end{defne}
\vspace{6.5mm}
We now make some comments about the previous definition.

\begin{oss}[weakness]
	We call the kind of descent datum of Definition \ref{weakdescentdatum} weak descent datum for two main reasons. The first one is that the morphisms of type $\phi^{f,g}$ are equivalences and not isomorphisms. The second one is that the conditions that are expressed by equalities for a descent datum are expressed by fixed isomorphisms for a weak descent datum. 
	
	Notice that the condition of effectiveness for a weak descent datum is equally weak. In fact the morphisms of type $\psi^f$ are just equivalences and so the global object allows one to recover the local data only up to equivalence.
\end{oss}

\begin{oss}[weak cocycle condition]
	The isomorphism $\beta^{g,h}_f$ of Diagram (\ref{weakcocycle}) can be viewed as a weak version of the cocycle condition satisfied by the usual descent data. Indeed, up to taking into account the pseudonaturality of the bisieve, it has the same shape of the usual cocycle condition. 
	\end{oss}

We are now ready to prove our third characterization result.

\begin{prop}\label{weffonobjects}
	Let $C\in \K$ and let $S\: R \Rightarrow \HomC{\K}{-}{C}$ be a bisieve on it. The pseudofunctor
	$$(-\c S)\c \Gamma\: F(C) \longrightarrow \Tricat(\K\op, \Bicat)(R, F)$$
	is surjective on equivalence classes of objects if and only if every weak descent datum for $S$ of elements of $F$ is weakly effective.
\end{prop}
Let $X\in F(C)$. The image of $X$ under $(-\c S)\c \Gamma$ is $\sigma^X \c S\: R \Rightarrow F$, where $\sigma^X$ is given as in Remark \ref{assYonedatricat}. Let then $\alpha\: R \Rightarrow F$ be a tritransformation. This is a tritransformation of the kind described in Example \ref{ourtritrans}. Then $\alpha$ is given by a family of pseudofunctors $\alpha_D\: R(D) \to F(D)$ indexed by the objects of $\K$. For every $f\: D\to C$ in $R(D)$ we call $W_f$ the image of $f$ under $\alpha_D$. The pseudofunctoriality of $\alpha_D$ implies that, given $f\: D\to C$ in $R(D)$, there exists an isomorphic 2-cell 
\begin{cd}
	{W_f}  \arrow[r,"{}",equal, bend right=40,""{name=B}] \arrow[r,"{\alpha_D(\id{f})}", bend left =40,""'{name=A}]  \arrow[r,"{}",iso]\& {W_f}
\end{cd}
and that, given morphisms $f,f',f''\: D\to C$ in $R(D)$ and 2-cells $\gamma\: f \Rightarrow f'$ and $\delta\: f' \Rightarrow f''$, there exists an invertible 2-cell
\begin{cd}
	{W_f} \arrow[rr,"{}",iso,shift left=-4ex] \arrow[rr, bend right =40, "{\alpha_D({\delta \c \gamma})}"'] \arrow[r,"{\alpha_D({\gamma})}"] \& {W_{f'}} \arrow[r,"{\alpha_D({\delta})}"] \& {W_{f''}}
\end{cd}
Furthermore, given a morphism $g\: E\to D$ in $\K$, we have a pseudonatural transformation 
	\csq[l][7][7][\alpha_g]{R(D)}{R(E)}{F(D)}{F(E)}{R(g)}{\alpha_D}{\alpha_E}{F(g)}
whose component relative to every $f\: D\to C$ in $R(D)$ is an equivalence
$$(\alpha_g)_f\: W_{\wt{f\c g}} \aequi g\st W_f$$ 
This means that, given $f,f'\: D\to C$ in $R(D)$ and a 2-cell $\gamma\: f \Rightarrow f'$, there exists an isomorphic 2-cell
		\begin{cd}
		{W_{\wt{f\c g}}} \arrow[r,"{(\alpha_g)_f}"] \arrow[d,"{{\alpha_E(\gamma \star g)}}"']\& {g\st W_f} \arrow[d,"{g\st (\alpha_D(\gamma)}"] \arrow[ld,"{\rho^g_{\gamma}}",twoiso, shorten <= 2ex, shorten >= 2ex]\\
		{W_{\wt{f'\c g}}} \arrow[r,"{(\alpha_g)_{f'}}"'] \& {{g'}\st W_{f}}
	\end{cd}
and these 2-cells are assigned in a pseudofunctorial way.

Moreover, the $\alpha_{g}$'s  for $g\: E\to D$ form a pseudonatural transformation $\wt{\alpha}$. So given $g,g'\: E\to D$ and a 2-cell $\delta\: g\Rightarrow g'$, for every $f\: D\to C$ in $R(D)$ we have an invertible 2-cell
	\begin{cd}
	{W_{\wt{f\c g}}} \arrow[r,"{(\alpha_g)_f}"] \arrow[d,"{{R(\delta)_f}}"']\& {g\st W_f} \arrow[d,"{F(\delta)_{W_f}}"] \arrow[ld,"{(\alpha_{\delta})_f}",twoiso, shorten <= 2ex, shorten >= 2ex]\\
	{W_{\wt{f\c g'}}} \arrow[r,"{(\alpha_{g'})_f}"'] \& {{g'}\st W_{f}}
\end{cd}
and these 2-cells are assigned in a pseudofunctorial way.

Furthermore, given $L\ar{h} E \ar{g} D$ in $\K$, we have invertible modifications 
	\begin{eqD*}
	\scalebox{0.9}{
		\begin{cdN}
			{R(D)} \arrow[r,"{R(g)}"] \arrow[d,"{\alpha_D}"'] \&  {R(E)} \arrow[ld,"{\alpha_g}", Rightarrow ,shorten <=3.5ex, shorten >= 3.5ex] \arrow[r,"{R(h)}"]  \arrow[d,"{\alpha_E}"]\& {R(L)} \arrow[ld,"{\alpha_h}", Rightarrow ,shorten <=3.5ex, shorten >= 3.5ex] \arrow[d,"{\alpha_L}"] \\
			{F(D)} \arrow[r,"{F(g)}"]  \arrow[rr,"{F(g\c h)}"',""{name=D}, bend right = 55] \&  {F(E)}  \arrow[to=D,"{\chi_{g,h}}" ,shorten <=1.5ex, shorten >= 1.5ex, Rightarrow] \arrow[r,"{F(h)}"] \& {F(L)}
		\end{cdN}
		\h[-9]
		\begin{cdN}
			{\hphantom{.}}  \arrow[r,"{\beta^{g,h}}", triple]\& {\hphantom{.}}
		\end{cdN}
		\h[-9]
		\begin{cdN}
			{} \& {R(E)} \arrow[d,"{\chi_{g,h}}", Rightarrow ,shorten <=2.5ex, shorten >= 2.5ex] \arrow[rd,"{R(h)}", bend left= 30] \& {} \\
			{R(D)} \arrow[ru,"{R(g)}", bend left= 30] \arrow[rr,"{R(g\c h)}"',""{name=E}]  \arrow[d,"{\alpha_D}"']\& {\hphantom{.}} \& {R(L)} \arrow[lld,"{\alpha_{g\c h}}", Rightarrow ,shorten <=9.5ex, shorten >= 9.5ex]\arrow[d,"{\alpha_L}"] \\
			{F(D)} \arrow[rr,"{F(g\c h)}"']\& {} \& {F(L)} 
	\end{cdN}}
\end{eqD*}
and 
\begin{eqD*}
	\begin{cdN}
		{R(D)} \arrow[r,"{R(\id{D})}"', ""{name=K}] \arrow[r,"{}"{name=H}, equal, bend left= 60] \arrow[d,"{\alpha_D}"'] \arrow[from=H, to=K ,"{\iota_D}", Rightarrow ,shorten <=1ex]\& {R(D)} \arrow[ld,"{\alpha_{\id{D}}}", Rightarrow,shorten <=3.5ex, shorten >= 3.5ex] \arrow[d,"{\alpha_D}"] \\
		{F(D)} \arrow[r,"{F(\id{D})}"']\& {F(D)} 
	\end{cdN}
	\h[-3]
	\begin{cdN}
		{\hphantom{.}}  \arrow[r,"{\rho}", triple]\& {\hphantom{.}}
	\end{cdN}
	\h[-3]
	\begin{cdN}
		{R(D)} \arrow[r,equal]  \arrow[d,"{\alpha_D}"'] \& {R(D)} \arrow[ld,"", equal,shorten <=4.5ex, shorten >= 4.5ex] \arrow[d,"{\alpha_D}"] \\
		{F(D)} \arrow[r,"{F(\id{D})}"'{name=H}, bend right= 60] \arrow[r,equal,""{name=G}] \arrow[from=G, to=H,"{\iota_D}",Rightarrow ,shorten <=1.5ex, shorten >= 1.5ex]\& {F(D)} 
	\end{cdN}
\end{eqD*}
These modifications have components relative to $f\: D\to C$ 	
\begin{cd}
	{W_{\wt{\wt{f\c g}\c h}}} \arrow[d,"{}",aiso]\arrow[r,"{ (\alpha_h)_{\wt{f\c g}}}"] \& {g\st W_{\wt{f\c g}}} \arrow[r,"{h\st (\alpha_g)_f}"]\& {h\st g\st W_f} \arrow[d,"{}",aeq] \arrow[lld,"{{\beta}^{g,h}_f}", Rightarrow ,shorten <=9.5ex, shorten >= 9.5ex]\\
	{W_{\wt{f\c g\c h}}} \arrow[rr,"{(\alpha_{g\c h})_f}"'] \& {} \& {(g\c h)\st W_f}
\end{cd}
and 
\begin{cd}
	{W_f}  \arrow[r,"{}",aiso, bend right=40,""{name=B}] \arrow[r,"{(\alpha_{\id{D})_f}}", bend left =40,""'{name=A}]  \arrow[twoiso,from=A, to=B,"\rho_f", shorten >=1ex, shorten <=1ex]\& {\id{D}\st W_f} \\
\end{cd}

We observe that the data of $\alpha$ are exactly the same of those of a weak descent datum on $S$ of elements of $F$. The assignment on a morphism $f\: D\to C$ in $R(D)$ is the image $W_f$ of $f$ under the pseudofunctor $\alpha_D$ and the assignment on a 2-cell $\gamma\: f \Rightarrow f'$ is the morphism $\alpha_D(\gamma)$. Moreover, the equivalence $\phi^{f,g}$ coincides with the equivalence $(\alpha_g)_f$. And the axioms of tritransformation for $\alpha$ (see Example \ref{ourtritrans}) computed on $f:D\to C$ are exactly the conditions of diagrams (\ref{cocycle4mor}), (\ref{idE}) and (\ref{idD}).

It suffices, then, to prove that the tritransformation $\alpha$ is equivalent to a tritransformation of the form $((-\c S)\c \Gamma)(W)$ for some object $W\in F(C)$ if and only if the corresponding weak descent datum is weakly effective.

If $\alpha$ is equivalent to $((-\c S)\c \Gamma)(W)$ , this means that there exists a trimodification $\varepsilon\: \alpha \aM{} \sigma^X \c S$ whose component 
$$(\varepsilon_D)_f\: W_f \ar{}f\st W$$
 is an equivalence for every $D\in K$ and every $f\: D\to C$ in $R(D)$. This happens exactly when the weak descent datum that corresponds to $\alpha$ is weakly effective with equivalence $\psi^f$ equal to $(\varepsilon_D)_f$. Indeed, the isomorphic 2-cells and the coherence conditions required by the definition of weak effectiveness coincide with the structure 2-cells and the axioms required for $\varepsilon$ to be a trimodification.
 
We are now able to prove the theorem of characterization of the notion of 2-stack in terms of explicit conditions.

\begin{teor}\label{char2stacks}
	A trihomomorphism $F\: \K\op \to \Bicat$ is a 2-stack if and only if for every $C\in \K$ and every bisieve $S\in \tau(C)$ the following conditions are satisfied:
	\begin{itemize}
		\item [(O)] every weak descent datum for $S$ of elements of $F$ is weakly effective;
		\item [(M)] every descent datum for $S$ of morphisms of $F$ is effective;
		\item [(2C)] every matching family for $S$ of 2-cells of $F$ has a unique amalgamation.
	\end{itemize}
\end{teor}

\begin{proof}
	Straightforward using Proposition \ref{charYoneda}, Proposition \ref{sheaveson2cells}, Proposition \ref{effonmorphisms} and Proposition \ref{weffonobjects}.
\end{proof}

The use of Theorem \ref{char2stacks} brings substantial advantages in calculations when proving that a certain trihomomorphism is a 2-stack. In \cite{Prin2bunquo2stacks}, we will use this result to prove that,when the underlying 2-category is nice enough, our higher dimensional analogues of the quotient stacks are 2-stacks.
 
\bibliographystyle{abbrv} 
\bibliography{Bibliography_2stacks.bib}
\end{document}